\documentclass[12pt]{amsart}

\usepackage{amscd,epsfig,amssymb}
\usepackage[all]{xy}

\def\comment #1{{\sf [#1]}}

\def\P{{\mathbb P}}
\def\Z{{\mathbb Z}}
\def\Q{{\mathbb Q}}
\def\R{{\mathbb R}}
\def\C{{\mathbb C}}
\def\N{{\mathbb N}}
\def\A{{\mathbb A}}
\def\D{{\mathbb D}}
\def\F{{\mathbb F}}

\def\V{{\mathbb V}}

\def\cC{{\mathcal C}}
\def\E{{\mathcal E}}

\def\L{{\mathcal L}}
\def\M{{\mathcal M}}
\def\O{{\mathcal O}}

\def\cV{{\mathcal V}}
\def\X{{\mathcal X}}

\def\h{{\mathfrak h}}

\def\w{\omega}
\def\G{\Gamma}

\def\ftilde{\tilde{f}}
\def\jtilde{\tilde{\jmath}}
\def\stilde{\tilde{s}}

\def\Ebar{\overline{\E}}

\def\Lbar{\overline{\L}}
\def\Mbar{\overline{\M}}

\def\That{\widehat{T}}

\def\Ttilde{\widetilde{T}}
\def\Xtilde{\widetilde{X}}

\def\Gtilde{\widetilde{\G}}
\def\Phitilde{\widetilde{\Phi}}

\def\aa{\mathbf{a}}
\def\bb{\mathbf{b}}

\def\ibar{\overline{\imath}}
\def\wbar{\overline{\w}}

\def\v{\vec{v}}

\def\bmu{{\pmb{\mu}}}

\def\Gm{{\mathbb{G}_m}}

\def\SL{\mathrm{SL}}
\def\GL{\mathrm{GL}}
\def\SO{\mathrm{SO}}
\def\PSL{\mathrm{PSL}}

\def\SLtilde{\widetilde{\SL}}

\def\gl{\mathfrak{gl}}

\def\red{\mathrm{red}}

\def\dot{\bullet}
\def\bs{{\backslash}}
\def\bbs{{\bs\negthickspace \bs}}	

\def\disjt{\overset{.}{\cup}}

\def\triv{\pmb{1}}
\def\Pminus{{\P^1 - \{0,1,\infty\}}}

\def\coarseM{\overline{M}}

\def\PZ{{\text{\small$\begin{pmatrix}1 & \Z \cr 0 & 1\end{pmatrix}$}}}

\newcommand\id{\operatorname{id}}

\newcommand\Aut{\operatorname{Aut}}

\newcommand\Hom{\operatorname{Hom}}

\newcommand\im{\operatorname{im}}
\newcommand\ord{\operatorname{ord}}
\newcommand\rank{\operatorname{rank}}
\newcommand\tr{\operatorname{tr}}
\newcommand\Iso{\operatorname{Iso}}
\newcommand\Div{\operatorname{Div}}
\newcommand\Pic{\operatorname{Pic}}
\newcommand\Char{\operatorname{Char}}
\newcommand\Spec{\operatorname{Spec}}
\newcommand\Cl{\operatorname{\cC\ell}}
\newcommand\Tr{\operatorname{Tr}}

\renewcommand\div{\operatorname{div}}
\renewcommand\Im{\operatorname{Im}}
\renewcommand\Re{\operatorname{Re}}

\newtheorem{theorem}{Theorem}           [section]
\newtheorem{lemma}[theorem]{Lemma}
\newtheorem{proposition}[theorem]{Proposition}
\newtheorem{corollary}[theorem]{Corollary}

\theoremstyle{definition}
\newtheorem{definition}[theorem]{Definition}
\newtheorem{example}[theorem]{Example}

\theoremstyle{remark}
\newtheorem{remark}[theorem]{Remark}

\newtheorem{exercise}{Exercise}

\begin{document}

\title{Lectures on Moduli Spaces of Elliptic Curves}
\author{Richard Hain}
\address{Department of Mathematics\\ Duke University\\
Durham, NC 27708-0320}
\email{hain@math.duke.edu}

\date{\today}

\thanks{Supported in part by grant DMS-0706955 from the National Science
Foundation.}



\begin{abstract}
The goal of these notes is to introduce and motivate basic concepts and
constructions (such as orbifolds and stacks) important in the study of moduli
spaces of curves and abelian varieties through the example of elliptic curves.
Moduli spaces of elliptic curves are rich enough so that one encounters most of
the important issues associated with moduli spaces, yet simple enough that most
of the constructions are elementary and explicit. These notes touch on all four
aspects of the study of moduli spaces of curves -- complex analytic,
topological, algebro-geometric, and number theoretic.
\end{abstract}

\maketitle

\tableofcontents

These informal notes are an expanded version of lectures on the moduli space of
elliptic curves given at Zhejiang University in July, 2008.\footnote{{\em
Geometry of Teichm\"uller spaces and moduli spaces of curves}, Zhejiang
University, July 14--20, 2008.} Their goal is to introduce and motivate  basic
concepts and constructions important in the study of moduli spaces of curves and
abelian varieties through the example of elliptic curves. The advantage of
working with elliptic curves is that most constructions are elementary and
explicit. All four approaches to moduli spaces of curves --- complex analytic,
topological, algebro-geometric, and number theoretic --- are considered. Topics
covered reflect my own biases. Very little, if anything, in these notes is
original, except perhaps the selection of topics and the point of view. 

Many moduli spaces are usefully regarded as orbifolds or stacks. The notes
include a detailed exposition of orbifolds, which is motivated by a discussion
of how the quotient of the upper half plane by the modular group $\SL_2(\Z)$ is
related to families of elliptic curves. The moduli space of elliptic curves
$\M_{1,1}$ and its Deligne-Mumford compactification $\Mbar_{1,1}$ are
constructed as orbifolds. Modular forms arise naturally as holomorphic sections
of powers of the Hodge bundle over the orbifold $\Mbar_{1,1}$. They, in turn,
are used to construct the extension of the universal elliptic curve
$\E\to\M_{1,1}$ to the universal stable elliptic curve $\Ebar \to \Mbar_{1,1}$.
The homotopy types and Picard groups of the orbifolds $\M_{1,1}$ and
$\Mbar_{1,1}$ are computed explicitly. The discussion of orbifolds is used to
motivate the definition of stacks, which is discussed very briefly in
Appendix~\ref{app:stacks}.

\bigskip

\noindent {\bf Note to the reader:} These notes are intended for students. The
exposition is generally elementary, but some sections, especially those later in
the notes, are more demanding.
\begin{enumerate}

\item The best way to learn about moduli spaces and orbifolds (and stacks) is to
work with them. For this reason, these notes contain over 100 exercises. The
reader is encouraged to work as many of them as possible.

\item Sections not central to the exposition are marked with an asterisk $\ast$.
These can be skipped.

\item Some basic background material on Riemann surfaces is reviewed in
Appendix~\ref{app:RS}.

\end{enumerate}

\bigskip

\noindent{\bf Background:} These notes assume familiarity with the definition of
and basic facts about Riemann surfaces, including the definition of holomorphic
and meromorphic functions and 1-forms, and of holomorphic line bundles. They
also assume familiarity with the basic concepts of algebraic topology, including
homology, fundamental groups and covering spaces. Some familiarity with sheaves
is desirable, but not essential. Good basic references for Riemann surfaces
include Forster's book \cite{forster} and Griffiths' China notes
\cite{griffiths}; Clemens' book \cite{clemens} is an excellent supplement.

\bigskip

\noindent{\bf Acknowledgments:} I am very grateful to Dan Edidin, Shahed Sharif
and the referee for numerous comments and corrections.

\section{Introduction to Elliptic Curves and the Moduli Problem}

A Moduli space of Riemann surfaces is a space whose points correspond to all
isomorphism classes of Riemann surface structures on a fixed compact oriented
surface. They themselves are algebraic varieties (or {\em orbifolds}). In this
section, we will construct the moduli space of elliptic curves, which is itself
a Riemann surface.

Before attempting to understand the moduli space of a structure such as a
Riemann surface, it is desirable to first understand the basic properties of the
structure itself. As we shall see in the case of elliptic curves, properties of
the object are reflected in properties of the moduli space. We therefore begin
with some basic facts from the theory of elliptic curves.

An elliptic curve is a ``1-pointed'' genus 1 curve:

\begin{definition}
An {\em elliptic curve} is a compact Riemann surface $X$ of genus 1 together
with the choice of a point $P\in X$.
\end{definition}

Since the genus of a compact Riemann surface is, by definition, the dimension of
its space of holomorphic 1-forms, the space of holomorphic 1-forms of a genus 1
Riemann surface has dimension 1.

\begin{exercise}
Use the Riemann-Roch formula (Appendix~\ref{sec:rr}) to prove that if $w$ is a
non-zero holomorphic 1-form on an elliptic curve $X$, then $w$ has no zeros.
Deduce that the canonical divisor $K_X$ of every elliptic curve is zero.
\end{exercise}

\begin{definition}
A subgroup $\Lambda$  of a finite dimensional real vector space $V$ is a
{\em lattice} if it is discrete and if $V/\Lambda$ is compact.
\end{definition}

\begin{exercise}
Show that a subgroup $\Lambda$ of the finite dimensional real vector space $V$
is a lattice if and only if $\Lambda$ is a finitely generated abelian group such
that every $\Z$-basis $\lambda_1,\dots,\lambda_n$ of $\Lambda$ is an $\R$-basis
of $V$. Deduce that if $\Lambda$ is a lattice in $V$, then $V/\Lambda$ is
diffeomorphic to the real $n$-torus $\R^n/\Z^n$.
\end{exercise}

The simplest examples of elliptic curves are 1-dimensional complex tori
$$
(X,P) = (\C/\Lambda,0)
$$
which are quotients of $\C$ by a lattice $\Lambda$. It is easy to write down
the holomorphic differentials on a complex torus:
$$
H^0(X,\Omega^1_X) = \C\,dz
$$
where $z$ is the coordinate in $\C$.

We shall show shortly that every elliptic curve is isomorphic to a 1-dimensional
complex torus. Before we do this, we need to introduce {\em periods}.

Suppose that $(X,P)$ is an elliptic curve. Fix a holomorphic 1-form $\w$ on
$X$. Define
$$
\Lambda = \big\{\int_\gamma \w : \gamma \in H_1(X;\Z)\big\}
$$
This is a group, elements of which are called the periods of $\w$.

\begin{lemma}
The group $\Lambda$ is a lattice in $\C$.
\end{lemma}

\begin{proof}
Choose a basis $\sigma_1,\sigma_2$ of $H_1(X;\Z)$. Set
$$
\lambda_j = \int_{\sigma_j}\w,\quad j=1,2.
$$
To prove that $\Lambda$ is a lattice, we have to show that $\lambda_1$ and
$\lambda_2$ are linearly independent over $\R$. If $\lambda_1 = a\lambda_2$ for
some $a\in \R$, then
$$
\int_{a\sigma_1 - \sigma_2}\overline\w = \int_{a\sigma_1 - \sigma_2}\w = 0,
$$
which implies that $\int \w$ and $\int \overline \w$ are linearly dependent over
$\C$ as functions $H_1(X;\Z) \to \C$. This implies that they represent
proportional elements of $H^1(C;\C)$ and therefore that
$$
\int_C \w \wedge \overline\w = 0.
$$

On the other hand, for each local holomorphic coordinate $w=u+iv$ on $C$, we can
write (locally) $\w = h(w)dw$. Consequently
$$
i\, \w \wedge \overline\w = 2 |h(w)|^2 du \wedge dv > 0
$$
from which it follows that
$$
i\int_C \w \wedge \overline\w \ > 0.
$$
It follows that $\lambda_1$ and $\lambda_2$ are linearly independent over $\R$
and that $\Lambda$ is indeed a lattice in $\C$.
\end{proof}

\begin{proposition}
\label{prop:torus}
Every elliptic curve is isomorphic to a 1-dimensional complex torus.
\end{proposition}

\begin{proof}
Let $(X,P)$ be an elliptic curve.
Choose a non-zero holomorphic differential $\w$ on $X$. We will show that
$(X,P)$ is isomorphic to $(\C/\Lambda,0)$ where $\Lambda$ is the period lattice
of $\w$. Define a holomorphic mapping
$$
F : X \to \C/\Lambda
$$
by
$$
F(x) = \int_P^x \w \mod \Lambda
$$
Here the integral is over any path in $X$ from $P$ to $x$. Since any two such
paths differ by an element of $H_1(X;\Z)$, the function $F$ is well defined.

By elementary calculus, the derivative of $F$ is $\w$. Since this is
holomorphic, this implies that $F$ is holomorphic. Further, since $\w$ has no
zeros, $F$ is a local biholomorphism at each point of $X$. By
Exercise~\ref{ex:order}, this implies that $F$ is a covering map. To complete
the proof, we show that $F$ has degree 1. To do this, it suffices to show that
the induced mapping
$$
F_\ast : H_1(X;\Z) \to H_1(\C/\Lambda;\Z),
$$
which is injective by covering space theory, is surjective. But this follows
as there is a natural isomorphism $H_1(\C/\Lambda;\Z) \cong \Lambda$ and
as, under this identification,
$$
F_\ast(\gamma) = \int_\gamma \w.
$$
\end{proof}

\begin{remark}
This also follows directly from the Uniformization Theorem, which implies that
the universal covering of $X$ is biholomorphic to $\C$. 
\end{remark}

Since every elliptic curve is isomorphic to a complex torus, and since every
complex torus is a group, we obtain:

\begin{corollary}
Every elliptic curve $(X,P)$ has the structure of a group with identity $P$
and where the multiplication $X\times X \to X$ and inverse $X \to X$ are 
holomorphic.
\end{corollary}

Shortly we will show that this group structure is unique.

\begin{corollary}
If $X$ is a compact Riemann surface of genus 1 and if $P,Q\in X$, then
the elliptic curves $(X,P)$ and $(X,Q)$ are isomorphic.
\end{corollary}

\begin{proof}
It suffices to prove this when $X$ is a complex torus $ \C/\Lambda$. In this
case the isomorphism is given by translation by $Q-P$:
$$
(\C/\Lambda,P) \to (\C/\Lambda,Q),\quad x\mapsto x-P+Q.
$$
\end{proof}

\begin{remark}
It is easier to construct moduli spaces of structures that have at most a finite
number of automorphisms. Since every genus 1 Riemann surface $X$ is isomorphic
to $\C/\Lambda$, its automorphism group $\Aut X$ contains $X$ as a group of
translations. For this reason, moduli problem for genus 1 curves is not well
behaved. We will see shortly that the automorphism group $\Aut(X,P)$ of an
elliptic curve is finite, which is one reason why we study the moduli problem
for elliptic curves rather than for genus 1 curves. In general, the automorphism
group of an $n$-pointed compact Riemann surface $(X,\{x_1,\dots,x_n\})$ is
finite if and only if $2g-2+n > 0$. This condition may seem mysterious, but it
is equivalent to the condition that the Euler characteristic of the punctured
surface $X':=X-\{x_1,\dots,x_n\}$ be negative. This, in turn (by the
Uniformization Theorem) is equivalent to the condition that $X'$ has a complete
hyperbolic metric.
\end{remark}

\begin{lemma}
\label{lem:isom}
Suppose that $\Lambda_1$ and $\Lambda_2$ are two lattices in $\C$. If
$f : (\C/\Lambda_1,0) \to (\C/\Lambda_2,0)$ is a holomorphic mapping, then
there exists $c\in \C$ such that $c\Lambda_1 \subseteq \Lambda_2$ and
$$
f\big(z + \Lambda_1\big) = cz + \Lambda_2.
$$
In particular, $f$ is a group homomorphism.
\end{lemma}

\begin{proof}
Note that $(\C,0) \to (\C/\Lambda,0)$ is a pointed universal covering of
$\C/\Lambda$. Covering space theory implies that there is a holomorphic map $F :
(\C,0) \to (\C,0)$ such that
$$
f\big(z + \Lambda_1\big) = F(z) + \Lambda_2.
$$
The result will follow if we show that $F$ is linear. For $j=1,2$ set
$$
\w_j = dz \in H^0(\C/\Lambda_j,\Omega^1_{\C/\Lambda_j}).
$$
Then there is a constant $c\in \C$ such that $f^\ast \w_2 = c\w_1$. Consequently
$dF = cdz$. Since $F(0) =0$, this implies that $F(z) = cz$.
\end{proof}

This yields the following fact, which will give us the leverage we need to
construct and understand the moduli space of elliptic curves.

\begin{corollary}
Two complex tori $(\C/\Lambda_1,0)$ and $(\C/\Lambda_2,0)$ are isomorphic if
and only if there is $c\in \C^\ast$ such that $\Lambda_2 = c\Lambda_1$.
\end{corollary}

\begin{exercise}
Show that
$$
\Aut(\C/\Lambda,0) = \{u \in \C^\ast : u\Lambda = \Lambda\}.
$$
Note that $\pm 1$ are automorphisms. Show that every $u \in \Aut(\C/\Lambda,0)$
has modulus 1. Deduce that $\Aut(\C/\Lambda,0)$ is isomorphic to the group
$\bmu_{2n}$ of $2n$th roots of unity for some $n\ge 1$. Deduce that the
automorphism group of every elliptic curve $(X,P)$ is a finite cyclic group of
even order.
\end{exercise}

The proposition also yields the following useful fact, which follows as every
elliptic curve is isomorphic to a 1-dimensional torus.

\begin{corollary}
Every holomorphic mapping $f : (X,P) \to (Y,Q)$ between elliptic curves is
a group homomorphism.
\end{corollary}

\begin{remark}
Another consequence of Proposition~\ref{prop:torus} is the well known statement
that every compact Riemann surface of genus 1 has a flat riemannian metric whose
conformal class is determined by the complex structure. The metric is unique up
to multiplication by a constant. Lemma~\ref{lem:isom} implies that holomorphic
maps between genus 1 Riemann surfaces are orientation preserving homotheties
with respect to their flat metrics. In higher genus, a similar statement holds
with flat replaced by hyperbolic. The main difference with the genus 1 case is
that there is a unique hyperbolic metric in each conformal class.
\end{remark}

\subsection{Moduli of elliptic curves}

To determine the moduli space of elliptic curves, we need only determine the
moduli space of lattices in $\C$.

As is typical in constructing the moduli space of curves in higher genus via
Teichm\"uller theory, and when constructing the moduli of principally polarized
abelian varieties, we begin by ``framing'' the object of interest.

\begin{definition}
A {\em framed elliptic curve} is an elliptic curve $(X,P)$ together with an
ordered basis $\aa,\bb$ of $H_1(X,\Z)$ with the property that the intersection
number $\aa\cdot\bb$ is 1.
\end{definition}

If $\Lambda$ is a lattice in $\C$ then $\lambda,\lambda' \in \Lambda$ are
linearly independent over $\R$ if and only if $\Im(\lambda'/\lambda)\neq 0$.
The condition that the corresponding elements of $H_1(\C/\Lambda)$ intersect
positively is that $\Im(\lambda'/\lambda)>0$.

\begin{definition}
A framing of a lattice $\Lambda$ in $\C$ is an ordered basis
$\lambda_1,\lambda_2$ such that $\lambda_2/\lambda_1$ has positive imaginary
part.
\end{definition}

Since $H_1(\C/\Lambda;\Z)$ is naturally isomorphic to $\Lambda$, a framing of
$(\C/\Lambda,0)$ corresponds to a framing of $\Lambda$.

Isomorphism of framed elliptic curves is defined in the obvious way. Two framed
lattices $(\Lambda;\lambda_1,\lambda_2)$ and $(\Lambda';\lambda_1',\lambda_2')$
are isomorphic if there is a non-zero complex number $u$ such that $\lambda_j' =
u\lambda_j$.

Clearly a framed lattice $(\Lambda;\lambda_1,\lambda_2)$ is determined by
its framing $\lambda_1,\lambda_2$ as
$$
\Lambda = \Z\lambda_1 \oplus \Z\lambda_2.
$$

\begin{exercise}
Show that the framed lattice with basis $\lambda_1,\lambda_2$ is isomorphic to
the framed lattice with basis $\w_1,\w_2$ if and only if $\lambda_2/\lambda_1 =
\w_2/\w_1$.
\end{exercise}

An immediate consequence is that every framed lattice is isomorphic to a unique
framed lattice of the form
$$
(\Z \oplus \Z\tau; 1,\tau)
$$
where $\tau$ lies in the upper half plane $\h$.

In summary:

\begin{proposition}
There are natural bijections
$$
\h \leftrightarrow
\left\{
\parbox{1.45in}{isomorphism classes of framed lattices}
\right\}
\leftrightarrow
\left\{
\parbox{1.6in}{isomorphism classes of framed elliptic curves}
\right\}.
$$
\end{proposition}

Under this correspondence, $\tau\in \h$ corresponds to the framed elliptic curve
$(\C/\Lambda_\tau;1,\tau)$ and the framed elliptic curve $(X,P;\aa,\bb)$
corresponds to
$$
\textstyle{\int_\bb\w/\int_\aa\w} \in \h,
$$
where $\w$ is a non-zero holomorphic differential on $X$ and
$$
\Lambda_\tau := \Z \oplus \Z\tau.
$$

At this stage, these correspondences are simply bijections of sets. In the next
lecture, we will show that the right-hand set has a natural Riemann surface
structure and that the bijection with $\h$ is a biholomorphism.

As a set, the moduli space of elliptic curves is the set of isomorphism classes
of elliptic curves. It is the quotient of the set of isomorphism classes of
framed elliptic curves that is obtained by forgetting the framing.

Two framings $(\aa,\bb)$ and $(\aa',\bb')$ of an elliptic curve are related by
\begin{equation}
\label{eqn:frames}
\begin{pmatrix}\bb' \cr \aa' \end{pmatrix} =
\begin{pmatrix}a & b \cr c & d\end{pmatrix}
\begin{pmatrix} \bb \cr \aa \end{pmatrix}.
\end{equation}
where
\begin{equation}
\label{eqn:gamma}
\gamma = \begin{pmatrix}a & b \cr c & d\end{pmatrix}
\end{equation}
is a $2\times 2$ integral matrix. Since $\aa\cdot \bb = \aa'\cdot\bb' = 1$,
$\gamma$ has determinant 1 and is thus an element of $\SL_2(\Z)$.

Denote the isomorphism class of the framed elliptic curve $(X,P;\aa,\bb)$ by
$[X,P;\aa,\bb]$. Define a left action of $\SL_2(\Z)$ on 
$$
\left\{
\parbox{1.6in}{isomorphism classes of framed elliptic curves}
\right\}
$$
by
$$
\begin{pmatrix}a & b \cr c & d\end{pmatrix} : [X,P;\aa,\bb] \mapsto
[X,P;\aa',\bb']
$$
where $\aa',\bb'$ are given by (\ref{eqn:frames}). The quotient is the set
of isomorphism classes of elliptic curves.

\begin{proposition}
The set of isomorphism classes of elliptic curves is the quotient
$$
\SL_2(\Z)\bs
\left\{
\parbox{1.6in}{isomorphism classes of framed elliptic curves}
\right\}.
$$
\end{proposition}

Let's compute the corresponding action of $\SL_2(\Z)$ on $\h$:
$$
\xymatrix{
\tau \ar@{<->}[r]\ar@{..>}[d]_(0.35)\gamma & (\C/\Lambda_\tau;1,\tau)
\ar[d]^\gamma\cr
\displaystyle{\frac{a\tau + b}{c\tau + d}}\ar@{<->}[r] &
(\C/\Lambda_\tau;c\tau + d,a\tau+b)
}
$$

To summarize:

\begin{proposition}
The set of isomorphism classes of elliptic curves is isomorphic to the quotient
$\SL_2(\Z)\bs \h$ of the upper half plane by the $\SL_2(\Z)$-action
$$
\begin{pmatrix}a & b \cr c & d \end{pmatrix}
: \tau \mapsto \frac{a\tau + b}{c\tau + d}.
$$
\end{proposition}

\begin{exercise}
Show that the isotropy group
$$
\{\gamma \in \SL_2(\Z) : \gamma \tau = \tau\}
$$
of $\tau \in \h$ is isomorphic to $\Aut(\C/\Lambda_\tau,0)$.
\end{exercise}

The problem of finding a fundamental domain for the $\SL_2(\Z)$-action on $\h$
can be solved easily by thinking of the upper half plane as the moduli space of
framed lattices in $\C$. We seek a natural basis (possibly up to finite
ambiguity) of every lattice $\Lambda$ in $\C$. A natural choice for the first
basis vector is a shortest vector $u \in \Lambda$. Since $\Lambda$ is a discrete
subset of $\C$, there is a finite number (generically 1) of these.

\begin{exercise}
Show that if $v \in \Lambda$ is a shortest vector that is not a real
multiple of $u$, then $u,v$ is a basis of $\Lambda$.
\end{exercise}

By replacing $v$ by $-v$ if necessary, we may assume that $\Im v/u > 0$.
The framed lattice $(\Lambda;u,v)$ is isomorphic to $(u^{-1}\Lambda,1,\tau)$,
where $\tau = v/u$, which we assume to be in $\h$.

\begin{exercise}
Show that, with these choices, $|\tau|\ge 1$ and that $|\Re \tau|\le 1/2$.
\end{exercise}

With a little more work (cf.\ \cite[p.~78]{serre}), we have:

\begin{proposition}
Every framed lattice in $\C$ is isomorphic to
one with basis $1,\tau$, where $\tau \in \h$ lies in the region
$$
F :=  \{\tau \in \h : |\Re(\tau)| \le 1/2 \text{ and } |\tau|\ge 1\}.
$$
If $\tau,\tau' \in F$ lie in the same $\SL_2(\Z)$ orbit, then either
$$
|\Re \tau| = |\Re \tau'| = 1/2 \text{ and } \tau' = \tau \pm 1,
$$
or
$$
|\tau| = 1 \text{ and } \tau' = -1/\tau.
$$
If $\gamma\tau = \tau$, then either $\gamma = \pm\id$ or
$$
\tau = i \text{ and } \gamma \text{ is a power of }
\begin{pmatrix} 0 & -1 \cr 1 & 0 \end{pmatrix},
$$
$$
\tau = \rho := e^{2\pi i/3} \text{ and } \gamma \text{ is a power of }
\begin{pmatrix} 0 & -1 \cr 1 & 1 \end{pmatrix},
$$
or
$$
\tau = -1/\rho \text{ and } \gamma \text{ is a power of }
\begin{pmatrix} 1 & -1 \cr 1 & 0 \end{pmatrix}.
$$
\end{proposition}

It is convenient (and standard) to set
$$
S = \begin{pmatrix} 0 & -1 \cr 1 & 0 \end{pmatrix},\
T = \begin{pmatrix} 1 & 1 \cr 0 & 1 \end{pmatrix},
\text{ and }
U = ST = \begin{pmatrix} 0 & -1 \cr 1 & 1 \end{pmatrix},
$$
Then $S$ has order 4, $U$ has order 6 and $S^2=U^3 = -I$. The stabilizer of $i$
is generated by $S$, the stabilizer of $\rho$ is generated by $U$. Serre
\cite[p.~78]{serre} proves that $\SL_2(\Z)$ is generated by $S$ and $T$ and has
presentation:\footnote{This is easily proved. Let $\G$ be the subgroup of
$\SL_2(\Z)$ generated by $S$ and $T$. Show that $F$ is a fundamental domain for
the action of $\G$ on $\h$. This is essentially the LLL algorithm.}
\begin{equation}
\label{eqn:presentation}
\SL_2(\Z) = \langle S,T : S^2 = (ST)^3, S^4\rangle
=\langle S, U : S^2 = U^3, S^4\rangle.
\end{equation}

\begin{exercise}
\label{ex:autos}
Show that if $\tau \in \h$, then
$$
\Aut(\C/\Lambda_\tau,0) \cong \{\gamma\in \SL_2(\Z) : \gamma(\tau) = \tau\}.
$$
Deduce that
$$
\Aut (\C/\Lambda_\tau) \cong \{\pm 1\}
$$
unless $\gamma$ lies in the $\SL_2(\Z)$-orbit of $i$ or $\rho$. Show that
$$
\Aut(\C/\Lambda_i,0) \cong \bmu_4 \text{ and }
\Aut(\C/\Lambda_\rho,0) \cong \bmu_6, 
$$
where $\bmu_n$ denotes the group of $n$th roots of unity.
\end{exercise}

\begin{corollary}
$\SL_2(\Z)\bs\h$ is homeomorphic to the disk.
\end{corollary}

By Exercise~\ref{ex:quotient}, we have:

\begin{theorem}
The quotient $\SL_2(\Z)\bs \h$ has a unique Riemann surface structure such that
the quotient mapping $\h \to \SL_2(\Z)\bs \h$ is holomorphic.
\end{theorem}

This is our first, but not final, version of the moduli space of elliptic
curves.

\section{Families of Elliptic Curves and the Universal Curve}

In the first lecture, we showed that the quotient $\SL_2(\Z)\bs \h$ of the upper
half plane $\h$ by the standard action of $\SL_2(\Z)$ is a Riemann surface
whose points correspond to the isomorphism classes of elliptic curves. Denote
this quotient by $M_{1,1}$.\footnote{In general, when $2g-2+n>0$, $M_{g,n}$
will denote the moduli space of compact Riemann surfaces of genus $g$ with $n$
marked points.}

A (holomorphic) family of elliptic curves over a complex manifold $T$ is a
complex manifold $X$ together with a holomorphic mapping $\pi : X \to T$ of
maximal rank and a section $s : T \to X$ of $\pi$ such that for each $t\in T$,
each fiber $(\pi^{-1}(t),s(t))$ is an elliptic curve.
$$
\xymatrix{
X \ar[r]_\pi & T\ar@/_/[l]_s
}
$$
For convenience, we denote the fiber $\pi^{-1}(t)$ of $\pi$ over $t\in T$ by
$X_t$.

To such a family we can associate the function $\Phi : T \to M_{1,1}$
defined by
$$
\Phi : t \mapsto [X_t,s(t)].
$$

We would like such a family of elliptic curves to be ``classified'' by mappings
$T \to M_{1,1}$. More precisely, we would like $M_{1,1}$ to satisfy:
\begin{enumerate}

\item the function $\Phi : T \to M_{1,1}$ is holomorphic;

\item every holomorphic mapping from a complex manifold $T$ to $M_{1,1}$
corresponds to a family of elliptic curves over $T$;

\item there should be a holomorphic family of elliptic curves $E \to M_{1,1}$
that is universal in the sense that the family $\pi :  X \to T$ should be
isomorphic to the pullback family\footnote{A slightly stronger version of (ii)
implies (iii). Namely, if in (ii) we also insist that the period mapping of the
pullback family $f^\ast X \to S$ of elliptic curves associated to $f:S \to T$ be
$\Phi\circ f :S \to M_{1,1}$, then the universal family in (iii) is the family
corresponding to the identity $M_{1,1}\to M_{1,1}$.}
$$
\xymatrix{
X \ar[r]^\simeq \ar[d]_\pi & \Phi^\ast E \ar[r]\ar[d] & E \ar[d] \cr
T \ar@{=}[r] & T \ar[r]^{\Phi} & M_{1,1}.
}
$$
The isomorphism $X\to \Phi^\ast E$ is unique up to an automorphism of the family
$X\to T$ that is the identity on the zero section.
\end{enumerate}
We will see shortly that the Riemann surface $M_{1,1}$ possesses the first
property, but not the second or third. Later in this lecture, we will define the
{\em orbifold} $\M_{1,1}$, which is endowed with a universal elliptic curve
$\E\to \M_{1,1}$ and has all three properties. In preparation for this, we first
consider families of {\em framed} elliptic curves.\footnote{The fancy
terminology is that $M_{1,1}$ is a {\em coarse} moduli space. The orbifold
$\M_{1,1}$ is a {\em fine} moduli space.}

\subsection{The universal elliptic curve $\E_\h$ over $\h$}

Recall that $\Lambda_\tau$ denotes the lattice $\Z\oplus\Z\tau$. It is easy
to construct a family of elliptic curves over $\h$ whose fiber over $\tau$
is $\C/\Lambda_\tau$.

The group $\Z^2$ acts on $\C\times \h$ on the left:
\begin{equation}
\label{eq:Z2}
(m,n) : (z,\tau)
\mapsto \bigg(z + \begin{pmatrix} m & n \end{pmatrix}
\begin{pmatrix} \tau \cr 1 \end{pmatrix}, \tau \bigg)
\end{equation}
This action is properly discontinuous and fixed point free. Therefore
the quotient
$$
\Z^2 \bs \big(\C\times \h\big)
$$
is a 2-dimensional complex manifold. Denote it by $\E_\h$. The projection
$\C\times \h \to \h$ induces a projection $\E_\h \to \h$ whose fiber over $\tau$
is $\C/\Lambda_\tau$.

\subsection{Families of framed elliptic curves}

Every family of elliptic curves
$$
\xymatrix{
X \ar[r]_\pi & T\ar@/_/[l]_s
}
$$
is locally trivial as a $C^\infty$ fiber bundle. For an open subset $U$ of $T$
set $X_U = \pi^{-1}(U)$. Suppose that $U$ is an open ball in $T$ over which $X$
is topologically trivial and $o\in U$:
$$
X_U \cong U\times X_o \text{ as a smooth manifolds, where } o\in U.
$$
Since $U$ is contractible, the inclusion $j_t : X_t \hookrightarrow X_U$ is a
homotopy equivalence for each $t\in U$. So if $s,t \in U$, then there are
natural isomorphisms
\begin{equation}
\label{eq:comparison}
\xymatrix{
H_1(X_t;\Z) \ar[r]^{j_{t\ast}} & H_1(X_U;\Z) & \ar[l]_{j_{s\ast}} H_1(X_s;\Z).
}
\end{equation}
A family of elements
$$
\big\{
c(t) \in H_1(X_t;\Z): t\in T\big\}
$$
is locally constant if for each open ball $U$ in $T$ over which $X_U$ is
topologically trivial and each pair $s,t$ of elements of $U$, $c(s)$ and $c(t)$
correspond under the isomorphism (\ref{eq:comparison}).

A family of framings
$$
\big\{
\aa(t),\bb(t) \in H_1(X_t;\Z): \aa(t)\cdot\bb(t) = 1,\ t\in T\big\}
$$
is locally constant if $\aa(t)$ and $\bb(t)$ are locally constant.

\begin{definition}
A family of elliptic curves is {\em framed} if its has a locally constant
framing.
\end{definition}

\begin{example}
The family $\E_\h \to \h$ is framed. The basis $\aa(\tau), \bb(\tau)$ of
$H_1(\C/\Lambda_\tau;\Z) \cong \Lambda_\tau$ is $1,\tau$.
\end{example}

\begin{exercise}
Show that every family of elliptic curves $X \to T$ over a simply connected
base $T$ has a framing.
\end{exercise}

Since $\h$ is the set of isomorphism classes of framed elliptic curves, a framed
family of elliptic curves $\pi : X \to T$ determines a function
$$
\Phi : T \to \h.
$$
It is defined by
$$
t \mapsto \textstyle{\int_{\bb(t)} \w_t/\int_{\aa(t)} \w_t}
$$
where $\w_t$ is any non-zero holomorphic differential on $X_t$. The mapping
$\Phi$ is called the {\em period mapping} of the family.

\begin{proposition}
If $\pi : X \to T$ is a family of framed elliptic curves, then the period
mapping $\Phi$ is holomorphic.
\end{proposition}

\begin{proof}[Sketch of Proof]
The main task is to show that we can choose the holomorphic differential $\w_t$
to vary holomorphically with $t$. In other words, we need to show that each
$o\in T$ has an open neighbourhood $U$ such that there is a holomorphic 1-form
$\w$ on $X_U$, defined modulo 1-forms that vanish on the fibers,  whose
restriction to $X_o$ is non-zero. More precisely, we need to construct an
element $\w$ of $H^0(U,\pi_\ast \Omega^1_{X/T})$ whose restriction to $X_o$ is
non-zero. By shrinking $U$ if necessary, the restriction $\w_t$ of $\w$ to $X_t$
will be non-zero for all $t\in U$.

Once we have done this, after further shrinking $U$ if necessary, we can
construct continuous mappings $\alpha$ and $\beta$ from $S^1 \times U \to X_U$
such that
$$
\xymatrix{
S^1 \times U \ar[rr]^{\alpha,\beta}\ar[dr]_{pr_2} && X_U\ar[dl]^{\pi} \cr
& U
}
$$
commutes and
\begin{enumerate}

\item for each $t \in U$, $\alpha_t : \theta \mapsto \alpha(\theta,t)$ and
$\beta_t : \theta \mapsto \beta(\theta,t)$ are piecewise smooth representatives
of $\aa(t)$ and $\bb(t)$;

\item for each $\theta$, $t \mapsto \alpha(\theta,t)$  and $t \mapsto
\beta(\theta,t)$ are holomorphic. (That this is possible follows from the
holomorphic implicit function theorem.)

\end{enumerate}
Basic calculus now implies that
$$
\int_{\alpha_t} \w_t \text{ and } \int_{\beta_t} \w_t
$$
vary holomorphically with $t\in U$ from which it follows that $\Phi(t)$ varies
holomorphically with $t \in T$.

We now establish the existence of $\w$. Let $N$ be the holomorphic normal bundle
in $X$ of the zero section $\im s$ of $X\to T$. This is a holomorphic line
bundle on the zero section of $X \to T$. Denote by $L$ the pullback to $T$ of
the dual of $N$ along the zero section $s : T \to X$. This has a local
holomorphic section $\sigma$ defined in a neighbourhood $U$ of $o\in T$ that
does not vanish at $o$. Since the holomorphic cotangent bundle of each $X_t$ is
trivial, there is a unique holomorphic differential $\w_t$ on $X_t$ whose value
at the identity $s(t)$ is $\sigma(t)$. The form defined is a holomorphic section
$\w$ of the sheaf $\pi_\ast \Omega^1_{X/T}$, as required.
\end{proof}

Each framed family $X \to T$ of elliptic curves determines a family of elliptic
curves by pulling back the family $\E_\h \to \h$ along the period mapping:
$$
\xymatrix{
\Phi^\ast \E_\h \ar[r]\ar[d] & \E_\h \ar[d] \cr
T \ar[r]_\Phi & \h
}
$$

\begin{exercise}
Show that the framed families $X \to T$ and $\Phi^\ast\E_\h \to T$ are naturally
isomorphic. That is, there is a biholomorphism $F : X \to \Phi^\ast \E_\h$ that
commutes with the projections to $T$:
$$
\xymatrix{
X \ar[r]^F\ar[d] & \Phi^\ast\E_\h\ar[d] \cr
T \ar@{=}[r] & T
}
$$
takes the zero section of $X$ to the zero section of $\Phi^\ast\E_\h$, and
preserves the framings. Hint: First observe that $(X_t,s(t);\aa(t),\bb(t))$ is
canonically isomorphic to $(\C/\Lambda_{\Phi(t)},0;1,\Phi(t))$. Show that these
isomorphisms can be assembled (locally) into a holomorphic mapping $X \to \E_\h$
by taking $x\in X_t$ to
$$
\textstyle{\int_c\w_t/\int_{\alpha(t)}\w_t} \mod \Lambda_{\Phi(t)}
$$
where $c$ is a smooth path in $X_t$ from $s(t)$ to $x$.
\end{exercise}

This proves that $\h$ is a fine moduli space for framed families of elliptic
curves.

\begin{proposition}
There is a 1-1 correspondence between framed families of elliptic curves $X \to
T$ and holomorphic mappings $\Phi :T \to \h$. Moreover, if $X \to T$ corresponds
to $\Phi : T \to \h$, then the framed family $X \to T$ is isomorphic to the
framed family $\Phi^\ast \E_\h \to T$.
\end{proposition}

\begin{remark}
Every family $X \to T$ of elliptic curves can be framed locally --- that is,
each $t\in T$ has a neighbourhood $U$ such that the restricted family $X_U \to
U$ has a framing and therefore admits a period mapping $\Phi_U : U \to \h$ so
that $X_U \to U$ is the pullback of the universal framed family $\E_\h \to \h$
along $\Phi_U$. The period mapping $T \to M_{1,1}$ associated to a family of
elliptic curves is thus ``locally liftable'' to a mapping $T \to \h$. Since the
action of $\SL_2(\Z)$ on $\h$ has fixed points, the identity $M_{1,1} \to
M_{1,1}$ is not locally liftable, and therefore not the period mapping of a
family of elliptic curves. It is this phenomenon which will lead us naturally
to orbifolds and stacks.
\end{remark}

The group $\SL_2(\Z)$ acts on the set of framings of a framed family $X \to T$
of elliptic curves via the formula
$$
\begin{pmatrix}a & b \cr c & d\end{pmatrix} :
\begin{pmatrix}\bb \cr \aa \end{pmatrix} \mapsto
\begin{pmatrix}a & b \cr c & d\end{pmatrix}
\begin{pmatrix} \bb \cr \aa \end{pmatrix}.
$$

\begin{exercise}
Show that if $\Phi : T \to \h$ is the period map of a family of elliptic
curves with framing $\aa,\bb$, then the period mapping with respect to the
framing
$$
\begin{pmatrix}a & b \cr c & d\end{pmatrix}
\begin{pmatrix} \bb \cr \aa \end{pmatrix}
$$
is $(a\Phi + b)/(c\Phi + d) : T \to \h$.
\end{exercise}

\subsection{The universal elliptic curve}

If
$$
\gamma = \begin{pmatrix}a & b \cr c & d\end{pmatrix} \in \SL_2(\Z)
$$
then the isomorphism
$$
(\C/\Lambda_\tau,0) \to (\C/\Lambda_{\gamma\tau},0)
$$
is induced by the mapping $z \mapsto (c\tau + d)^{-1} z$. This suggests that
we consider the action of $\SL_2(\Z)$ on $\C\times \h$ defined by
$$
\gamma : (z,\tau) \mapsto \big(z/(c\tau+d),(a\tau+b)/(c\tau+d)\big).
$$

\begin{exercise}
Prove that this is indeed an action.
\end{exercise}

We would like to combine this with the action of $\Z^2$ that we used to define
the universal curve over $\h$.

The group $\SL_2(\Z)$ acts on $\Z^2$ by {\em right} multiplication:
\begin{equation}
\label{eq:CH}
\begin{pmatrix} a & b \cr c & d\end{pmatrix} :
\begin{pmatrix} m & n \end{pmatrix} \mapsto
\begin{pmatrix}m & n\end{pmatrix}\begin{pmatrix} a & b \cr c & d\end{pmatrix}.
\end{equation}
Denote the corresponding semi-direct product $\SL_2(\Z) \ltimes \Z^2$ by $\G$.
This is the set $\SL_2(\Z)\times \Z^2$ with multiplication:
$$
(\gamma_1,v_1)(\gamma_2,v_2) = (\gamma_1\gamma_2,v_1\gamma_2+v_2)
$$
where $\gamma_1,\gamma_2 \in \SL_2(\Z)$ and $v_1,v_2 \in \Z^2$.

\begin{exercise}
Show that $\SL_2(\Z)\ltimes \Z^2$ is isomorphic to the group
$$
\Big\{
\begin{pmatrix} \gamma & 0 \cr v & 1 \end{pmatrix}
: \gamma \in \SL_2(\Z) \text{ and } v \in \Z^2
\Big\}
$$
\end{exercise}

\begin{exercise}
\label{ex:big_action}
Show that (\ref{eq:Z2}) and (\ref{eq:CH}) determine a well defined left action
of $\G$ on $\C\times \h$. Show that if $(\gamma,v) : (z,\tau) \to (z',\tau')$,
where
$$
\gamma = \begin{pmatrix} a & b \cr c & d\end{pmatrix} \text{ and } v = (m,n),
$$
then
$$
\begin{pmatrix}\tau' \cr 1 \cr z'\end{pmatrix}
= (c\tau + d)^{-1}
\begin{pmatrix} a & b & 0 \cr c & d & 0 \cr m & n & 1 \end{pmatrix}
\begin{pmatrix}\tau \cr 1 \cr z\end{pmatrix}
$$
\end{exercise}

Set
$$
E = (\SL_2(\Z) \ltimes \Z^2)\bs (\C\times \h).
$$
There is a projection $E \to M_{1,1}$.

\begin{exercise}
Show that the fiber of $E$ over the point $[X,P]$ of $M_{1,1}$ corresponding to 
the elliptic curve $(X,P)$ is $X/\Aut(X,P)$. (Cf.\ Exercise~\ref{ex:autos}.)
Show that if $\Aut (X,P)$ is cyclic of order 2, then $X/\Aut(X,P)$ is isomorphic
to the Riemann sphere $\P^1$. In particular, no fiber of $E$ is an elliptic
curve.
\end{exercise}

This problem can be rectified by pulling back the family $X\to T$ of elliptic
curves to the universal covering $p : \Ttilde \to T$ of $T$:
$$
\xymatrix{
p^\ast X \ar[r]\ar[d] & X \ar[d] \cr
\Ttilde \ar[r]^p & T
}
$$
Since $\Ttilde$ is simply connected, the family $p^\ast X \to \Ttilde$ admits
a framing. It is therefore obtained by pulling back the universal framed 
family $\E_\h \to \h$ along the period mapping $\Phi : \Ttilde \to \h$.

Note that the fiber of $p^\ast X$ over $t$ is canonically isomorphic to the
fiber of $X\to T$ over $p(t)$. This means that if $\gamma \in \Aut(\Ttilde/T)$
and $t\in\Ttilde$, then the fibers of $p^\ast X$ over $t$ and $\gamma t$ are
canonically isomorphic. So if $\aa(t),\bb(t)$ is a framing of $p^\ast X \to T$,
then $\aa(t),\bb(t)$ and $\aa(\gamma t),\bb(\gamma t)$ are both framings
of $H_1(X_{p(t)};\Z)$, and therefore differ by an element of $\SL_2(\Z)$.

Define a homomorphism $\phi : \Aut(\Ttilde/T) \to \SL_2(\Z)$ from the group of
deck transformations to $\SL_2(\Z)$ by
$$
\begin{pmatrix} \bb(\gamma t) \cr \aa(\gamma t) \end{pmatrix}
=
\phi(\gamma)
\begin{pmatrix} \bb(t) \cr \aa(t) \end{pmatrix}.
$$

\begin{exercise}
Show that the period mapping $\Phi$ is equivariant with respect to $\phi$ in the
sense that
$$
\Phi(\gamma t) = \phi(\gamma)\Phi(t)
$$
for all $t\in \Ttilde$ and $\gamma\in \Aut(\Ttilde/T)$.
\end{exercise}

\begin{exercise}
Show that the action of $\SL_2(\Z)\ltimes \Z^2$ on $\C\times \h$ induces an
action of $\SL_2(\Z)$ on $E_\h$.
\end{exercise}

\section{The Orbifold $\M_{1,1}$}

\subsection{Local theory: basic orbifolds}

The discussion in the previous section suggests a generalization of topological
spaces which includes quotients $\G\bs X$ and in which morphisms $\G\bs X \to
\G'\bs X'$ are $\G$-equivariant mappings $X \to X'$ with respect to a group
homomorphism $\phi : \G \to \G'$.

\begin{definition}
A basic {\em pointed orbifold} is a triple $(X,\G,\rho)$ where $X$ is a 
connected, simply connected topological space $X$ (typically a smooth manifold)
and $\G$ is a discrete group that acts on $X$ via the homomorphism $\rho : \G
\to \Aut X$. A {\em pointed morphism}
$$
(f,\phi) : (X,\G,\rho) \to (X',\G',\rho')
$$
of orbifolds consists of a continuous mapping $f : X \to X'$ and a group
homomorphism $\phi : \G \to \G'$ such that for all $\gamma\in \G$, the diagram
$$
\xymatrix{
X \ar[d]_{\rho(\gamma)}\ar[r]^f & X' \ar[d]^{\rho'(\phi(\gamma))} \cr
X \ar[r]^f & X'
}
$$
commutes. A morphism $(X,\G,\rho) \to (X',\G',\rho')$ of orbifolds is a
$\G$-orbit of pointed morphisms where $\G$ acts on the pointed morphism
$(f,\phi) : (X,\G,\rho) \to (X',\G',\rho')$ by conjugation:
$$
g : (f,\phi) \to (gfg^{-1},g\phi g^{-1}),\qquad g \in \G.
$$
Define the pointed orbifolds $(X,\G,\rho)$ and $(X,\G,\rho')$ to be {\em
equivalent} if there exists $g\in \G$ such that $\rho' = g\rho g^{-1}$. In this
case, there is an isomorphism
$$
(g,\text{conjugation by }g) : (X,\G,\rho) \to (X,\G,\rho').
$$
A {\em basic orbifold} is an equivalence class of pointed orbifolds.
\end{definition}

We will usually omit $\rho$ from the notation. We shall write $\G\bbs X$ for
$(X,\G)$, which we will regard as the orbifold quotient of $X$ by $\G$. When
$\G$ is trivial we shall denote the orbifold $(X,\triv)$ by $X$. The identity $X
\to X$ induces a natural quotient morphism $p:X \to \G\bbs X$, which we shall
regard as a universal covering of $\G\bbs X$.

The quotient mapping $p:X \to \G\bbs X$ should be thought of as a base point of
$\G\bbs X$. Pointed morphisms preserve these base points.

A path connected topological space $X$ with a universal covering $p: \Xtilde \to
X$ will be regarded as the orbifold $(\Xtilde,\Aut(\Xtilde/X))$.

\begin{exercise}
Show that if $\G$ acts on $X$, then there is an orbifold mapping from the
orbifold $\G\bbs X$ to the topological space $\G\bs X$.
\end{exercise}

\begin{exercise}
\label{ex:principal_bundles}
Suppose that the discrete group $\G$ acts (trivially) on the one point space
$\ast$. Show that there is a 1-1 correspondence between orbifold mappings $T \to
\G\bbs \ast$ from a topological space $T$ to $\G\bbs \ast$ and isomorphism
classes of principal $\G$-bundles over $T$.
\end{exercise}

We shall regard any $\G$-invariant structure on $X$ as a structure on the
orbifold quotient $\G\bbs X$. For example, if $X$ is a Riemann surface with a
$\G$-invariant complex structure, then we will regard $\G\bbs X$ as a Riemann
surface in the category of orbifolds. The holomorphic functions on $\G\bbs X$
are, by definition, the $\G$-invariant holomorphic functions on $X$. Properties
of holomorphic functions between Riemann surfaces extend to holomorphic mappings
$f : \G\bbs X \to \G'\bbs X'$ between orbifold Riemann surfaces. For example, we
say that $f$ is unramified if the mapping $X \to X'$ of ``universal coverings''
is unramified.

The quotient of a non-simply connected space $X$ by a discrete group $\G$ has a
natural orbifold structure. Suppose, for simplicity, that $X$ is path connected.
Suppose that $p:\Xtilde \to X$ is a universal covering of $X$. Set $G =
\Aut(\Xtilde/X) = \pi_1(X,p)$. Define $\Gtilde$ to be the group consisting of
the pairs $(\gamma,g) \in \G \times \Aut \Xtilde$ such that the diagram
$$
\xymatrix{
\Xtilde \ar[r]^g \ar[d] & \Xtilde\ar[d] \cr
X \ar[r]^{\gamma} & X
}
$$
commutes. This is an extension
$$
1 \to G \to \Gtilde \to \G \to 1
$$ 
Define $\G\bbs X$ to be the orbifold $\Gtilde\bbs \Xtilde$.

\begin{remark}
\label{rem:groupoid}
Orbifolds are examples of stacks. Stacks can be defined as groupoids in an
appropriate category, such as the category of complex analytic manifolds. (See
Appendix~\ref{app:stacks} for the definition.) A basic orbifold $(X,\G)$ may be
viewed as a groupoid in the category of topological spaces. The set of objects
of the groupoid is $X$, and the set of morphisms is $\G \times X$. The morphism
$(\gamma,x)$ has source $x$ and target $\gamma x$. Two morphisms $(\gamma,x)$
and $(\mu,y)$ are composable when $y = \gamma x$. Their composition is given by
$$
(\mu,\gamma x) \circ (\gamma,x) = (\mu\gamma,x).
$$
The inverse of $(\gamma,x)$ is $(\gamma^{-1},\gamma x)$. For more details,
see Appendix~\ref{app:stacks}.
\end{remark}

\subsection{Points*}

One has to distinguish various kinds of ``points'' of orbifolds. Suppose that
$\G\bbs X$ is an orbifold. A point $x$ of $X$ corresponds to an inclusion $i_x :
\ast \to X$. The composite
$$
\xymatrix{\ast \ar[r]^{i_x} & X \ar[r] & \G\bbs X}
$$
will be denoted $x : \ast \to \G\bbs X$ and regarded as a {\em closed point} of
$\G\bbs X$. The closed point $x$ induces an orbifold mapping
$$
\ibar_x : \G_x \bbs \ast \to \G\bbs X
$$
where $\G_x$ denotes the isotropy group of $x$ in $\G$. When $\G_x$ is finite
and non-trivial we will call $\ibar_x$ an {\em orbifold point} of $\G\bbs X$. In
this case, we define the {\em degree} $\deg(x)$ of $x$ to be the order of if its
isotropy group $\G_x$.

Two closed points $x,x' \in X$ are said to be {\em conjugate} if they lie in the
same $\G$ orbit. Denote the conjugacy class of $x$ by $(x)$. Conjugacy classes
of points of $X$ are in 1-1 correspondence with the points of the orbit space
$\G\bs X$. If $\G$ acts virtually freely\footnote{That is, $\G$ has a finite
index subgroup that acts fixed point freely on $X$.} on $X$, then conjugate
closed points have the same degree.

\subsection{Homotopy theory of basic orbifolds}
\label{sec:htpy_type}

Suppose that $(X,\G,\rho)$ is a pointed orbifold. Denote the unit interval
$[0,1]$ by $I$. Define $I\times(X,\G,\rho)$ to be the pointed orbifold $(I\times
X,\G,\id_I\times\rho)$, where $\id_I\times\rho : \G \to \Aut(I\times X)$ is
given by
$$
\id_I\times\rho(\gamma) : (t,x) \mapsto (t,\rho(\gamma)(x)).
$$

\begin{definition}
A homotopy between two morphisms $(f,\phi),(g,\psi) : (X,\G) \to
(X',\G')$ of pointed orbifolds is a morphism
$$
(F,\phi) : I\times (X,\G) \to (X',\G')
$$
of pointed orbifolds that satisfies
\begin{enumerate}

\item 
$\phi = \psi$,

\item $\phi : \G \to \G'$ is a homomorphism,

\item  $f(x) = F(0,x)$ and $g(x)=F(1,x)$ for all $x\in X$.

\end{enumerate}
\end{definition}

Homotopy of orbifold morphisms is an equivalence relation. Two pointed orbifolds
$(X,\G)$ and $(X',\G')$ are defined to be {\em homotopy equivalent} if there are
morphisms $(f,\phi) : (X,\G) \to (X',\G')$ and $(g,\psi) : (X',\G') \to (X,\G)$
such that $(g,\psi)\circ(f,\phi)$ is homotopic to
$\id_{(X,\G)}$ and $(f,\phi)\circ(g,\psi)$ is homotopic to $\id_{(X',\G')}$.

Denote by $S^1$ the orbifold $(\R,\Z)$, where $\Z$ acts on $\R$ by translation.
The fundamental group $\pi_1(\G\bbs X,p)$ of the pointed orbifold $(X,\G)$ with
respect to $p : X\to \G\bbs X$ is defined by
\begin{multline*}
\pi_1(\G\bbs X,p) \cr
= \big\{\text{homotopy classes of pointed morphisms }
(f,\phi) : S^1 \to (X,\G)\big\}.
\end{multline*}
Denote the homotopy class of $(f,\phi) : S^1 \to (X,\G)$ by $[f,\phi]$.

\begin{exercise}
Show that the fundamental group of a basic orbifold is a group. Show that the
function $\pi_1(\G\bbs X,p) \to \G$ that takes $[f,\phi]$ to $\phi(1)$ is a
group
isomorphism.
\end{exercise}

The orbifold and usual fundamental groups of $\G\bbs X$ agree when $\G$ acts
freely and discontinuously on $X$. In particular, if $X$ is a topological space
with universal covering $p:\Xtilde \to X$, then $\pi_1(X,p) = \Aut(\Xtilde/X)$,
which agrees with the standard definition of $\pi_1(X,p)$.

Define a morphism $(f,\phi) : (X,\G) \to (X',\G')$ between two pointed orbifolds
to be a weak homotopy equivalence if $\phi$ is an isomorphism and $f : X \to X'$
induces an isomorphism $H_\dot(X) \to H_\dot(X')$.\footnote{Since $X$ and $X'$
are simply connected, this is equivalent, by a classical theorem of Hurewicz,
to the statement that $X \to X'$ induces an isomorphism on homotopy groups. A
proof may be found in a standard text such as \cite{spanier}.} \footnote{A
classical theorem of J.~H.~C.\ Whitehead states that a weak homotopy equivalence
of CW-complexes is a homotopy equivalence. This is proved in many standard
texts, such as \cite{spanier}.}

Every basic orbifold $(X,\G)$ is weak homotopy equivalent to a topological
space. Indeed, if $(X,\G)$ is a basic orbifold, then we can consider the
orbifold
$$
(E\G\times X,\G)
$$
where $E\G$ is any contractible space on which $\G$ acts properly
discontinuously and fixed point freely and where $\G$ acts diagonally on
$E\G\times X$.\footnote{One can take $E\G$ to be the simplicial complex whose
set of $n$-simplices is $\G^{n+1}$. The group $\G$ acts diagonally on this
complex. This complex is contractible as it is a cone with vertex the identity.}
The projection
$$
(E\G\times X,\G) \to (X,\G)
$$
is a weak homotopy equivalence.\footnote{If $X$ has the homotopy type of a
CW-complex and $\G$ acts properly discontinuously and fixed point freely on $X$,
then this morphism is a homotopy equivalence. The homotopy inverse is given by
any $\G$-invariant mapping $X \to E\G\times X$ that is the identity in the
second factor.}

In other words, the weak homotopy type of a basic orbifold $(X,\G)$ is the
homotopy type of the homotopy quotient $E\G\times_\G X := \G\bs (E\G\times X)$
of $X$ by $\G$. This motivates the following definition of the homology,
cohomology and higher homotopy groups of a basic orbifold, which agree with
the standard definitions on topological spaces.

A {\em local system} $\V$ on the pointed orbifold $(X,\G)$ with fiber $V$ over
the base point $p$ is simply a representation $\G \to \Aut V$. We will denote
the corresponding local system  on $E\G\times_\G X$ by $\V$.

\begin{definition}
Suppose that $(X,\G)$ is a basic orbifold and that $V$ is a $\G$-module. Define
the higher homotopy, homology, and cohomology groups of the orbifold $\G\bbs X$
by
\begin{enumerate}

\item $\pi_n(\G\bbs X,p) := \pi_n(X)$, when $n\ge 2$,

\item $H_\dot(\G\bbs X;\V) = H_\dot(E\G\times_\G X;\V)$,

\item $H^\dot(\G\bbs X;\V) = H^\dot(E\G\times_\G X;\V)$.

\end{enumerate}
\end{definition}

Note that $H^\dot(E\G\times_\G X;\V)$ is the $\G$-equivariant cohomology
$H^\dot_\G(X;\V)$ of $X$.

\begin{exercise}
Show that there is a natural isomorphism
$$
\pi_1(\G\bbs X,p) \cong \pi_1(E\G\times_\G X,p')
$$
where $p'$ is the quotient mapping $E\G\times X\to E\G\times_\G X$.
\end{exercise}

\begin{example}
The homotopy type of the orbifold quotient $\G\bbs \ast$ of a one-point space
$\ast$ by a discrete group $\G$ is that of the classifying space $B\G := \G\bs
E\G$ of $\G$. The cohomology groups of $\G\bbs \ast$ with coefficients in the
local system that corresponds to the $\G$-module $V$ are those of the $B\G$,
which, by definition, are the cohomology groups of $\G$:
$$
H^\dot(\G\bbs \ast;\V) = H^\dot(\G,V):= H^\dot(B\G;\V).
$$
The higher homotopy groups of $\G\bbs \ast$ vanish.
\end{example}

\begin{example}
The orbifold quotient $\bmu_d\bbs\D$ of the unit disk $\D$ by the natural action
of the group $\bmu_d$ of $d$th roots of unity on $\D$. The quotient of $\D$ by
$\bmu_d$ in the category of topological spaces is the disk, which is simply
connected. Thus the topological and orbifold quotients can be different, even
when the group action is effective. The projection $\D \to \ast$ is a
$\bmu_d$-equivariant homotopy equivalence which induces a homotopy equivalence
$$
\bmu_d\bbs \D \simeq \bmu_d\bbs \ast = B\bmu_d.
$$
The cohomology groups of $\bmu_d\bbs \D$ are therefore those of the group
$\bmu_d$. In particular, $\bmu\bbs \D$ does not have finite cohomological
dimension.
\end{example}

A vector bundle over the orbifold $\G\bbs X$ is a vector bundle $V \to X$
together with a lift of the $\G$-action on $X$ to $V$. The bundle over $\G\bbs
X$ is denoted by $\G\bbs V \to \G\bbs X$. Sections of this bundle correspond
precisely to $\G$ invariant sections of $V \to X$.

\begin{example}
A vector bundle over $\G\bbs \ast$ is simply a $\G$-module $V$. Its space of
sections consists of all $\G$-equivariant functions $f : \ast \to V$ and is
therefore the subspace $V^\G$ of $\G$-invariant vectors in $V$:
$$
H^0(\G\bbs \ast, V) = V^\G.
$$
\end{example}

\begin{example}
If $X$ is a manifold, then a smooth $\G$-action on $X$ lifts naturally to a
smooth $\G$-action on the tangent bundle $TX$ of $X$. In this case, the quotient
$\G\bbs X$ can be regarded as a manifold in the category of orbifolds with
tangent bundle $\G\bbs (TX) \to \G\bbs X$. Denote this by $T(\G\bbs X)$.
Sections of $T(\G\bbs X)$ are $\G$-invariant vector fields on $X$. This example
extends to all natural bundles on $X$, such as the cotangent bundle and its
exterior powers. In particular, differential $k$-forms on $\G\bbs X$ are
$\G$-invariant sections of $\Lambda^k T^\ast X \to X$, which are just
$\G$-invariant differential forms on $X$.
\end{example}

Note that the de~Rham theorem does not hold for all orbifolds. For example,
it does not hold for $\Z\bbs \ast$.

\begin{definition}
\label{def:vfree}
An action of a group $\G$ on a space $X$ is {\em virtually free} if $\G$ has a
finite index subgroup $\G'$ that acts freely.
\end{definition}

The action of $\SL_2(\Z)$ on $\h$ is virtually free by
Exercise~\ref{ex:quotient}.

\begin{exercise}
Show that if $\G$ is a finitely generated discrete group that acts properly
discontinuously and virtually freely on the simply connected manifold $X$, then
there is a natural ring isomorphism
$$
H^\dot(\G\bbs X;\R) \cong
H^\dot(\text{$\G$-invariant, smooth real-valued forms on $X$}).
$$
\end{exercise}

\subsection{Orbifold Euler characteristic}
\label{para:euler}
The notion of Euler characteristic of a finite complex extends to orbifolds that
satisfy some finiteness restrictions.

Suppose that the discrete group $\G$ acts virtually freely and properly
discontinuously on $X$. If $\G'$ is a finite index normal subgroup $\G'$ that
acts freely and discontinuously on $X$, then the quotient mapping $X \to \G'\bs
X$ is an unramified covering. The map $\G'\bs X \to \G\bbs X$ is an orbifold
morphism that can be thought of as an unramified covering of degree $[\G:\G']$.

When $\G'\bs X$ is a finite complex, we can define the {\em orbifold Euler
characteristic} of $\G\bbs X$ by
$$
\chi\big(\G\bbs X\big) = \frac{1}{[\G:\G']}\chi\big(\G'\bs X\big)
$$
where $\chi(\G'\bs X)$ is the usual Euler characteristic of $\G'\bs X$.

\begin{exercise}
Show that the orbifold Euler characteristic of $\G\bbs X$ is well defined ---
that is, it is independent of the choice of the finite index subgroup $\G'$.
\end{exercise}

\begin{example}
The orbifold Euler characteristic of $\bmu_d\bs \D$ is $1/d$.
\end{example}

\begin{exercise}
Suppose that $X$ is a finite simplicial complex and that $\G$ acts simplicially
on $X$. Show that if $\G$ acts virtually freely on $X$, then the orbifold Euler
characteristic of the semi-simplicial complex $\G\bbs X$ is given by
\begin{equation}
\label{eqn:simp}
\chi\big(\G\bbs X\big) = \sum_{k\ge 0} (-1)^k \sum_{\sigma \in (\G\bs X)_k}
\frac{1}{|\G_\sigma|} 
\end{equation}
where $(\G\bs X)_k$ denotes the set of $k$-simplices of $\G\bs X$ and
$\G_\sigma$ denotes the isotropy group of $\sigma$.
\end{exercise}

\begin{example}
Suppose that $K$ is the hexagonal decomposition of the disk. Let $\G$ be the
symmetric group $S_3$, which is generated by the reflections in the 3 diagonals.
This action is simplicial.
\begin{figure}[!ht]
\epsfig{file=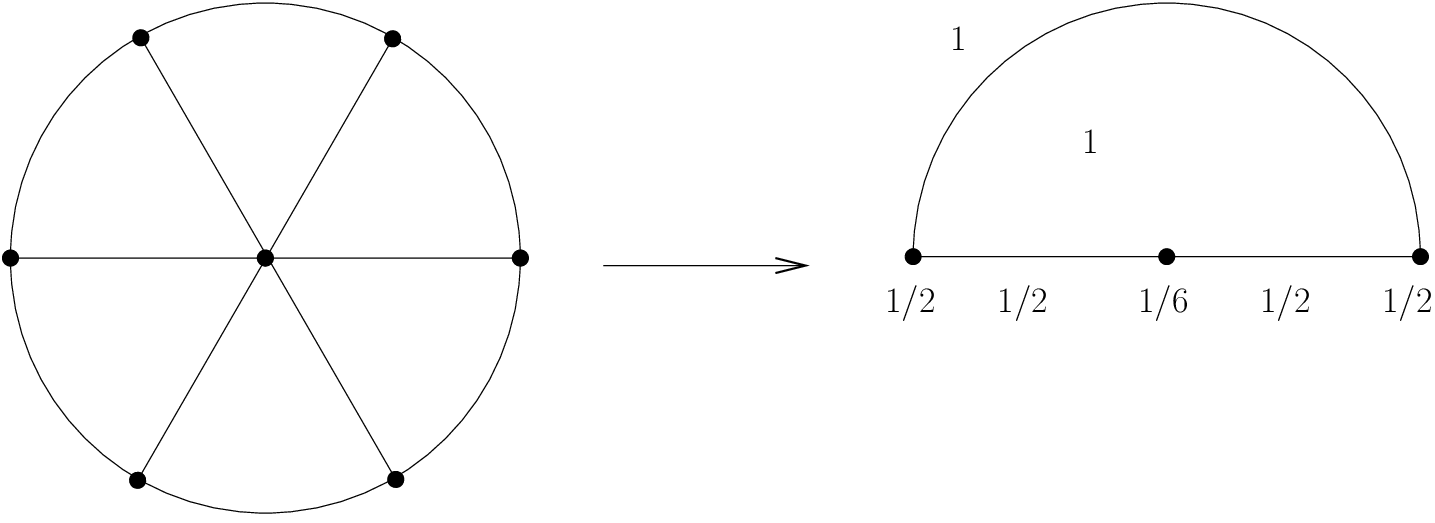, width=2.75in}
\caption{The quotient map}\label{fig:quot}
\end{figure}
The orders of the stabilizers of orbit representatives are shown in
Figure~\ref{fig:quot}. For example, the origin is ``$1/6$ of a point'' and the
orbit of a diagonal consists of two copies of a ``1/2'' edge. Formula
(\ref{eqn:simp}) thus gives
$$
\chi(S_3\bbs K) = (1/6 + 1/2 + 1/2) - (1/2 + 1/2 + 1) + 1 = 1/6 = \chi(K)/|S_3|
$$
as it should.
\end{example}

\subsection{The orbifold $\M_{1,1}$} Our primary example of an orbifold is the
moduli space of elliptic curves.

\begin{definition}
Define $\M_{1,1}$ to be the orbifold $\SL_2(\Z)\bbs\h$. It is a Riemann surface
in the category of orbifolds. There is a natural isomorphism
$$
\pi_1(\M_{1,1},p)\cong \SL_2(\Z)
$$
where the base point is the covering projection $p : \h \to \M_{1,1}$.
\end{definition}

Recall that $M_{1,1}$ is the quotient $\SL_2(\Z)\bs \h$ in the category of
Riemann surfaces. It is called the {\em coarse moduli space associated to}
$\M_{1,1}$. There is a natural morphism $\M_{1,1} \to M_{1,1}$. Each elliptic
curve $(E,P)$ determines a point $[E,P]$ of $\M_{1,1}$, which is called the {\em
moduli point of $(E,P)$}.

Families of elliptic curves $X \to T$ whose coarse period mapping $T \to
M_{1,1}$ is constant are said to be {\em isotrivial}.

\begin{exercise}[isotrivial families]
\label{ex:isotrivial1}
Show that if $T$ is a compact Riemann surface and $X \to T$ is a family of
elliptic curves over $T$, then the coarse period mapping $T \to M_{1,1}$ is
constant. (Hint: $M_{1,1}$ is a non-compact Riemann surface.) Examples of such
families over a (not necessarily compact) base $T$ can be constructed as
follows: fix an elliptic curve $(E,0)$ and an automorphism $\sigma \in
\Aut(E,0)$ of order $d$. Choose a cyclic unramified covering $S \to T$ of degree
$d$. Choose a generator $\phi$ of $\Aut(S/T)$. Then $\Z/d\Z$ acts diagonally on
$E\times S$ by $k : (x,s)\mapsto (\sigma^k(x),\phi^k(s))$. Define $X \to T$ to
be the quotient
$$
(\Z/d\Z)\bs(E\times S) \to (\Z/d\Z)\bs S.
$$
Show that $X\to T$ is a family of elliptic curves and that the coarse period
mapping $T \to M_{1,1}$ takes the constant value $[E,0]$. Show that the period
mapping $T\to \M_{1,1}$ factors through $\Aut(E,0)\bbs \ast$. Finally, show that
a family of elliptic curves $X\to T$ is isotrivial if and only if its coarse
period mapping is constant. 
\end{exercise}

\begin{proposition}
The low-dimensional homology and cohomology groups of $\M_{1,1}$ are
$$
H_1(\M_{1,1};\Z) = \Z/12\Z,\quad H^1(\M_{1,1};\Z) = 0,\quad
H^2(\M_{1,1};\Z) = \Z/12\Z.
$$
The morphism $\M_{1,1} \to M_{1,1}$ induces an isomorphism on rational homology
and rational cohomology, so that $\M_{1,1}$ has the rational homology and
cohomology of a point.
\end{proposition}

\begin{proof}
Since $\h$ is contractible, $\M_{1,1}$ has the homotopy type of the classifying
space $B\SL_2(\Z)$ of $\SL_2(\Z)$. Therefore
$$
H_\dot(\M_{1,1};\Z) \cong H_\dot(\SL_2(\Z);\Z)
\text{ and }
H^\dot(\M_{1,1};\Z) \cong H^\dot(\SL_2(\Z);\Z).
$$
Since $\SL_2(\Z)$ is finitely presented, this implies that its homology and
cohomology are finitely generated in degrees 1 and 2.\footnote{It is not
difficult to see that they are finitely generated in all degrees as $B\SL_2(\Z)$
can be realized as a CW-complex with a finite number of cells in each degree.}
In particular, $H_1(\M_{1,1};\Z)$ is the maximal abelian quotient of
$\SL_2(\Z)$. Using the presentation (\ref{eqn:presentation}) of $\SL_2(\Z)$, we
have
$$
H_1(\M_{1,1};\Z) = (\Z s \oplus \Z u)/\langle 2s-3u,\ 4s\rangle \cong \Z/12\Z
$$
from which it follows that
$$
H^1(\M_{1,1};\Z) = \Hom(H_1(\M_{1,1}),\Z) = 0.
$$
By Exercise~\ref{ex:level_tf}, the finite index, normal subgroup $\SL_2(\Z)[m]$
of $\SL_2(\Z)$ is free when $m\ge 3$.\footnote{A group $\G$ is {\em virtually
free} if it has a free subgroup of finite index. Thus $\SL_2(\Z)$ is virtually
free.} Standard arguments imply that
$$
H^k(\SL_2(\Z);V) = H^k(\SL_2(\Z)[m];V)^{\SL_2(\Z/m)}
$$
whenever $V$ is a $\Q$-module. Since $\SL_2(\Z)[m]$ is free, these vanish when
$k\ge 2$. It follows that $H^2(\SL_2(\Z);\Z)$ is a finitely generated torsion
group. The universal coefficient theorem implies that
$$
H^2(\SL_2(\Z);\Z) \cong \Hom(H_1(\SL_2(\Z),\Q/\Z)) \cong \Z/12\Z.
$$
\end{proof}

The purpose of the following examples and exercises is to give some feel for the
orbifold structure of $\M_{1,1}$.

\begin{exercise}
\label{ex:no_lift}
Consider the orbifold morphism $\M_{1,1} \to M_{1,1}$. Show that if $j
:\D\hookrightarrow M_{1,1}$ is the inclusion of a coordinate disk centered at
the orbit $[i]$ of $i$ (or $[\rho]$ of $\rho$), then there is no orbifold lift
$\jtilde : \D \to \M_{1,1}$ of the restriction of $j$ to $\D$. Deduce that there
is no universal curve over the coarse moduli space $M_{1,1}$.
\end{exercise}

\begin{example}
Denote the subgroup $\{\pm I\}$ of $\SL_2(\Z)$ by $C_2$. It acts trivially on
$\h$. The fundamental group of the quotient $\{\pm I\}\bbs\h$ is cyclic of order
2. The projection $\h \to C_2\bbs\h$ is viewed as a 2:1 cover of orbifolds, even
though it is a homeomorphism in the category of topological spaces. The group
$\PSL_2(\Z)$, which is the quotient $\SL_2(\Z)/\{\pm I\}$, acts faithfully on
$\h$. The orbifold quotient $\PSL_2(\Z)\bbs\h$ has fundamental group
$\PSL_2(\Z)$ and is not isomorphic to $\M_{1,1}$.
\end{example}

Given two sets $\{a_1,a_2,a_3\}$ and $\{b_1,b_2,b_3\}$ of 3 distinct points of
the Riemann sphere $\P^1$, there is a unique $\phi \in \Aut \P^1$ such that
$\phi(a_j) = b_j$ for all $j$. The permutations of $\{0,1,\infty\}$ therefore
define an action of the symmetric group $S_3$ on $\P^1$ which restricts to a
homomorphism $S_3 \hookrightarrow \Aut \Pminus$. The group $C_2\times S_3$ acts
on $\Pminus$ via the projection onto $S_3$:
$$
C_2 \times S_3 \to S_3 \hookrightarrow \Aut(\C-\{0,1\}).
$$

\begin{proposition}
\label{prop:quot}
The orbifold Riemann surface $\M_{1,1}$ is isomorphic to the orbifold
$(C_2\times S_3)\bbs(\C-\{0,1\})$.
\end{proposition}

\begin{proof}
The level 2 subgroup $\SL_2(\Z)[2]$ of $\SL_2(\Z)$ is the subgroup of
$\SL_2(\Z)$ consisting of matrices congruent to the identity mod 2. The quotient
$\SL_2(\Z)/\SL_2(\Z)[2]$ is isomorphic to $\SL_2(\F_2)$, which is isomorphic to
the symmetric group $S_3$. It follows from Exercise~\ref{ex:level_tf} that the
image $\PSL_2(\Z)[2]$ of $\SL_2(\Z)$ in $\PSL_2(\Z)$ is torsion free. The
quotient $\PSL_2(\Z)[2]\bs \h$ is thus a Riemann surface. It is biholomorphic to
$\Pminus$ as can be seen, for example, by considering the fundamental domain of
the action of $\SL_2(\Z)$ on $\h$. (See \cite{clemens}.) It follows that
$\PSL_2(\Z)[2]$ is a free group $F_2$ of rank 2. Choose a splitting of
$\SL_2(\Z)[2]\to F_2$ (not unique!) and use it to identify $\SL_2(\Z)[2]$ with
$F_2\times C_2$. Then
$$
\SL_2(\Z)[2]\bbs \h \cong C_2\bbs (F_2\bs \h) \cong C_2\bbs(\C-\{0,1\}.
$$
Since the actions of $C_2$ and $S_3$ on $\C-\{0,1\}$ commute, we have
$$
\M_{1,1} \cong S_3 \bbs \big(C_2\bbs(\C-\{0,1\})\big)
\cong (C_2\times S_3)\bbs (\C-\{0,1\}).
$$
\end{proof}

Since $\chi(\C-\{0,1\})=-1$ and $C_2\times S_3$ has order 12, we have:

\begin{corollary}
The orbifold Euler characteristic of $\M_{1,1}$ is $-1/12$.
\end{corollary}

\begin{exercise}
\label{ex:disk}
Show that the function $q : \h \to \D^\ast$ defined by $q(\tau) = \exp(2\pi
i\tau)$ induces an orbifold isomorphism
$$
\{\pm 1\}\PZ\bbs \h \to C_2\bbs\D^\ast.
$$
Deduce that there is an orbifold morphism $\D^\ast \to \M_{1,1}$ which factors
through the quotient of $\D^\ast$ by the {\em trivial} $C_2$-action:
$$
\xymatrix{
\D^\ast \ar[r] & C_2 \bbs \D^\ast \ar[r] & \M_{1,1}.
}
$$
\end{exercise}

\begin{example}
The cyclic group $C_2 = \{\pm I\}$ acts on the line bundle $\C\times \h \to \h$
by $-I : (z,\tau) \mapsto (-z,\tau)$. Sections of the orbifold line bundle
$$
C_2\bbs (\C\times \h) \to C_2 \bbs \h
$$
correspond to $C_2$ invariant functions $f : \h \to \C$, and are thus zero.
\end{example}

\begin{exercise}
\label{ex:L}
Show that the function
$$
\SL_2(\Z)\times \C\times \h \to \C \times \h
$$
defined by
$$
\begin{pmatrix} a & b \cr c & d \end{pmatrix} :
(z,\tau) \mapsto \big((c\tau+d)^k z,(a\tau+b)/(c\tau +d) \big)
$$
is an action that lifts the standard action of $\SL_2(\Z)$ on $\h$. Set
$$
\L_k = \SL_2(\Z)\bs(\C\times \h)
$$
This is an orbifold line bundle over $\M_{1,1}$. Show that $\L_k = \L_1^{\otimes
k}$. Show that the holomorphic sections of $\L_k$ correspond to holomorphic
functions $f : \h \to \C$ that satisfy
$$
f\big((a\tau+b)/(c\tau+d)\big) = (c\tau+d)^k f(\tau).
$$
Show that $\L_k$ has no non-zero sections when $k$ is odd.
\end{exercise}

\subsection{The universal elliptic curve $\E \to \M_{1,1}$}

Define $\E$ to be the orbifold quotient
$$
(\SL_2(\Z)\ltimes \Z^2)\bs (\C \times \h)
$$
where the action is defined in Exercise~\ref{ex:big_action}. The projection
$\C\times \h \to \h$ induces an orbifold morphism
$$
\E \to \M_{1,1}.
$$
It is a family of elliptic curves over $\M_{1,1}$ in the category of orbifolds.

\begin{exercise}
Show that every orbifold morphism $T \to \M_{1,1}$ is a locally liftable
mapping $T\to M_{1,1}$. Show that the universal elliptic curve $\E \to \M_{1,1}$
pulls back along an orbifold morphism $\Phi : T \to \M_{1,1}$ to a family of
elliptic curves $X \to T$. Use this to prove the following theorem.
\end{exercise}

\begin{theorem}
\label{thm:families}
There is a 1-1 correspondence between isomorphism classes of families of
elliptic curves $X \to T$ over a complex manifold and holomorphic orbifold
morphisms $T \to \M_{1,1}$. This orbifold morphism is induced by the period
mapping . The morphism $\Phi : T \to \M_{1,1}$ corresponds to the isomorphism
class of the pullback family $\Phi^\ast \E \to T$.
\end{theorem}

\begin{exercise}
Denote the $\SL_2(\Z)$-orbit of $\tau\in \h$ by $[\tau]$. Show that the
inclusion $j: M_{1,1}-\{[i],[\rho]\} \hookrightarrow M_{1,1}$ is locally
liftable to a map to $\h$, but that there is no orbifold mapping $\jtilde$ such
that the diagram
$$
\xymatrix{
& \M_{1,1} \ar[d]\cr
M_{1,1}-\{[i],[\rho]\} \ar[r] \ar@{..>}[ur]^{\jtilde} & M_{1,1}
}
$$
commutes. Deduce that there is no universal elliptic curve over either $M_{1,1}$
or $M_{1,1}-\{[i],[\rho]\}$.
\end{exercise}

Theorem~\ref{thm:families} is more subtle than it may at first appear due to the
subtleties of orbifold mappings. This is illustrated by {\em isotrivial
families} --- non-trivial families with constant period mappings.

\begin{example}
\label{ex:isotrivial2}
This example is a continuation of Exercise~\ref{ex:isotrivial1}. Let $X \to T$
be the isotrivial family of elliptic curves associated to $\sigma \in \Aut(E,0)$
and an unramified covering $S\to T$.  The coarse period mapping $T \to M_{1,1}$
takes the constant value $[E,0]$. Even though the period mapping $T \to M_{1,1}$
to the coarse moduli space is constant, the orbifold period mapping $T \to
\M_{1,1}$ is non-trivial when $\sigma$ is non-trivial. To compute the period
mapping, fix a framing $\aa,\bb$ of $H_1(E;\Z)$. Let
$$
\tau = \textstyle{\int_\bb \w /\int_\aa \w}.
$$
This is the point of $\h$ that corresponds to the framed elliptic curve
$(E,0;\aa,\bb)$. Define $A\in \SL_2(\Z)$ by
$$
\sigma_\ast
\begin{pmatrix}
\bb \cr \aa
\end{pmatrix}
= A
\begin{pmatrix}
\bb \cr \aa
\end{pmatrix}
$$
Then $A\tau = \tau$ and the induced mapping $\rho : \pi_1(T) \to \pi_1(\M_{1,1}) =
\SL_2(\Z)$ is the composite $\pi_1(T) \to \Aut(S/T) \to \SL_2(\Z)$ where $\phi
\in \Aut(S/T)$ is mapped to $A$.
The period mapping $T \to \M_{1,1}$ is represented by the mapping
$(\Ttilde,\pi_1(T)) \to (\h,\SL_2(\Z))$ that takes $(t,\gamma)$ to
$(\tau,\rho(\gamma))$.
\end{example}

\section{The Orbifold $\Mbar_{1,1}$ and Modular Forms}

In this section we explain how to construct the orbifold compactification
$\Mbar_{1,1}$ of $\M_{1,1}$; it is the prototypical example of an orbifold
obtained by patching. We are able to do this as the orbifold $\Mbar_{1,1}$ is
easy to define because it is obtained by gluing two basic orbifolds along
another basic orbifold. In general, to construct an orbifold, it is necessary to
glue more than two basic orbifolds. In this case, one has a compatibility
condition for the gluing maps that is analogous to the familiar cocycle
condition $g_{\alpha\gamma} = g_{\alpha\beta}g_{\beta\gamma}$. For completeness,
we  discuss stacks briefly in Appendix~\ref{app:stacks}. Analytical (resp.\
topological) orbifolds are stacks in the category of analytic varieties (resp.\
topological spaces).

To construct the orbifold $\Mbar_{1,1}$, we begin with the diagram
$$
\xymatrix{
& \h\ar@(ur,ul)[]_{C_2\times\Z}
\ar[dl]_p\ar[dr]^q \cr
\h\ar@(ul,dl)[]_{\SL_2(\Z)} && \ar@(dr,ur)[]_{C_2}\D
}
$$
of spaces with compatible group actions. Here $C_2\times \Z$ acts on $\h$ by
$(\pm1,n) : \tau \mapsto \tau + n$, $C_2$ acts trivially on $\D$, $q(\tau) =
\exp(2\pi i\tau)$, $p$ is the identity, and $(\pm 1,n) \in C_2\times \Z$ is
mapped to $\pm\begin{pmatrix} 1 & n \cr 0 & 1 \end{pmatrix} \in \SL_2(\Z)$ and
to $\pm 1$ in $C_2$. These induce orbifold mappings
$$
\xymatrix{
& C_2\bbs \D^\ast \ar[dl]\ar[dr] \cr
\M_{1,1} && C_2\bbs\D
}
$$
where we identify $\Z\bs\h$ with the punctured $q$-disk $\D^\ast$, via the
mapping $\tau \mapsto \exp(2\pi i \tau)$, and the right hand arrow is the
quotient of the inclusion $\D^\ast \hookrightarrow \D$ by the trivial $C_2$
action.

The compactification $\Mbar_{1,1}$ of $\M_{1,1}$ is essentially obtained by
adding one point with automorphism group $C_2$ to $\M_{1,1}$. Formally,
$\Mbar_{1,1}$ is the orbifold obtained by gluing $\M_{1,1}$ and $C_2\bbs\D$
together via $C_2\bbs\D^\ast$. It is a Riemann surface in the category of
orbifolds. The parameter $q$ is a local holomorphic coordinate about the new
closed point $\infty$. One works with $\Mbar_{1,1}$ in the obvious way. For
example, a line bundle on $\Mbar_{1,1}$ consists of a line bundle on each of the
orbifolds $\M_{1,1}$ and $C_2\bbs\D$, together with an isomorphism of their
pullbacks to $C_2\bbs\D^\ast$. Sections of a line bundle over $\Mbar_{1,1}$
consist of sections of the line bundle over the charts $\M_{1,1}$ and
$C_2\bbs\D$ that agree on their pullbacks to $C_2\bbs\D^\ast$.

The coarse moduli space associated to $\Mbar_{1,1}$ is
$$
\overline{M}_{1,1} := M_{1,1}\cup_{\D^\ast} \D
$$
where $\D$ is the $q$-disk.

The orbifold $\Mbar_{1,1}$ can also be expressed as a finite quotient of the
Riemann sphere. This gives an algebraic description of $\Mbar_{1,1}$.

\begin{exercise}
Show that the orbifold isomorphism of Proposition~\ref{prop:quot} extends to
an orbifold isomorphism
$$
\Mbar_{1,1} \cong (C_2\times S_3)\bbs \P^1
$$
where $C_2$ acts trivially on $\P^1$ and $S_3$ acts on $\P^1$ by permuting
$\{0,1,\infty\}$. Show that the quotient mapping $\P^1 \to \Mbar_{1,1}$ is
ramified at $\{0,1,\infty\}$ and that it is locally 2:1 about each of these
points. Deduce that the orbifold Euler characteristic of $\Mbar_{1,1}$ is
$5/12$.
\end{exercise}

The line bundles $\L_k$ extend naturally to $\Mbar_{1,1}$.

\begin{proposition}
The orbifold line bundle $\L_k \to \M_{1,1}$ extend  naturally to a holomorphic
line bundles $\Lbar_k \to \Mbar_{1,1}$.
\end{proposition}

\begin{proof}
Define $\Lbar_k$ to be the line bundle over $\Mbar_{1,1}$ whose restriction to
$\M_{1,1}$ is $\L_k$, whose restriction to $C_2\bbs\D$ is the quotient of the
trivial bundle $\C\times \D \to \D$ by the action $\pm 1 : (z,q)\mapsto ((\pm
1)^k z, q)$. The isomorphism
$$
C_2\bbs(\C\times \D^\ast) \cong p^\ast \L_k \to q^\ast \Lbar_k \cong
C_2 \bbs(\C\times \D^\ast)
$$
is the identity.
\end{proof}

\begin{exercise}
Prove that $\Lbar_k$ is isomorphic to $\Lbar_1^{\otimes k}$.
\end{exercise}

\subsection{Modular forms}
\label{sec:modular_forms}

A holomorphic (resp.\ meromorphic) modular  function of weight $k \in \N$ is a
holomorphic (resp.\ meromorphic) function $f : \h \to \C$ that satisfies
$$
f(\gamma\tau) = (c\tau + d)^k f(\tau)
$$
for all
$$
\gamma = \begin{pmatrix} a & b \cr c & d \end{pmatrix} \in \SL_2(\Z).
$$
Since $-I \in \SL_2(\Z)$, each modular function of weight $k$ satisfies $f(\tau)
= (-1)^kf(\tau)$, from which it follows that all modular functions of odd weight
vanish. As we have seen in Exercise~\ref{ex:L}, holomorphic (resp.\ meromorphic)
modular functions of weight $k$ are precisely the holomorphic (resp.\
meromorphic) sections of the orbifold line bundle $\L_k$ over $\M_{1,1}$.

\begin{example}[Eisenstein Series]
Fix an integer $k>2$. For a lattice $\Lambda$ in $\C$ define $S_{k}$ by the
absolutely convergent series
$$
S_{k}(\Lambda) = \sum_{\substack{\lambda \in \Lambda\cr \lambda \neq 0}}
\frac{1}{\lambda^{k}}.
$$
Observe that when $u\in \C^\ast$
\begin{equation}
\label{eqn:fun_eqn}
S_k(u\Lambda) = u^{-k}S_k(\Lambda)
\end{equation}
Since $\Lambda = -\Lambda$, this implies that $S_k$ is identically zero when
$k$ is odd.

Recall that for $\tau\in\h$, $\Lambda_\tau = \Z\oplus \Z\tau$. For $\tau \in
\h$, define
$$
G_k(\tau) := S_k(\Lambda_\tau).
$$
This is holomorphic on $\h$. Since $\Lambda_{\gamma\tau} = (c\tau +
d)^{-1}\Lambda_\tau$ for all
$$
\gamma = \begin{pmatrix}a & b \cr c & d \end{pmatrix} \in \SL_2(\Z),
$$
the identity (\ref{eqn:fun_eqn}) becomes
$G_k(\gamma\tau) = (c\tau + d)^k G_k(\tau)$. That is, $G_k$ is a holomorphic
modular function of weight $k$.  For more details, see any standard book on
modular forms, such as \cite[pp.~81--84]{serre}.
\end{example}

Since {\small $\begin{pmatrix} 1 & 1 \cr 0 & 1 \end{pmatrix}$} $\in \SL_2(\Z)$
each modular function satisfies $f(\tau + 1) = f(\tau)$ and thus has a Fourier
expansion (its ``$q$-expansion''):
$$
f(\tau) = \sum_{-\infty}^\infty a_n q^n, \quad q = e^{2\pi i \tau}.
$$

The $q$-expansion of $G_{2k}$ is
\begin{equation}
\label{eqn:q_exp}
G_{2k}(\tau) = 2\zeta(2k)
+2\frac{(2\pi i)^{2k}}{(2k-1)!}\sum_{n=1}^\infty\sigma_{2k-1}(n)q^n,
\end{equation}
where $\zeta$ denotes the Riemann zeta function and $\sigma_k(n) := \sum_{d|n}
d^k$. Details can be found in \cite[p.~92]{serre}.

\begin{definition}
Suppose that $k\in \N$. A {\em modular form of weight $k$} is a holomorphic
modular function that is holomorphic at $q=0$. That is, the coefficients $a_n$
of its $q$-expansion vanish when $n < 0$. A {\em cusp form} is a modular form
whose $q$-expansion vanishes at $q=0$. A {\em meromorphic modular form} of
weight $k$ is a meromorphic modular function of weight $k$ whose $q$-expansion
is meromorphic on the $q$-disk.
\end{definition}

\begin{exercise}
Show that the holomorphic modular forms of weight $k$ correspond to holomorphic
sections of $\Lbar_k \to \Mbar_{1,1}$. Show that cusp forms correspond to those
sections that vanish at the point $\infty$, the origin of the $q$-disk.
\end{exercise}

Equation (\ref{eqn:q_exp}) implies that each $G_{2k}$ is a modular form of
weight $2k$ when $k\ge 2$. Since
$$
\big(60\cdot 2 \cdot \zeta(4)\big)^3 = \frac{64}{27}\,\pi^{12}
= 27\big(140\cdot 2 \cdot \zeta(6)\big)^2,
$$
the {\em Ramanujan tau function}
$$
\Delta(\tau) := g_2(\tau)^3 - 27 g_3(\tau)^2
$$
is a cusp form of weight $12$, where $g_2(\tau) := 60\,G_4(\tau)$ and $g_3(\tau)
:= 140\,G_6(\tau)$. It has $q$-expansion \cite[p.~95]{serre}
$$
\Delta = (2\pi)^{12} q\prod_{n=1}^\infty(1-q^n)^{24}.
$$
The function $\Delta$ has no zeros in $\h$ \cite[p,~88]{serre} and a simple zero
at $q=0$.\footnote{A geometric proof is given in Section~\ref{sec:e_bar}. Cf.
Corollary~\ref{cor:tau_zerofree}.}

For a line bundle $L$ over $\Mbar_{1,1}$ and $d\in\Z$, define $L(d\infty) =
L\otimes \O_{\Mbar_{1,1}}(d\infty)$.\footnote{Sections of the orbifold sheaf
$\O_{\Mbar_{1,1}}(d\infty)$ are meromorphic functions on $\Mbar_{1,1}$ that are
holomorphic on $\M_{1,1}$ and whose Fourier expansion in the $q$-disk has a pole
of order $\le d$. Cf.\ Section~\ref{sec:picard}.}

\begin{exercise}
\label{ex:relation}
Show that $\Lbar_{12}\cong \O_{\Mbar_{1,1}}(\infty)$.
\end{exercise}

\begin{proposition}
\label{prop:log_canonical}
The log canonical bundle $\Omega^1_{\Mbar_{1,1}}(\infty)$ of $\Mbar_{1,1}$ is
isomorphic to $\Lbar_2$.
\end{proposition}

\begin{proof}
Since $q = \exp(2\pi i \tau)$, we have that $2 \pi i d\tau = dq/q$. That is,
$d\tau$ is a trivialization of the pullback of the log canonical bundle of
$\Mbar_{1,1}$ to the $q$-disk $\D$. On the other hand, since
$$
d\bigg(\frac{a\tau+b}{c\tau+d}\bigg) = \frac{d\tau}{(c\tau + d)^2},
$$
a meromorphic form $f(\tau)d\tau$ on $\h$ descends to a meromorphic section of
$\Omega^1_{\M_{1,1}}$ if and only if $f(\tau)$ is a meromorphic modular function
of weight 2. In particular,
$$
\w = \frac{G_4(\tau)}{G_2(\tau)} d\tau
$$
is such a meromorphic form. Since $G_4$ and $G_2$ are both non-zero at infinity
(Cf.\ equation~(\ref{eqn:q_exp})), the restriction of $\w$ to a neighbourhood of
$\infty$ is a nowhere vanishing holomorphic multiple of $dq/q$. From this it
follows that $f(\tau)$ is a meromorphic section of $\Lbar_2$ if and only if
$f(\tau)d\tau$ is a meromorphic section of $\Omega_{\Mbar_{1,1}}(\infty)$. The
result follows.
\end{proof}

Since $g_2$ does not vanish at $q=0$, the modular function
$$
j(\tau) := 1728\, g_2(\tau)^3/\Delta(\tau)
$$
of weight 0 is holomorphic on $\h$ and has a simple pole at $q=0$. It has
$q$-expansion
$$
j(\tau) = \frac{1}{q} + 744 + \sum_{n=1}^\infty c_n q^n,
$$
where each $c_n \in \Z$.

\begin{exercise}
Show that $j$ may be viewed as a holomorphic function $j: \Mbar_{1,1} \to \P^1$.
Show that its restriction to $\M_{1,1}$ induces a biholomorphism $M_{1,1}\to
\C$:
$$
\xymatrix@R1pc{
&  \ar[dl] \M_{1,1} \ar[dd]^j\ar[r] & \Mbar_{1,1}\ar[dd]^j \cr
M_{1,1} \ar[dr]_\approx \cr
& \C \ar[r] & \P^1
}
$$
Note that $j$ can be used to define a local parameter about each point of
$M_{1,1}$ except the points $[i]$ and $[\rho]$ as the map $\h \to M_{1,1}$ is
ramified above $[i]$ and $[\rho]$.
\end{exercise}

\begin{exercise}
Show that if $X \to T$ is a family of smooth elliptic curves over a {\em
compact} Riemann surface $T$, then the period mapping $T \to M_{1,1}$ is
constant. Deduce that every such family is isotrivial.
\end{exercise}

Denote the space of holomorphic modular forms of weight $k$ by $M_k$. These form
an evenly graded ring $M_\dot$ with respect to multiplication of functions. For
each even $k$, evaluation at $q=0$ defines a linear surjection $M_k \to \C$
whose kernel is $M_k^o$, the space of cusp forms of weight $k$. Since $G_{2k}$
does not vanish at $q=0$
$$
M_{2k} = M_{2k}^o \oplus \C G_{2k}
$$
for each $k>1$.

The following basic result is proved in \cite[p.~89]{serre}. It can be deduced
from the results of Exercise~\ref{ex:mod_forms}.

\begin{proposition}
\label{prop:dimensions}
The graded ring $M_\dot$ is generated freely by $G_4$ and $G_6$. The subspace of
cusp forms $M_\dot^o$ is the ideal of $M_\dot$ generated by $\Delta$. In
addition
$$
\dim M_{2k} = 1 + \dim M_{2k}^o =
\begin{cases}
\lfloor k/6\rfloor & k = 1 \bmod 6,\ k \ge 0;\cr 
1+\lfloor k/6\rfloor & k \neq 1 \bmod 6,\ k \ge 0. 
\end{cases}
$$
\qed
\end{proposition}

\subsection{Level structures*}
\label{sec:level}

In the first Chapter, we used a framing of $H_1(E;\Z)$ to rigidify and then
solve the moduli problem for elliptic curves. Since $H_1(E;\Z)$ has an infinite
number of framings, the moduli space of framed elliptic curves is an infinite,
and therefore transcendental, covering of $\M_{1,1}$. To construct finite
coverings of $\M_{1,1}$, we consider the moduli space of elliptic curves plus a
framings of $H_1(E;\Z/m\Z)$, where $m\ge 1$. 

\begin{definition}
Suppose that $m$ is a positive integer. A {\em level $m$ structure} on an
elliptic curve is a basis $\aa,\bb$ of $H_1(E;\Z/m\Z)$, where the mod $m$
intersection number $\aa\cdot \bb$ is $1$.
\end{definition}

\begin{exercise}
Show that the homomorphism $\SL_2(\Z) \to \SL_2(\Z/m\Z)$ that takes a matrix to
its reduction mod $m$ is surjective. Its kernel is called the {\em level $m$
subgroup of $\SL_2(\Z)$} and will be denoted by $\SL_2(\Z)[m]$. Show that the
set of isomorphism classes of elliptic curves with a level $m$ structure is the
quotient $\SL_2(\Z)[m]\bs \h$. Show that this is a Riemann surface with
fundamental group $\SL_2(\Z)[m]$ for all $m\ge 3$. (Cf.\
Exercise~\ref{ex:level_tf}.)
\end{exercise}

Denote the orbifold Riemann surface $\SL_2(\Z)[m]\bbs \h$ by $\M_{1,1}[m]$. It
is called the moduli space of elliptic curves with a level $m$ structure. Points
of the corresponding coarse moduli space $M_{1,1}[m] := \SL_2(\Z)[m]\bbs \h$ are
isomorphism classes of elliptic curves with a level $m$ structure.  Since
$\SL_2(\Z)[m]$ is torsion free for all $m\ge 3$, $\M_{1,1}[m] = M_{1,1}[m]$ for
all $m\ge 3$. The group $\SL_2(\Z/m\Z)$ acts on $\M_{1,1}[m]$ and $\M_{1,1}$ is
the orbifold quotient $\SL_2(\Z/m\Z)\bbs \M_{1,1}[m]$. The projection 
$\M_{1,1}[m]\to \M_{1,1}$ takes the isomorphism class of $(X,P;\aa,\bb)$ to the
isomorphism class of $(X,P)$. It has orbifold degree equal to the order of
$\SL_2(\Z/m\Z)$.

\begin{exercise}
Use the Chinese Remainder Theorem to show that
$$
\SL_2(\Z/m\Z) \cong \prod_{p} \SL_2(\Z/p^{\nu_p}),
$$
where $p$ ranges over all prime numbers and where $\nu_p := \ord_p(m)$. Show
that $I+pA \in \GL_2(\Z/p\Z)$ for all $A \in \gl_2(\Z/p^{n-1}\Z)$.\footnote{Here
$\gl_n(R)$ denotes the set of $n\times n$ matrices over $R$.} Deduce that there
is an exact sequence
$$
1 \to I + p\gl_2(\Z/p^{n-1}\Z) \to \GL_2(\Z/p^n\Z) \to \GL_2(\F_p) \to 1.
$$
Use this to show that, for all $n\ge 1$,
$$
|\GL_2(\Z/p^n\Z)| = p^{4n-3}(p^2-1)(p-1)
$$
and that
$$
|\SL_2(\Z/p^n\Z)| = |\GL_2(\Z/p^n\Z)|/|(\Z/p^n\Z)^\times| = p^{3n}(1-p^{-2}).
$$
Deduce that
$$
|\SL_2(\Z/m\Z)| = m^3 \prod_{p|m}\Big(1-\frac{1}{p^2}\Big).
$$
\end{exercise}

The moduli space $\M_{1,1}[m]$ can be compactified by adding a finite number of
points, called {\em cusps}, as we now explain. (Cf.\
Exercise~\ref{ex:completion}.)

The boundary of the upper half place is $\R\cup\{\infty\}$, which is usefully
regarded as the real projective line $\P^1(\R)$. The $\SL_2(\R)$-action on the
upper half plane extends to its boundary $\P^1(\R)$; it acts by fractional
linear transformations.

\begin{exercise}
Show that the $\SL_2(\Z)$-orbit of $\infty \in \P^1(\R)$ is the subset
$\P^1(\Q)$ of $\P^1(\R)$. Deduce that $\SL_2(\Z)[m]\bs \P^1(\Q)$ is finite.
\end{exercise}

Let $U_\infty$ denote the subset $\Im(\tau) > 1$ of $\h$. The stabilizer of
$U_\infty$ in $\SL_2(\Z)$ is the isotropy group of $\infty$, which is
$$
\G_\infty := \{\pm 1\}\times \PZ.
$$
The stabilizer of $U_\infty$ in $\SL_2(\Z)[m]$ is $\G_\infty[m] := \G_\infty
\cap \SL_2(\Z)[m]$. For each $x \in \P^1(\Q)$, choose $\gamma\in \SL_2(\Z)$ such
that $x = \gamma \infty$. Set
$$
U_x = \gamma U_\infty.
$$
Its stabilizer in $\SL_2(\Z)[m]$ is $\G_x[m] := \gamma \G_\infty[m]
\gamma^{-1}$. Both $U_x$ and $\G_x[m]$ depend only on $x$ and not on the choice
of $\gamma \in \SL_2(\Z)$. Further $\gamma$ induces a biholomorphism
$$
\G_\infty[m]\bs U_\infty \overset{\simeq}{\longrightarrow} \G_x[m]\bs U_x.
$$

For each $x\in \P^1(\Q)$, the quotient mapping (in the category of Riemann
surfaces)
$$
\G_x[m]\bs U_x \cong \G_\infty[m]\bs U_\infty \to \G_\infty\bs U_\infty
\cong \D_R^\ast
$$
has degree $m$ where $R = \exp(-2\pi)$. It follows that $\G_x[m]\bs U_x$ is a
punctured disk with coordinate the $m$th root $\exp(2\pi i \tau/m)$ of $q$.

Fix $m\ge 3$. The inclusion $U_x \to \h$ induces an inclusion
$$
\G_x[m]\bs U_x \to \M_{1,1}[m]
$$
of a punctured disk. With the identifications above, this map depends only on
the $\SL_2(\Z)[m]$-orbit $c$ of $x$. We shall therefore denote this punctured
disk by $V_c$.

\begin{definition}
For each $m\ge 3$, define $\Mbar_{1,1}[m]$ to be the Riemann surface whose
underlying set is
$$
\SL_2(\Z)[m]\bs\big(\h\cup \P^1(\Q)\big).
$$
As a Riemann surface, it is obtained from $\M_{1,1}[m]$ by attaching one disk
$\D_S$ of radius $S=\sqrt[m]{R}$ for each $c\in D_m := \SL_2(\Z)[m]\bs \P^1(\Q)$
by identifying the punctured disk $\D_S^\ast$ with the punctured disk $V_c$ in
$\M_{1,1}[m]$. Elements of $C_m := \SL_2(\Z)[m]$ are called {\em
cusps}.\footnote{This terminology is confusing as each $c\in D_m$ is a smooth
point of $\Mbar_{1,1}[m]$. Elements of $D_m$ are not cusps in the sense of
singularity theory, but they are related to cusp forms, which are sections of
powers of $\Lbar$ that vanish on $D_m$.}
\end{definition}

\begin{exercise}
Suppose that $m\ge 3$. Show that the Riemann surface $\Mbar_{1,1}[m]$ is
compact, that the action of $\SL_2(\Z/m\Z)$ on $\M_{1,1}[m]$ extends to
$\Mbar_{1,1}[m]$, and that the isomorphism $\M_{1,1} \cong \SL_2(\Z/m\Z)\bbs
\M_{1,1}[m]$ extends to an orbifold isomorphism $\Mbar_{1,1} \cong
\SL_2(\Z/m\Z)\bbs \Mbar_{1,1}[m]$.
\end{exercise}

When $m=2$, one attaches one copy of $D_2\bbs\D_{\sqrt[m]{R}}$ to $\M_{1,1}[2]$
for each cusp $c\in \{0,1,\infty\}=\SL_2(\Z)[2]\bs \P^1(\Q)$ to obtain
$\Mbar_{1,1}[2]$.

\begin{exercise}
Show that $\Mbar_{1,1}[2]$ is isomorphic to the quotient of $\P^1$ by the
trivial $C_2$-action. (Cf.\ the proof of Prop.~\ref{prop:quot}.)
\end{exercise}

\begin{exercise}
Suppose that $m\ge 3$. Set $d_m = |SL_2(\Z/m\Z)|$ and $c_m = \# D_m$. Show that
$\chi(\M_{1,1}[m]) = d_m\chi(\M_{1,1}) = -d_m/12$ and that
$$
c_m = \frac{d_m}{2m} = \frac{m^2}{2}\prod_{p|m}\Big(1-\frac{1}{p^2}\Big).
$$
Use this to show that
$$
\chi(\Mbar_{1,1}[m]) = c_m + \chi(\M_{1,1}[m]) = 
\frac{m^2}{2}\Big(1-\frac{m}{6}\Big)\prod_{p|m}\Big(1-\frac{1}{p^2}\Big).
$$
Deduce that the genus $g_m$ of $\Mbar_{1,1}[m]$ is given by
$$
g_m = 1-\frac{m^2}{4}\Big(1-\frac{m}{6}\Big)\prod_{p|m}\Big(1-\frac{1}{p^2}\Big).
$$
Except when $m$ is small, $\Mbar_{1,1}[m]$ is not rational. For example, $g_3 =
g_4 = g_5 =0$, $g_7 = 3$, $g_8=5$, $g_{41} = 2\,451$, $g_{5^3} = 74\,376$.
\end{exercise}

Modular forms of weight $k$ for $\SL_2(\Z)[m]$ are simply holomorphic sections
of $\Lbar_k$ over $\Mbar_{1,1}[m]$; cusp forms are holomorphic sections of
$\Lbar_k$ that vanish at the cusps.

\begin{remark}
Moduli spaces of elliptic curves with a level are frequently used by  number
theorists. They typically work with more refined level structures, such as the
moduli space of elliptic curves $E$ together with an element order $m$ of
$H_1(E;\Z/m\Z)$, or of elliptic curves plus a cyclic subgroup of $H_1(E;\Z/m\Z)$
of order $m$.
\end{remark}

\section{Cubic Curves and the Universal Curve $\Ebar \to \Mbar_{1,1}$}
\label{sec:e_bar}

In this section we construct the extension of the universal curve $\E\to
\M_{1,1}$ to the universal stable elliptic curve $\Ebar \to \Mbar_{1,1}$.

\subsection{Plane cubics}

The description of an elliptic curve as the quotient of $\C$ by a lattice is
very transcendental. In algebraic geometry it is more natural to consider an
elliptic curve as a smooth plane cubic curve rather than as the quotient of $\C$
by a lattice.

\begin{exercise}
\label{ex:smooth_cubics}
Suppose that $f(x) \in \C[x]$ is a cubic polynomial. Show that the curve $C$ in
$\P^2$ defined by $y^2 = f(x)$ is smooth if and only if $f(x)$ has 3 distinct
roots. (To study this curve at infinity, use the homogenized version $y^2 z =
z^3f(x/z)$.) Show that the algebraic differential $dx/y$ is a non-zero
holomorphic differential on $C$.
\end{exercise}

\begin{proposition}
Every smooth plane cubic curve has genus 1.
\end{proposition}

\begin{proof}
This is an immediate consequence of the genus formula.\footnote{The genus
formula states that the genus of a smooth curve in $\P^2$ of degree $d$ is
$(d-1)(d-2)/2$.} It can also be proved directly as follows. Suppose that $C$ is
a smooth plane cubic curve. Consider the exact sequence
$$
0 \to TC \to T\P^2|_C \to N \to 0
$$
of vector bundles over $C$. Since $C$ is a cubic, its normal bundle $N$ is the
restriction of $\O_{\P^2}(3)$ to $C$. This has degree 9. The first Chern class
of the $T\P^2$ is the negative of the first Chern class of the canonical bundle
$K_{\P^2}$ of $\P^2$. Since the canonical bundle of $\P^2$ is $\O_{\P^2}(-3)$,
its restriction to $C$ has degree $-9$. By the standard formula,
$$
c_1(T\P^2|_C) = - c_1(K_{\P^2}|_C) = c_1(TC) + c_1(N) \in H^2(C;\Z).
$$
Since the degree of a line bundle $L$ on $C$ is $\int_C c_1(L)$, we have
$$
2-2g(C) = \deg TC = -\deg(K_{\P^2}|_C) - \deg N = 9-9 = 0.
$$
Thus $g(C) = 1$.
\end{proof}

\begin{exercise}
Show that $y^2 = x^3 - x$ and $y^2 = x^3 - 1$ are smooth elliptic curves (with
distinguished point $[0,1,0]$). Show that the first has an automorphism of order
4 and the second an automorphism of order 6. Deduce that they are isomorphic to
$\C/\Z[i]$ and $\C/\Z[\rho]$, respectively. Cf.\ Exercise~\ref{ex:autos}.
\end{exercise}

The discriminant of the polynomial
$$
f(x) = 4 x^3 - ax - b 
$$
is $16(a^3 - 27 b^2)$. For convenience, we will divide it by $16$ and instead
call $D(a,b) := a^3 - 27b^2$ the discriminant of $f(x)$. Every curve of the form
$$
y^2 = 4 x^3 - ax - b 
$$
is an elliptic curve, with distinguished point $[0,1,0]$. The converse is
also true.

\begin{proposition}
\label{prop:plane_cubics}
Every elliptic curve $(X,P)$ is isomorphic to a smooth plane cubic of the
form
$$
y^2 = 4 x^3 - ax - b,
$$
where $P\in X$ corresponds to $[0,1,0]\in\P^2$ and $D(a,b)\neq 0$. Moreover, the
elliptic curve $(y^2 = 4 x^3 - ax - b,[0,1,0])$ is isomorphic to
$$
(y^2 = 4 x^3 - Ax - B,[0,1,0])
$$
if and only if there exists $u\in \C^\ast$ such that $A = u^2 a$ and $B = u^3b$.
\end{proposition}

\begin{proof}
This is an exercise in the use of the Riemann-Roch formula. There is an
inclusion of vector spaces
$$
L(P)\subseteq L(2P) \subseteq L(3P) \subseteq L(4P) \subseteq L(5P)
\subseteq L(6P).
$$
The Riemann-Roch formula implies that, when $n\ge 1$,
$$
\ell(nP) := \dim L(nP) = n.
$$
Since $\C \subseteq L(P)$, $L(P)$ is spanned by the constant function 1. Since
$\ell(2P) = 2$, there exists a non-constant function $x : X \to \P^1$ that is
holomorphic away from $P$ and where $P$ is at worst a double pole. If the pole
had degree 1, then $x : X \to \P^1$ would have degree 1, and would therefore be
a biholomorphism. This is impossible as $g(X)=1$. Thus $x$ has a double pole at
$P$ and $x : X \to \P^1$ has degree 2. The Riemann-Hurwitz formula
(Exercise~\ref{ex:RH}) implies that $x$ has 4 critical values, one of which is
infinity. Let $c_1,c_2,c_3$ be the 3 critical values in $\C$. By adding a
constant to $x$ if necessary, we may assume that $c_1+c_2+c_3 = 0$. This
condition determines $x$ up to a constant multiple.

Since $\ell(3P) = 3$, there is a function $y : X \to \P^1$ whose only pole is
$P$ and which does not lie in $L(2P)$. The pole is therefore a triple pole.
Denote the deck transformation of the covering $x : X \to \P^1$ by $\sigma$.
Observe that $\sigma$ acts trivially on $L(2P)$, but that $\sigma^\ast y \neq y$
as $y$ has an odd order pole at $P$. By replacing $y$ by $y-\sigma^\ast y$ if
necessary, we may assume that $\sigma^\ast y = -y$. This condition determines
$y$ up to a constant. Thus $1,x,y$ is a basis of $L(3P)$.

Since $\ell(4P) = 4$, and since $1,x,y,x^2 \in L(4P)$ are linearly independent,
they comprise a basis. Likewise, $1,x,y,x^2,xy$ is a basis of $L(5P)$.

Since $L(6P)$ contains the 7 functions $1,x,y,x^2,xy,x^3,y^2$, and since
$\ell(6P)= 6$, it follows that they are linearly dependent. This linear
dependence can be written as the sum of a $\sigma$-invariant term and a
$\sigma$-anti-invariant term. Since the $\sigma$-anti-invariant functions $y,xy$
are in $L(5P)$, they are linearly independent. Consequently, the
$\sigma$-invariant basis elements must be linearly dependent. The coefficients
of $x^3$ and $y^2$ in this dependence are both non-zero, otherwise $x^3$ or $y^2
\in L(5P)$, which is a contradiction. We therefore have a relation of the form
$$
ey^2 = 4x^3 - u x^2 - a x - b
$$
where $e\neq 0$. Since the critical values $c_1,c_2,c_3$ are the roots of the
right-hand side, and since we chose $x$ so that their sum is 0, $u=0$. By
rescaling $y$, we may assume that $e = 1$. That is, $X$ has an equation of the
form
$$
y^2 = 4x^3 - ax - b
$$
in which $x(P) = \infty$. As remarked, $x$ is unique up to multiplication by a
constant. If we multiply it by $u^2$, we have to multiply $y$ by $u^3$ so that
$y^2-4x^3$ remains in $L(2P)$.

The uniqueness statement is easily proved and is left as an exercise for the
reader.
\end{proof}

It is useful to give a second proof that every elliptic curve is isomorphic to a
curve of the form $y^2 = 4x^3 - ax - b$. Recall from
Proposition~\ref{prop:torus} that every elliptic curve is isomorphic to a
1-dimensional complex torus $(\C/\Lambda,0)$. So it suffices to show that every
1-dimensional complex torus is isomorphic to a plane cubic.

Suppose that $\Lambda$ is a lattice in $\C$. The Weierstrass
$\wp_\Lambda$-function is defined by
$$
\wp_\Lambda(z) =
\frac{1}{z^2} + \sum_{\substack{\lambda\in\Lambda\cr \lambda \neq 0}}
\bigg(\frac{1}{(z-\lambda)^2} - \frac{1}{\lambda^2}\bigg).
$$
This converges to a meromorphic function on $\C$ that is periodic with respect
to $\Lambda$ and has a double pole at each lattice point and is holomorphic
elsewhere. Consequently, the induced holomorphic function $\wp_\Lambda :
\C/\Lambda \to \P^1$ has a unique double pole at $0\in \C/\Lambda$ and is 2:1.
Since $\wp_\Lambda(z) = \wp_\Lambda(-z)$, the automorphism of the map
$x : \C/\Lambda \to \P^1$ is $z\mapsto -z$.

For $\tau \in \h$, denote the Weierstrass $\wp$-function of the lattice
$\Lambda_\tau := \Z\oplus \Z\tau$ by $\wp_\tau$:
$$
\wp_\tau(z) := \wp_{\Lambda_\tau}(z).
$$

\begin{exercise}
Show that if $u\in \C^\ast$, then
$$
\wp_{u\Lambda}(uz) = u^{-2}\wp_\Lambda(z).
$$
Deduce that if
$$
\gamma = \begin{pmatrix}a & b \cr c & d\end{pmatrix}\in \SL_2(\Z)
$$
then
$$
\wp_\tau(\gamma\tau)(z/(c\tau+d)) = (c\tau+d)^2 \wp_\tau(z).
$$
\end{exercise}

\begin{proposition}
For all $\tau \in \h$
$$
(\wp_\tau')^2 = 4\wp_\tau^3 - g_2(\tau)\wp_\tau - g_3(\tau)
$$
where $\wp_\tau' := \partial \wp_\tau/\partial z$ and $g_2=60 G_4$ and $g_3 =
140 G_6$ are Eisenstein series of weights 4 and 6, respectively.
\end{proposition}

\begin{proof}
The identity
$$
\frac{1}{(z-\lambda)^2} = \frac{d}{dz}\bigg(\frac{1}{\lambda-z}\bigg) =
\frac{1}{\lambda^2}\sum_{m=0}^\infty (m+1)\frac{z^m}{\lambda^m}
$$
implies that
\begin{equation}
\label{eqn:wp_ident}
\wp_\tau(z) = \frac{1}{z^2} + \sum_{m=1}^\infty (2m+1)G_{2m+2}(\tau) z^{2m}.
\end{equation}
Differentiating, we obtain the expansion
$$
\wp_\tau'(z) =
-\frac{2}{z^3} + \sum_{m=1}^\infty 2m(2m+1)G_{2m+2}(\tau) z^{2m-1}.
$$
Since
$$
\wp_\tau(z)^3 \equiv \frac{1}{z^6} + 9G_4(\tau)\frac{1}{z^2} + 15 G_6(\tau)
\mod (z)
$$
and
$$
\wp'_\tau(z)^2 \equiv 4\frac{1}{z^6} - 24G_4(\tau)\frac{1}{z^2} - 80 G_6(\tau)
 \mod (z)
$$
it follows that
$$
4\wp_\tau(z)^3 - g_2(\tau)\wp_\tau(z) - g_3(\tau) - \wp_\tau'(z)^2 
\equiv 0 \mod (z).
$$
Since $\wp_\tau$ and $\wp_\tau'$ are holomorphic away from $\Lambda_\tau$, the
left hand side of the previous expression is a holomorphic function on
$\C/\Lambda_\tau$ that vanishes at $0$. It is therefore zero.
\end{proof}

\begin{proposition}
\label{prop:embedding}
For all $\tau\in \h$, the polynomial $y^2 = 4x^3 - g_2(\tau)x - g_3(\tau)$ has
non-vanishing discriminant and the holomorphic map
$$
[\wp_\tau,\wp_\tau',1] : \C/\Lambda_\tau \to \P^2
$$
imbeds $E_\tau$ in $\P^2$ as the smooth cubic $y^2 = 4x^3 - g_2(\tau)x -
g_3(\tau)$. Moreover, the rational differential $dx/y$ on $\P^2$ pulls back to
the holomorphic differential $dz$ on $\C/\Lambda_\tau$.
\end{proposition}

\begin{proof}
Set $E_\tau = \C/\Lambda_\tau$. Since $E_\tau$ has genus 1 and $\wp_\tau :
(E_\tau,0) \to (\P^1,0)$ has degree two, the Riemann-Hurwitz formula
(Exercise~\ref{ex:RH}) implies that $\wp_\tau$ has 4 distinct critical
points.\footnote{Since $\wp_\tau(z) = \wp_\tau(-z)$, these are $\infty$ and the
3 non-zero points of order 2 of $E_\tau$.} Consequently, $\wp_\tau$ has three
distinct critical values in $\C$. Since
$$
(\wp_\tau')^2 = 4\wp_\tau^3 - g_2(\tau)\wp_\tau - g_3(\tau),
$$
these are the three roots of the cubic $4x^3 - g_2(\tau)x - g_3(\tau)$. Since
they are distinct, its discriminant
$$
\Delta(\tau) = g_2(\tau)^3 - 27 g_3(\tau)^2
$$
is non-zero. By Exercise~\ref{ex:smooth_cubics}, this implies that
$[\wp_\tau,\wp_\tau',1]$ imbeds $E_\tau$ as a smooth cubic.

The last statement holds because $x = \wp_\tau$ and $y = \wp_\tau'$, so that
$$
\frac{dx}{y} = \frac{\wp_\tau' dz}{\wp_\tau'} = dz.
$$
\end{proof}

An immediate consequence of the proof is a topological/geometric proof that the
Ramanujan tau function has no zeros in $\h$. (Cf.\ \cite[p.~84]{serre}.)

\begin{corollary}
\label{cor:tau_zerofree}
The Ramanujan tau function $\Delta := g_2^3 - 27 g_3^2$ has no zeros in $\h$.
\end{corollary}

\subsection{Extending the universal curve}
\label{sec:extended_curve}

The description of an elliptic curve as a plane cubic curve allows us to  extend
explicitly the universal elliptic curve over $\M_{1,1}$ to $\Mbar_{1,1}$.

Consider the family
$$
E = \big\{([x,y,z],q) \in \P^2 \times \D :
zy^2 = 4x^3 - g_2(\tau)xz^2 - g_3(\tau)z^3,\ q = e^{2\pi i \tau} \big\}
$$
of cubic curves over the disk. This family has a natural $C_2$-action in which
the generator acts by taking $([x,y,z],q)$ to $([x,-y,z],q)$.

\begin{exercise}
Show that $E$ is a smooth surface in $\P^2 \times \D$. Show that the projection
$E \to \D$ is proper.
\end{exercise}

\begin{lemma}
The restriction of $E\to \D$ to the punctured disk $\D^\ast$ is the pullback of
the universal elliptic $\E\to \M_{1,1}$ along the natural mapping
$$
\D^\ast = \PZ\bs \h \to \M_{1,1}.
$$
which is equivariant with respect to the natural $C_2$-actions on $\E_h$ and
$E$.
\end{lemma}

\begin{proof}
Proposition~\ref{prop:embedding} implies that the mapping $\C\times \h \to E$
defined by
\begin{equation}
\label{eqn:embedding}
(z,\tau) \mapsto [\wp_\tau(z),\wp_\tau'(z),e^{2\pi i \tau}]
\end{equation}
induces a holomorphic mapping $q : \E_\h \to E$ such that the diagram
$$
\xymatrix@C=4pc{
\E_\h \ar[r]^q \ar[d] & E\ar[d] \cr
\h \ar[r]^{\tau \mapsto e^{2\pi i \tau}} & \D
}
$$
commutes and that it is an isomorphism on each fiber. The first assertion
follows.

The generator of $C_2$ acts on $\E_h$ by $(z,\tau) \mapsto (-z,\tau)$. It acts
on $E$ by $([x,y,z],q) \mapsto ([x,-y,z],q)$. The $C_2$ equivariance $q$ follows
as $\wp_\tau(-z) = \wp_\tau(z)$ and $\wp_\tau'(-z) = - \wp_\tau'(z)$.
\end{proof}

We can thus extend the universal curve $\E_\h \to \M_{1,1}$ to $\Mbar_{1,1}$ by
gluing it to a copy of $E$:
$$
\xymatrix@R=-0.5pc@C=3pc{
& \boxed{\begin{CD}\E_\h \cr @VVV \cr \h \end{CD}}
\ar@(ul,dl)[]_{\pm\PZ} \ar[dl]_p \ar[dr]^q \cr
\boxed{\begin{CD}\E_\h \cr @VVV \cr \h\end{CD}}
\ar@(ul,dl)[]_{\SL_2(\Z)} &&
\ar@(ul,dl)[]_{C_2}\boxed{\begin{CD} E_\D \cr @VVV \cr\D\end{CD}}
}
$$
The extended family $\E \to \Mbar_{1,1}$ has smooth total space. Its fiber
over $\infty$ is the nodal cubic:

\begin{proposition}
\label{prop:nodal_cubic}
The fiber $E_0$ of $\Ebar \to \Mbar_{1,1}$ over $q=0$ is isomorphic to the nodal
cubic
$$
Y^2 = \frac{4}{27} (3X+2)(3X-1)^2.
$$
\end{proposition}

\begin{proof}
Set $x = (\pi i)^2 X$ and $y = (\pi i)^3 Y$. In these coordinates the equation
of $E$ is
$$
Y^2 = 4 X^3 - \frac{g_2}{(\pi i)^4} X - \frac{g_3}{(\pi i)^6}.
$$
When $q=0$,
$$
g_2 = 60\, G_4|_{q=0} = \frac{2^2}{3}\,\pi^4 \text{ and }
g_3 = 140\, G_6|_{q=0} = \frac{2^3}{27}\,\pi^6
$$
so that the equation of $E_0$ is
\begin{equation}
\label{eqn:nodal_cubic}
Y^2 = 4X^3 - \frac{4}{3} X + \frac{8}{27} = \frac{4}{27} (3X+2)(3X-1)^2.
\end{equation}
\end{proof}

\begin{exercise}
\label{ex:E0}
Use the identity (\ref{eqn:wp_ident}) to show that when $q=0$
$$
(2\pi i)^{-2}\wp_0(z)
= \frac{1}{(2\pi i z)^2}
-\sum_{m = 1}^\infty \frac{B_{2m+2}}{(2m+2)(2m)!}(2\pi i z)^{2m}
$$
where $B_n$ denotes the $n$th Bernoulli number.\footnote{Bernoulli numbers are
defined by the power series $x/(e^x-1) = \sum_{n=0}^\infty B_n x^n/n!$. The
first few Bernoulli numbers are $B_0 = 1$, $B_1 = -1/2$, $B_2 = 1/6$. When $k\ge
1$, $B_{2k+1} = 0$. Bernoulli numbers are related to values of the Riemann zeta
function at positive even integers by $\zeta(2k) = - (2\pi i)^{2k}
B_{2k}/4(2k)!$} Deduce that when $q=0$, the mapping (\ref{eqn:embedding})
factors through the quotient mapping $\C \to \C/\Z \cong \C^\ast$ defined by $w=
\exp(2\pi i z)$. Show that the rational differential $2\pi i dx/y$ on $\P^2$
pulls back to $dw/w \in H^0(\P^1,\Omega^1([0]+[\infty]))$.

Differentiate the identity
$$
\frac{1}{2} \coth\big(u/2\big) =
\sum_{m=0}^\infty \frac{B_{2m}}{(2m)!}u^{2m-1}
$$
(that is obtained by manipulating the defining series for Bernoulli numbers) to
show that
$$
\frac{1/4}{\sinh^2(u/2)}
= \frac{1}{u^2} - 
\sum_{m=0}^\infty \frac{B_{2m+2}}{(2m+2)(2m)!}u^{2m+1}.
$$
Deduce that
$$
(2\pi i)^{-2} \wp_0(z) = \frac{1}{12} + \frac{1/4}{\sinh^2(\pi i z)}
= \frac{1}{12} + \frac{w}{(w-1)^2}.
$$
Now use the fact that $X/4 = x/(2\pi i)^2 = (2\pi i)^{-2} \wp_0$ and the
equation (\ref{eqn:nodal_cubic}) to show that
$$
X = \frac{1}{3} + \frac{4w}{(w-1)^2},\text{ and } Y = \frac{8w(w+1)}{(w-1)^3}.
$$
Finally show that the map $\P^1\to \P^2$ defined by $w\mapsto [X(w),Y(w),1]$
maps $\P^1$ onto $E_0$. Show that it takes the identity $1\in \C^\ast$ to the
identity $[0,1,0]$ of $E_0$ and that $0$ and $\infty$ are both mapped to the
double point $[1/3,0,1]$ of $E_0$; show that the map is otherwise injective.
\end{exercise}

\begin{remark}
The nodal cubic $(E_0,[0,1,0])$ is an example of a {\em stable} pointed curve.
In general, a stable pointed curve is a pointed compact, connected, complex
analytic (or algebraic) curve all of whose singularities are nodes (i.e., 
analytically isomorphic to $xy=0$) and whose automorphism group (as a pointed
curve) is finite. The marked point is required to be distinct from the nodes.
\end{remark}

\begin{exercise}
Prove that all singular stable 1-pointed genus 1 curves are isomorphic to $E_0$.
\end{exercise}

\subsection{Families of stable elliptic curves}

The extension of the universal curve to $\Mbar_{1,1}$ allows us to study the
period mappings of algebraic families of smooth elliptic curves.

\begin{exercise}
Suppose that $F$ is a non-empty finite subset of a compact Riemann surface $T$.
Let $D$ be the divisor $\sum_{P\in F} [P]$. Show that there exists a positive
integer $n$ such that the linear system $H^0(T,\O(nD))$ embeds $T$ into
projective space. Deduce that $T-F$ is an affine complex algebraic curve.
\end{exercise}

Suppose that $T$ is a compact Riemann surface and that $F$ is a (possibly empty)
finite subset of $T$. Set $T' = T-F$. The following result establishes the
Stable Reduction Theorem (cf.\ \cite[p.~118]{harris-morrison}) in the special
case of elliptic curves.

\begin{theorem}[Stable reduction for families of elliptic curves]
\label{thm:stable_redn}
If $X \to T'$ is a family of smooth elliptic curves over $T'$, then
\begin{enumerate}

\item the coarse period mapping $T' \to M_{1,1}$ extends to a holomorphic
mapping $T \to \coarseM_{1,1}$;

\item after passing to a finite covering $S \to T$ unramified over $T'$, the
period mapping $\Phi : T'\to \M_{1,1}$ extends to a morphism $\Phitilde : S \to
\Mbar_{1,1}$:
$$
\xymatrix@R=1pc@C=1pc{
& S\ar[dr]\ar[rr]^\Phitilde && \Mbar_{1,1} \cr
S'\ar[ur]\ar[dr]_(0.435){\text{\rm unramified}} & & T \cr
& T' \ar[ur]\ar[rr]^\Phi && \M_{1,1}\ar[uu]
}
$$

\end{enumerate}
In other words, after passing to a finite unramified covering of $T'$, the
family $X$ extends to a family of stable elliptic curves.
\end{theorem}

This says that \'etale locally on $T'$ every family of elliptic curves can be
extended to a family of elliptic curves over a compact curve $S$, each of whose
fibers is stable. Note, however, that the total space of the extended family
$\Phitilde^\ast \Ebar$ over $S$ is typically singular over $S-S'$.

To prove Theorem~\ref{thm:stable_redn}, we first study the local version of
stable reduction.

\begin{proposition}
\label{prop:loc_stab_redn}
If $f : \D^\ast \to \M_{1,1}$ is a holomorphic mapping, then either
\begin{enumerate}

\item the image of $f_\ast : \pi_1(\D^\ast) \to \SL_2(\Z)$ is finite. In this
case, there is a finite covering $p : \D^\ast \to \D^\ast$ and a holomorphic
mapping $\ftilde : \D \to \M_{1,1}$ whose restriction to $\D^\ast$ is $f\circ
p$; or

\item the image of $f_\ast : \pi_1(\D^\ast) \to \SL_2(\Z)$ is infinite. In this
case, there is a double covering $p : \D^\ast \to \D^\ast$ and a holomorphic
mapping $\ftilde : \D \to \Mbar_{1,1}$ whose restriction to $\D^\ast$ is $f\circ
p$.

\end{enumerate}
If the image of $f_\ast : \pi_1(\D^\ast) \to \SL_2(\Z)$ lies $\SL_2(\Z)[m]$ for
some $m\ge 3$, then $f$ extends to a holomorphic mapping $\ftilde : \D \to
\Mbar_{1,1}$ without passing to a finite covering $p$.
\end{proposition}

\begin{proof}
Denote the image of the positive generator of $\pi_1(\D^\ast)$ under $f_\ast :
\pi_1(\D^\ast) \to \SL_2(\Z)$ by $A$. Identify $(\coarseM_{1,1},\infty)$ with
$(\P^1,\infty)$ via the modular function $j$. Let $q : \h \to \D^\ast$ be the
universal covering $z \mapsto \exp(2\pi i z)$. Let $F :\h \to \h$ be a
$\pi_1$-equivariant lift of $f$:
$$
\xymatrix{
\h \ar[d]_q \ar[r]^F & \h \ar[d] \cr
\D^\ast \ar[r]^f & \M_{1,1}
}
$$
The Schwartz Lemma \cite{ahlfors} implies that $F$ is distance decreasing in the
Poincar\'e metric. This implies that the composite
$$
\xymatrix{\D^\ast \ar[r]^f & \M_{1,1} \ar[r]^j & \P^1}
$$
cannot have an essential singularity at the origin as we now explain. If it did,
the image of each angular sector of each subdisk about the origin of
$\D_\epsilon^\ast$ would be dense in $\P^1$. But this implies that the image of
every strip $\Im(z)\ge c$, $|\Re(z-z_o)| \le  \epsilon$ has dense image in
$\M_{1,1}$, which contradicts Schwartz's Lemma. It follows that $j\circ f$ has a
removable singularity.\footnote{A pole is a removable singularity of a map to
$\P^1$.} Denote its extension to $\D$ by $G : \D \to \P^1$. If $G(0) \in
M_{1,1}$, then there is a finite covering $p : \D \to \D$ and a holomorphic
mapping $\ftilde : \D \to \h$ that lifts $G$. In this case $A$ fixes
$\ftilde(0)$, and thus has finite order.

Suppose now that $G(0) = \infty$. By standard complex variables, one can choose
a holomorphic coordinate $w$ on $\D$ centered at the origin such that $G$ is
given by $q = w^n$ in a neighbourhood of the origin for some positive integer
$n$. By choosing $r > 0$ to be small enough, we may assume that $\D$ is the disk
$|w| < r$. But this implies that
$$
A = \pm \begin{pmatrix} 1 & n \cr 0 & 1\end{pmatrix}.
$$
If the diagonal entries of $A$ are 1, then $f : \D^\ast \to \M_{1,1}$ extends to
a holomorphic mapping $\D \to \Mbar_{1,1}$. If the diagonal entries of $A$ are
$-1$, the composition of $f$ with a double covering $p : \D^\ast \to \D^\ast$
extends to a holomorphic function $\ftilde : \D \to \Mbar_{1,1}$.

Finally, if $A \in \SL_2(\Z)[m]$ where $m\ge 3$, then $A$ cannot have finite
order (as $\SL_2(\Z)[m]$ is torsion free), and $A$ cannot be the negative of a
unipotent matrix:
$$
A \notin \begin{pmatrix}-1 & \Z \cr 0 & -1\end{pmatrix}
$$
It follows that one can take $p$ to be the identity when $A$ lies in a subgroup
of level $m\ge 3$.
\end{proof}

\begin{proof}[Proof of Theorem~\ref{thm:stable_redn}]
Suppose that $T' = T-F$, where $T$ is a compact Riemann surface and $F$ is a
finite subset. Suppose that $\Phi : T' \to \M_{1,1}$ is the period mapping of a
family $X \to T'$ of smooth elliptic curves. Fix an integer $m\ge 3$. The kernel
of the homomorphism
$$
\Phi_\ast : \pi_1(T') \to \SL_2(\Z/m\Z)
$$
is a finite index subgroup of $\pi_1(T')$. It determines a finite, unramified
covering $p : S' \to T'$. By standard arguments (cf.\
Exercise~\ref{ex:completion}), there is a compact Riemann surface $S$ and with a
finite subset $F_S$ such that $S'=S-F_S$ and a holomorphic mapping $S \to T$
whose restriction to $S'$ is $p$. The composite
$$
\xymatrix{
S' \ar[r]^p & T' \ar[r]^\Phi & \M_{1,1}
}
$$
is the period mapping of the family $p^\ast X \to S'$. The associated monodromy
representation $(\Phi\circ p)_\ast$ is the composite
$$
\pi_1(S') \to \pi_1(T') \to \SL_2(\Z).
$$
The image of $(\Phi\circ p)_\ast$ lies in $\SL_2(\Z)[m]$.

For each $P\in F_S$, choose a coordinate disk $U_P \cong \D$ centered at $P$
such that $U_P\cap F = \{P\}$. Since the image of a generator of
$\pi_1(U_P^\ast)$ under $(\Phi\circ p)_\ast$ lies in $\SL_2(\Z)[m]$ ($m\ge 3$),
it follows from Proposition~\ref{prop:loc_stab_redn} that the period mapping $S'
\to \M_{1,1}$ extends across $P$ and that the period mapping extends to a
holomorphic mapping
$$
\Phitilde : S \to \Mbar_{1,1}.
$$
The family $p^\ast X \to S'$ of smooth elliptic curves extends to the family
$\Phitilde^\ast \Ebar \to S$ of stable curves.
\end{proof}

\begin{example}
Suppose that $\sigma$ is a non-trivial automorphism of the elliptic curve
$(E,0)$. Let $d$ be the order of $\sigma$. Let $X \to \D^\ast$ be the isotrivial
family associated to $\sigma$ and the $d$-fold covering $p:\D^\ast \to \D^\ast$.
(See Exercise~\ref{ex:isotrivial1} for the construction.) It follows from
Example~\ref{ex:isotrivial2} that the period mapping $\D^\ast \to \M_{1,1}$ does
not extend to a mapping $\D \to \Mbar_{1,1}$, for if it did extend, the induced
mapping $\pi_1(\D^\ast) \to \SL_2(\Z)$ would be trivial. Since the pullback of
$X\to \D^\ast$ along the $d$-fold covering $p:\D^\ast \to \D^\ast$ is the
trivial family $E\times\D^\ast \to \D^\ast$, the period mapping of $p^\ast X$ is
the constant map with value $[E]$, which trivially extends to a mapping $\D \to
\M_{1,1}$.
\end{example}

\begin{exercise}
For $e\in \Z$ set
$$
X_e =
\big\{([x,y,z],t) \in \P^2\times \D^\ast : t^e zy^2 = (x^2 - tz)(x-z)\big\}.
$$
This is a family of elliptic curves over $\D^\ast$ with zero section $t\mapsto
[0,1,0]$. Show that the fiber of $X_e$ over $t\in \D^\ast$ is isomorphic to the
fiber of $X_0$ over $t$. Show that the monodromy representation $\pi_1(\D^\ast)
\to \SL_2(\Z)$ takes the positive generator of $\pi_1(\D^\ast)$ to a conjugate
of
$$
(-1)^e \begin{pmatrix}1 & 1 \cr 0 & 1\end{pmatrix}.
$$
(Hint: set $Y = (\sqrt{t})^e y$.) Deduce that the families $X_0$ and $X_1$ are
not isomorphic.\footnote{The families $X_0$ and $X_1$ are said to differ by a
{\em quadratic twist}.} Show that the period mapping $\D^\ast \to \M_{1,1}$
extends to $\D$ if and only if $e$ is even.
\end{exercise}

One consequence of the stable reduction theorem is that families of elliptic
curves over affine algebraic curves are either isotrivial (cf.\
Exercise~\ref{ex:isotrivial1} and Example~\ref{ex:isotrivial2}) or have ``large
monodromy''.

\begin{corollary}
Suppose that $X\to T'$ is a family of elliptic curves over a Riemann surface
$T'$. If the coarse period mapping $T' \to M_{1,1}$ is constant, then the
monodromy representation
$$
\phi : \pi_1(T') \to \SL_2(\Z)
$$
has finite image and the family is isotrivial.  If $T' = T-F$ where $F$ is a
finite subset of a compact Riemann surface $T$, and if the coarse period mapping
is non-constant, then the image of the monodromy representation $\phi$ has
finite index in $\SL_2(\Z)$.
\end{corollary}

\begin{proof}[Sketch of Proof]
Denote the period mapping of the family $X \to T'$ by $\Phi : T' \to \M_{1,1}$.
Fix a universal covering $Y \to T'$. If the coarse period mapping $T'\to
M_{1,1}$ is constant, then the framed period mapping $\Phitilde : Y \to \h$ is
constant. Let $\{\tau\}$ be the image of $\Phitilde$. Since $\Phitilde$ is
equivariant with respect to $\pi_1(T') \to \SL_2(\Z)$, this implies that the
image of $\phi : \pi_1(T') \to \SL_2(\Z)$ fixes $\tau$ and is therefore finite.
The pullback of the family $X \to T'$ to the covering $S'\to T'$ determined by
$\ker\phi$ is trivial. The family $X\to T'$ is a quotient of the trivial family
$E_\tau\times S'\to S'$ by $\pi_1(T')/\ker\phi$ and is thus isotrivial.

Now suppose that $T'=T-F$ where $F$ is a finite subset of a compact Riemann
surface $T$. Suppose also that the period mapping $T' \to M_{1,1}$ is
non-constant. Fix an integer $m\ge 3$. The  inverse image of $\SL_2(\Z)[m]$ in
$\pi_1(T)$ is a finite index normal subgroup of $\pi_1(T)$. It determines a
finite covering $S' \to T'$. This extends to a finite holomorphic mapping $S\to
T$, where $S$ is a compact Riemann surface that contains $S'$ as the complement
of a finite subset. The period mapping $T'\to \M_{1,1}$ lifts to a holomorphic
mapping $S' \to \M_{1,1}[m]$ to the level-$m$ moduli space.\footnote{See
Section~\ref{sec:level}.} It extends to a holomorphic mapping $S \to
\Mbar_{1,1}[m]$. Since $S$ is compact and the period mapping is non-constant, $S
\to \Mbar_{1,1}[m]$ is surjective. Exercise~\ref{ex:fte_index} (below) implies
that the image of $\pi_1(S') \to \pi_1(\Mbar_{1,1}[m]) = \SL_2(\Z)[m]$ has
finite index. Since the diagram
$$
\xymatrix{
\pi_1(S') \ar[r]\ar[d] & \pi_1(T') \ar[d] \cr
\SL_2(\Z)[m] \ar[r] & \SL_2(\Z)
}
$$
commutes, the image of $\pi_1(T') \to \SL_2(\Z)$ has finite index in
$\SL_2(\Z)$.
\end{proof}

\begin{exercise}
\label{ex:fte_index}
Suppose that $f: X\to Y$ is a non-constant mapping of compact Riemann surfaces.
Show that if $F_X$ and $F_Y$ are finite subsets of $X$ and $Y$, respectively
such that $f(F_X) \supseteq F_Y$, then the image of
$$
f_\ast : \pi_1(X - F_X,x) \to \pi_1(Y - F_Y,f(x))
$$
has finite index in $\pi_1(Y-F_Y,f(x))$. Hints: (1) first show that if $F_Z$ is
a discrete subset of a Riemann surface $Z$ then $Z-F_Z \hookrightarrow Z$
induces a surjection on fundamental groups; (2) reduce to the case where $f$ is
an unramified covering by enlarging $F_X$ and $F_Y$.
\end{exercise}

\subsection{The Hodge bundle*}

The {\em Hodge bundle} is defined to be the line bundle
$$
\pi_\ast \Omega^1_{\Ebar/\Mbar_{1,1}}(\log E_0)
$$
over $\Mbar_{1,1}$, where $\pi : \Ebar \to \Mbar_{1,1}$ is the universal curve
and $E_0$ is the fiber of $\Ebar$ over $q=0$.\footnote{If $X$ is a smooth
variety and $D$ is a normal crossings divisor in $X$, then $\Omega^1_X(\log D)$
is the $\O_X$-module that is generated locally by $du_1/u_1,\dots,du_r/u_r$ and
$du_{r+1},\dots,du_n$, where $D$ is defined locally by $u_1u_2\dots u_r = 0$
with respect to local holomorphic coordinates $(u_1,\dots,u_n)$. It is a locally
free $\O_X$-module of rank equal to $\dim X$. If $f : X \to T$ is a holomorphic
family over a smooth curve $T$ whose fiber $X_t$ over $t\in T$ is smooth when
$t$ is not in the finite subset $F$ of $T$ and where the fiber $X_t$ over each
$t\in F$ is reduced and has normal crossings, then $\Omega^1_{X/T}(\log D)$ is
defined to be the sheaf $\Omega^1_X(\log D)/f^\ast \Omega^1_T(F)$, where
$D=f^{-1}F$. It is a locally free $\O_X$-module of rank $\dim X - 1$.} It (and
its generalizations in higher genus) play an important role in the enumerative
geometry of algebraic curves and their moduli. In Section~\ref{sec:picard} we
show that the Picard groups of $\M_{1,1}$ and $\Mbar_{1,1}$ are both generated
by the Hodge bundle. In this section, we show that the Hodge bundle is
isomorphic to $\Lbar$.

The first step in proving that the Hodge bundle is $\Lbar$ is to show that its
restriction to $\M_{1,1}$ is $\L$. The restriction of the Hodge bundle to
$\M_{1,1}$ is the line bundle $\pi_\ast \Omega^1_{\E/\M_{1,1}}$, whose fiber
over $[E] \in \M_{1,1}$ is the space of holomorphic differentials
$H^0(E,\Omega_E^1)$ of $E$. That this is isomorphic to $\L$ follows from the
next result:

\begin{lemma}
\label{lem:hodge}
The set of isomorphism classes of triples $(X,P,\w)$, where $(X,P)$ is an
elliptic curve and $\w$ is a holomorphic differential on $X$, is isomorphic in
bijective correspondence with $\L$.
\end{lemma}

\begin{proof}
We begin by considering isomorphism classes of framed triples $(X,P,\w)$. That
is, isomorphism classes of 5-tuples $(X,P,\w;\aa,\bb)$, where $\aa,\bb$ is a
basis of $H_1(X;\Z)$ with $\aa\cdot\bb = 1$.

Since every framed elliptic curve is isomorphic to one of the form
$$
(\C/\Lambda_\tau,0;1,\tau),
$$
we need only consider isomorphism classes of 5-tuples
$$
(\C/\Lambda_\tau,0,\w;1,\tau).
$$
The differential $\w_\tau := dz$ is the unique holomorphic differential on
$\C/\Lambda_\tau$ such that $\int_\aa \w_\tau = 1$. There is therefore a
bijection
\begin{equation}
\label{eqn:corresp}
\C\times \h \to
\big\{\text{isomorphism classes of 5-tuples $(X,P,\w;\aa,\bb)$} \big\}
\end{equation}
defined by
$$
(u,\tau)\mapsto (\C/\Lambda_\tau,0,u\w_\tau;1,\tau).
$$
To complete the proof, we will show that the correspondence is
$\SL_2(\Z)$-equivariant.

The element
$$
\gamma = \begin{pmatrix} a & b \cr c & d \end{pmatrix}
$$
of $\SL_2(\Z)$ takes $(u,\tau)$ to $((c\tau+d)u,\gamma \tau)$ and takes the
framing $(1,\tau)$ of $\C/\Lambda_\tau$ to the framing $(c\tau+d,a\tau+b)$. The
isomorphism
$$
(\C/\Lambda_\tau,0;c\tau+d,a\tau+b) \to (\C/\Lambda_{\gamma\tau},0;1,\gamma\tau)
$$
is obtained by dividing by $c\tau + d$, which implies that this isomorphism
takes $\w_\tau$ to $(c\tau+d)^{-1}\w_{\gamma\tau}$. Since $u\w_\tau = (c\tau+d)u
\w_{\gamma\tau}$ the mapping (\ref{eqn:corresp}) is $\SL_2(\Z)$-equivariant.
\end{proof}

\begin{corollary}
The restriction of the Hodge bundle to $\M_{1,1}$ is isomorphic to $\L$.
\end{corollary}

\begin{exercise}
Show that the fiber of the Hodge bundle over $\Mbar_{1,1}$ over the moduli
points $[E_0]$ of the nodal
cubic is
$$
H^0(\P^1,\Omega_{\P^1}([0]+[\infty])) = \C\frac{dw}{w}.
$$
Here we are identifying $E_0$ with $\P^1$ (coordinate $w$) with $0$ and $\infty$
identified.
\end{exercise}

\begin{exercise}
Show that the rational differential $dx/y$ on $\P^2$ pulls back to a section of
the restriction of the Hodge bundle $\pi_\ast \Omega^1_{E/\D}(\log E_0)$ to the
$q$-disk. Deduce that it trivializes the Hodge bundle over the $q$-disk.
\end{exercise}

Proposition~\ref{prop:embedding} and Exercise~\ref{ex:E0} imply that the local
framing $dx/y$ of the Hodge bundle takes the value $\w_\tau \in
H^0(E_\tau,\Omega^1_{E_\tau})$ if $q = \exp(2\pi i \tau)$ and $(2\pi
i)^{-1}dw/w$ when $q=0$. In other words, the local framing of the Hodge bundle
about $q=0$ agrees with the local framing of $\Lbar$ about $q=0$ when restricted
to  the punctured $q$-disk when the Hodge bundle over $\M_{1,1}$ is identified
with $\L$. This implies that the Hodge bundle over $\Mbar_{1,1}$ is isomorphic
to $\Lbar$.

\begin{theorem}
\label{thm:hodge_bdle}
The Hodge bundle over $\Mbar_{1,1}$ is isomorphic to $\Lbar$.
\end{theorem} 

In the elliptic curve case, the Hodge bundle is also isomorphic to the conormal
bundle of the identity section of $\Ebar$. Denote the identity section of $\pi :
\Ebar \to \Mbar_{1,1}$ and its image by $Z$. The {\em relative cotangent bundle
of $\pi$} is defined to be the dual $\check{N}$ of the normal bundle $N$ of $Z$
in $\Ebar$.

\begin{proposition}
\label{prop:lambda_equals_psi}
The Hodge bundle is isomorphic to the relative cotangent bundle
$s^\ast \check{N}$ of the zero section.
\end{proposition}

\begin{proof}
Since the holomorphic tangent bundle of every smooth elliptic curve is
trivial, there is a natural isomorphism
$$
T_0 E \cong H^0(E,\Omega_E^1)^\ast
$$
for all smooth elliptic curves $E$. There is also a natural isomorphism
$$
T_1 E_0 \cong T_1 \C^\ast \cong H^0(\P^1,\Omega^1_{\P^1}([0]+[\infty]))
$$
for the nodal cubic. The result follows.
\end{proof}

\begin{remark}
In Hao Xu's talk, $\psi_1$ denotes the first Chern class of the relative
cotangent bundle of the universal elliptic curve $\Ebar \to \Mbar_{1,1}$ and
$\lambda_1$ denotes the first Chern class of the Hodge bundle.
Theorem~\ref{thm:hodge_bdle} implies that $\lambda_1$ is the class of $\Lbar$
and Proposition~\ref{prop:lambda_equals_psi} implies (in the case of elliptic
curves) that $\lambda_1 = \psi_1$. Xu denotes the class of the nodal cubic by
$
\xymatrix@-2.5pc{
\text{---}\ar@{-}[r] & \dot&*+[o][F][r]{\phantom{x}},
}
$
which, by Proposition~\ref{prop:nodal_cubic}, is the class of the boundary point
$\infty$. Exercise~\ref{ex:relation} states that $[\infty] = 12 \lambda_1$. In
Xu's notation, this reads:
$$
\psi_1 = \lambda_1 = \frac{1}{12}\bigg(
\xymatrix@-2.5pc{
\text{---}\ar@{-}[r] & \dot&*+[o][F][r]{\phantom{x}} \cr
}
\bigg)
$$
\end{remark}

\subsection{Natural metrics*}

The restriction of the Hodge bundle to $\M_{1,1}$ has a natural metric. This is
because there is a natural metric on the space of holomorphic 1-forms on an
elliptic curve $X$. Namely:
$$
\|\w\|^2 = \frac{i}{2}\int_X \w\wedge \overline{\w},
\qquad \w \in H^0(X,\Omega^1_X).
$$
In particular,
$$
\|w_\tau\|^2 = \int_{X_\tau} dz\wedge d\overline{z} = \Im(\tau).
$$
Since this metric is intrinsically defined, it follows that the metric
$$
\|(u,\tau)\|^2 = |u|^2 \Im(\tau)
$$
on the line bundle $\C\times \h \to \h$ is invariant under the action
(\ref{eqn:action}) of $\SL_2(\Z)$ and thus descends to a metric on $\L \to
\M_{1,1}$. (This is easy to check directly.) The $k$th power of this metric
$$
\|(u,\tau)\|^2 = |u|^2 \Im(\tau)^k
$$
defines a metric on $\L_k \to \M_{1,1}$.

These metrics do not extend to metrics on $\Lbar_k$. To see this, write $q = r
e^{i\theta}$, so that
$$
\Im(\tau) = \Im((\log q)/2\pi i) = -(\log r)/2\pi = -(\log|q|)/2\pi,
$$
which blows up as $|q| \to 0$. Nonetheless, this metric is still useful as it is
$L^1$ on the $q$ disk as $|\log r|$ is $L_1$ on the unit disk.

\begin{exercise}
Show that the metric on the tangent bundle $T \M_{1,1}$ induced by the
isomorphism $T\M_{1,1} \cong \L_{-2}$ equals the hyperbolic metric
$$
ds^2 = \Im(\tau)^{-2}d\tau d\overline{\tau}.
$$
Show that the punctured $q$-disk $\D^\ast_R$ has finite volume in this metric,
where $R=\exp(-2\pi)$. Deduce that $\M_{1,1}[m]$ has finite volume for all $m\ge
1$.
\end{exercise}

The metric on $\L_{k}$ can be used to define an inner product of two modular
forms $f$ and $g$ of weight $k$ of $\SL_2(\Z)[m]$ by integrating the
$\SL_2(\Z)[m]$-invariant function
$$
f(\tau) \overline{g(\tau)}\Im(\tau)^{k}
$$
over a fundamental domain of the action of $\SL_2(\Z)[m]$ on $\h$ with respect
to the invariant volume form. This defines a positive definite hermitian form on
the space of modular forms $H^0(\Mbar_{1,1}[m],\Lbar_{k})$ of weight $k$ of
$\SL_2(\Z)[m]$:
$$
(f,g) :=
\frac{i}{2}\int_{\M_{1,1}[m]} f(\tau) \overline{g(\tau)}\Im(\tau)^{k-2}
d\tau \wedge d\overline{\tau}.
$$
It is called the {\em Petersson inner product}.

\section{The Picard Groups of $\M_{1,1}$ and $\Mbar_{1,1}$}
\label{sec:picard}

In this section we compute the Picard groups of $\M_{1,1}$ and $\Mbar_{1,1}$.
This requires a detailed discussion of divisors and line bundles on orbifold
Riemann surfaces.

\subsection{Assumptions}

We consider only orbifolds that are locally of the form $\G\bbs X$ where $\G$
acts virtually freely on $X$. In particular, the isotropy group of $X$
$$
\G_X := \{g \in \G : g x = x \text{ for all } x \in X\}
$$
is finite. In addition, we will always assume that $\G_X$ is cyclic and central
in $\G$. These conditions are satisfied by $\M_{1,1}$, $\Mbar_{1,1}$ and the
universal curves over them.

The group $\G/\G_X$ acts effectively on $X$. The {\em reduced orbifold}
associated to $\G\bbs X$ is defined by
$$
(\G\bbs X)^\red := (\G/\G_X)\bbs X.
$$
We say that $\G\bbs X$ is {\em reduced} if $\G$ acts effectively on $X$. That
is, when $\G_X$ is trivial.

\begin{example}
The moduli space $\M_{1,1}$ is not reduced. The corresponding reduced orbifold
$\M_{1,1}^\red$ is $\PSL_2(\Z)\bbs \h$.
\end{example}

There are natural morphisms
$$
\G\bbs X \to (\G\bbs X)^\red \to \G\bs X
$$
which are induced by the obvious morphisms
$$
(X,\G) \to (X,\G/\G_X) \to (\G\bs X,\triv).
$$
If $\G\bbs X$ is an orbifold in the category of Riemann surfaces, then so  is
$(\G\bbs X)^\red$ and the natural morphisms above are both holomorphic.

We define the {\em degree} of each of the morphisms
$$
\G\bbs X \to (\G\bbs X)^\red\text{ and } \G\bbs X \to \G\bs X
$$
to be $|\G_X|$. Note that $(\G\bbs X)^\red \to \G\bs X$ has degree 1.

\subsection{Local theory}

Here we develop the theory for basic orbifolds. For simplicity, we consider
only the 1-dimensional case.

Suppose that $\G\bbs X$ is a Riemann surface in the category of orbifolds, where
$\G$ acts virtually freely on $X$. Denote by $[x]$ the $\G$-orbit of $x\in X$.
To this we can associate the order $|\G_x|$ of the isotropy group of $x$. This
depends only on the orbit $[x]$ and not on the choice of the representative $x$.

Define a {\em divisor} on $\G\bbs X$ to be a locally finite, formal linear
combination
$$
\sum_{[x]\in \G\bs X} \frac{n_x}{|\G_x|} [x]
$$
of points of $\G\bs X$, where each $n_x \in \Z$. Denote the group of divisors on
$\G\bbs X$ by $\Div(\G\bbs X)$.

\begin{remark}
Motivation for the definition of a divisor as an integral linear combinations of
the $[x]/|\G_x|$ comes from the discussion of orbifold Euler characteristic in
Paragraph~\ref{para:euler}.
\end{remark}

To each section of a holomorphic line bundle over $\G\bbs X$, we can associate a
divisor. A section $s$ of a holomorphic line bundle $\G\bbs L \to \G\bbs X$ is a
$\G$-equivariant holomorphic section $\stilde$ of $L\to X$. Define the {\em
order} $\nu_{[x]}(s)$ of $s$ at $[x]$ to be the order of $\stilde$ at $x\in X$.
This is well defined as $s$ is $\G$-equivariant. Define the {\em divisor} of a
non-zero section $s$ by
$$
\div(s) = |\G_X| \sum_{[x]\in \G\bs X} \frac{\nu_{[x]}(s)}{|\G_x|} [x]
\in \Div(\G\bbs X).
$$
The factor $|\G_X|$ is present as all non-zero sections $s$ are
$\G_X$-invariant, which means that such $s$ are pulled back from $(\G\bbs
X)^\red$.

\begin{exercise}
Suppose that $f : \G'\bbs X' \to \G\bbs X$ is a holomorphic mapping between
orbifolds. Show that a holomorphic line bundle $L \to \G'\bbs X'$ pulls back to
a holomorphic line bundle $f^\ast L \to \G\bbs X$ and that a section $s$ of $L$
pulls back to a section $f^\ast s$ of $f^\ast L$. Show that if $\G$ and $\G'$
act virtually freely on $X$ and $X'$, respectively, then there is a homomorphism
$$
f^\ast : \Div(\G'\bbs X') \to \Div(\G\bbs X)
$$
such that $f^\ast \div(s) = \div(f^\ast s)$.
\end{exercise}

Suppose that $D$ is divisor on $\G\bbs X$. Let $\pi : X \to \G\bs X$ be the
natural projection.  For each open subset $U$ of $\G\bs X$, define $\O_{\G\bbs
X}(D)$ to consist of the $\G$-invariant sections $\O_X(\pi^\ast D)(\pi^{-1}(U))$
of $\O_X(\pi^\ast D)$ over $\pi^{-1}(U)$. Then $\O_{\G\bbs X}(D)$ is an example
of a sheaf on $\G\bbs X$.

The group of divisors of the Riemann surface $\G\bs X$ consists of all formal
linear combinations
$$
\sum_{[x]\in \G\bs X} n_x [x]
$$
where each $n_x \in \Z$.

\begin{exercise}
\label{ex:pic=cl}
Show that the mapping
$$
\Div(\G\bs X) \to \Div\big((\G\bbs X)^\red\big) \to \Div(\G\bbs X)
$$
induced by the canonical quotient mappings $\G\bbs X \to (\G\bbs X)^\red \to
\G\bs X$ satisfy $[x] \mapsto [x]$. In particular, these mappings are injective.
\end{exercise}

\begin{exercise}
Show that if $Y$ is a Riemann surface, then every holomorphic mapping $\G\bbs X
\to Y$ factors through the quotient mapping $\pi : \G\bbs X \to \G\bs X$. In
particular, every meromorphic function $\G\bbs X \to \P^1$ is pulled back from a
meromorphic function $\G\bs X \to \P^1$.
\end{exercise}

The definitions of divisor class groups and Picard groups can be extended to
basic orbifolds.

\begin{definition}
A {\em principal divisor} on $\G\bbs X$ is the divisor of a non-zero meromorphic
function $f : \G\bbs X \to \P^1$. The {\em divisor class group} of $\G\bbs X$
is the group
$$
\Cl(\G\bbs X) := \Div(\G\bbs X)/\{\text{principal divisors}\}.
$$
The {\em Picard group} of $\G\bbs X$ is defined to be the group of isomorphism
classes of holomorphic line bundles over $\G\bbs X$, where the group operation
is tensor product of line bundles. Denote it by $\Pic(\G\bbs X)$.
\end{definition}

\begin{exercise}
Show that if $\G\bbs X$ is reduced (i.e., $\G_X$ is trivial), there is a well
defined group homomorphism
$$
\Pic(\G\bbs X) \to \Cl(\G\bbs X)
$$
that takes the isomorphism class of a holomorphic line bundle to the divisor
class of a non-zero meromorphic section.\footnote{For this you will need to show
that every orbifold line bundle $L \to \G\bbs X$ has a non-zero meromorphic
section. This can be proved by first noting that, since the action of $\G$ on
$X$ is virtually free and effective, $L \to \G\bbs X$ is the quotient of a line
bundle $M\to Y$ over a Riemann surface by a finite group $G$. One can then use
standard results about Riemann surfaces to show that such a line bundle has a
non-zero $G$-invariant meromorphic section. When $Y$ is compact, you can do this
using Riemann-Roch. When $Y$ is non-compact, you can use the fact that $Y$ is
Stein, so that $M$ is trivial.} Show that it is an isomorphism.
\end{exercise}

When $\G\bbs X$ is not reduced, there are line bundles that have no meromorphic
sections.

\begin{example}
The line bundle $\L_k \to \M_{1,1}$ has no non-zero meromorphic sections when
$k$ is odd.
\end{example}

Because of this, we compute $\Pic \M_{1,1}^\red$ before computing $\Pic
\M_{1,1}$.

\begin{proposition}
\label{prop:pic_m11}
There are natural isomorphisms
$$
\Pic \M_{1,1}^\red \cong \Cl(\M_{1,1}^\red) \cong \Z/6\Z.
$$
The Picard group is generated by the class of $\L_2$.
\end{proposition}

\begin{proof}
To compute $\Pic \M_{1,1}^\red$, it suffices, by Exercise~\ref{ex:pic=cl}, to
compute $\Cl(\M_{1,1}^\red)$. Since $M_{1,1} \cong \C$, it follows that $[\tau]$
is trivial in $\Cl(\M_{1,1})$ for all $\tau \in \h$. Consequently,
$$
\Cl(\M_{1,1}^\red) \cong 
\Big\{\frac{n_i}{2}[i] + \frac{n_\rho}{3}[\rho] :
n_i,\ n_\rho \in \Z\Big\}/\Z[i]\oplus \Z[\rho],
$$
which is isomorphic to $\Z/6\Z$.

To see that $\L_2$ generates $\Pic\M_{1,1}^\red$ we use the facts
\cite[p.~80]{serre}\footnote{A direct proof of these facts using results 
developed in these notes can be given.  See Exercise~\ref{ex:mod_forms}.}
\begin{equation}
\label{eqn:orderG}
\nu_i(G_4) = 0,\ \nu_\rho(G_4) = 1,\quad \nu_i(G_6) = 1,\ \nu_\rho(G_6) = 0.
\end{equation}
where $G_k$ denotes the Eisenstein series of weight $k$. Since $G_6/G_4$ is a
meromorphic section of $\L_2$, its divisor $[\rho]/3 - [i]/2$ generates
$\Cl(\M_{1,1}^\red)$. This implies that $\L_2$ generates $\Pic \M_{1,1}^\red$.
\end{proof}

To compute $\Pic \M_{1,1}$, we need to relate it to $\Pic \M_{1,1}^\red$.

\begin{exercise}
\label{ex:ses}
Show that if $X$ is a simply connected Riemann surface, then there is an exact
sequence
$$
0 \to \Pic\big((\G\bbs X)^\red\big) \to \Pic(\G_X\bbs X) \to \Char(\G_X) \to 1
$$
where $\Char(\G_X)$ denotes the group of characters $\chi : \G_X \to \C^\ast$.
\end{exercise}

\begin{theorem}
The group $\Pic \M_{1,1}$ is cyclic of order $12$. It is generated by the class
of $\L_1$.
\end{theorem}

\begin{proof}
Since  the square of $[\L] \in \Pic \M_{1,1}$ generates $\Pic \M_{1,1}^\red$,
which has order 6, it follows that $[\L]$ has order 12 in $\Pic \M_{1,1}$.
By Exercise \ref{ex:ses} the sequence
$$
0 \to \Pic\M_{1,1}^\red \to \Pic \M_{1,1} \to \Char(C_2) \to 0
$$
is exact. Since $[\L]$ maps to the non-trivial character $C_2 \to \C^\ast$, it
follows that $\Pic \M_{1,1}$ is generated by $[\L]$ and has order 12.
\end{proof}

\begin{remark}
These definitions in this section generalize easily to complex analytic
orbifolds of higher dimension.
\end{remark}

\subsection{The Picard group of $\Mbar_{1,1}$}

The constructions of the previous section generalize to all orbifold Riemann
surfaces. In this section we explain how to do this for $\Mbar_{1,1}$ and the
corresponding reduced orbifold $\Mbar_{1,1}^\red$, which we define below.

We shall view $\Mbar_{1,1}$ as the union of the basic orbifolds $\M_{1,1}$ and
$C_2\bbs\D_R$, where $R = e^{-2\pi}$, which ``intersect'' in the basic orbifold
$C_2\bbs \D_R^\ast$. In both cases, the $C_2$-action is trivial. Thus, to each
$P\in \overline{M}_{1,1}$, we can associate an ``automorphism group'' $\Aut(P)$,
which is well defined up to isomorphism. If $P = [x] \in \G\bbs X$, where
$(X,\G) = (\h,\SL_2(\Z))$ or $(\D_R,C_2)$, and $x$ is a lift of $P$ to $X$, then
$\Aut(P)$ is the isomorphism class of $\G_x$. For all but $[i],[\rho]\in
\M_{1,1}$, this isotropy group is isomorphic to $C_2$; $\Aut([i]) = \bmu_4$ and
$\Aut([\rho]) = \bmu_6$.

The orbifold $\Mbar_{1,1}^\red$ is obtained by gluing $\M_{1,1}^\red :=
\PSL_2(\Z)\bs \h$ to $\D$ along the orbifold $\D_R^\ast$. The only points in
$\M_{1,1}^\red$ with non-trivial automorphism groups are $[i]$ and $[\rho]$,
whose automorphism groups are cyclic of orders 2 and 3, respectively.

A divisor on $\Mbar_{1,1}$ is a finite sum
$$
\sum_{P\in \overline{M}_{1,1}} \frac{n_P}{|\Aut(P)|} P
$$
where each $n_P \in \Z$ for all $P \in \coarseM_{1,1}$. These form a group
$\Div(\Mbar_{1,1})$.

Since the size of the stabilizer of $P$ in $\PSL_2(\Z)$ is $|\Aut(P)|/2$,
divisors on $\Mbar_{1,1}^\red$ are finite sums
$$
\sum_{P\in \overline{M}_{1,1}} \frac{2 n_P}{|\Aut(P)|} P
$$
where $n_P \in \Z$ for all $P\in \coarseM_{1,1}$. These form a subgroup
$\Div(\Mbar_{1,1}^\red)$ of $\Div(\Mbar_{1,1})$.

The group of divisors on the Riemann surface
$\coarseM_{1,1}$ is the free abelian group generated by the $[x]\in
\coarseM_{1,1}$. The quotient mappings $\Mbar_{1,1} \to \Mbar_{1,1}^\red \to
\coarseM_{1,1}$ induce the inclusions
$$
\Div(\coarseM_{1,1}) \hookrightarrow \Div(\Mbar_{1,1}^\red)
\hookrightarrow \Div(\Mbar_{1,1})
$$
which take $P$ to $P$.

The divisor
$$
\div(s) \in \Div(\Mbar_{1,1})
$$
of a section $s$ of a holomorphic line bundle $L \to \Mbar_{1,1}$ is computed
locally on the two basic orbifold patches as in the previous section. A
principal divisor on $\Mbar_{1,1}$ is the divisor of a non-zero rational
function $f : \Mbar_{1,1} \to \P^1$. All such functions are pulled back from
rational functions $\coarseM_{1,1} \to \P^1$.

At this point, one can define the sheaves $\O_{\Mbar_{1,1}}(D)$ locally in the
two patches $\M_{1,1}$ and $C_2\bbs\D$.

\begin{exercise}
Show that if $D\in \Div(\Mbar_{1,1})$ and $L$ is a holomorphic line bundle over
$\Mbar_{1,1}$, then a meromorphic section $s$ of $L$ is a holomorphic section of
$L(D)$ if and only if $\div(s) + D \ge 0$. In particular, if $P\in
\coarseM_{1,1}$, then $\O_{\Mbar_{1,1}}(P)$ is the pullback of
$\O_{\coarseM_{1,1}}(P)$ to $\Mbar_{1,1}$. This proves that the ad hoc
definition of the twist $L(d\infty)$ of $L$ given in the discussion preceding
Exercise~\ref{ex:relation} is consistent with the definitions given in this
section.
\end{exercise}

As in the local case, not every line bundle over $\Mbar_{1,1}$ has a meromorphic
section. Because of this, we define the divisor class group only for
$\Mbar_{1,1}^\red$. 

Define the divisor class group of $\Mbar_{1,1}^\red$ to be
$$
\Cl(\Mbar_{1,1}^\red) := \Div(\Mbar_{1,1}^\red)/ \{\text{principal divisors}\}
$$

\begin{exercise}
Show that there is an exact sequence
$$
0 \to \Cl(\coarseM_{1,1}) \to \Cl(\Mbar_{1,1}^\red) \to\Cl(\M_{1,1}^\red)\to 0.
$$
Show that $\Cl(\coarseM_{1,1})$ is infinite cyclic and is generated by the class
of any point of $\coarseM_{1,1}$. Show directly that $\Cl(\M_{1,1}^\red)$ is
cyclic of order 6.  Deduce that $\Cl(\Mbar_{1,1}^\red)$ is infinite
cyclic.\footnote{It is not necessary to use Proposition~\ref{prop:pic_m11}.
Bypassing Prop.~\ref{prop:pic_m11} is desirable as one can then deduce
(\ref{eqn:orderG}).}
\end{exercise}

Define the Picard group of an orbifold Riemann surface $\X$ to be the group of
isomorphism classes of holomorphic line bundles over $\X$ with operation tensor
product. Holomorphic mappings between orbifolds induce mappings on their Picard
groups. In particular, we have natural pullback homomorphisms
$$
\Pic \coarseM_{1,1} \to \Pic \Mbar_{1,1}^\red \to \Pic \Mbar_{1,1}.
$$

\begin{proposition}
There is a natural isomorphism
$$
\Pic\Mbar_{1,1}^\red \overset{\simeq}{\longrightarrow} \Cl(\Mbar_{1,1}^\red).
$$
Both are isomorphic to $\Z$ and generated by the class of $\Lbar_2$. For all
$P\in \coarseM_{1,1}$,
$$
[\O_{\Mbar_{1,1}^\red}(P)] = 6[\Lbar_2] \in \Pic\Mbar_{1,1}^\red.
$$
\end{proposition}

\begin{proof}[Sketch of Proof]
We first construct a homomorphism
\begin{equation}
\label{eqn:homom}
\Pic \Mbar_{1,1}^\red \to \Cl(\Mbar_{1,1}^\red).
\end{equation}
To do this, we have to show that every holomorphic line bundle over
$\Mbar_{1,1}^\red$ has a non-zero meromorphic section. The homomorphism is
defined by taking the isomorphism class of a line bundle to the divisor class of
any non-zero meromorphic section.

To see that every holomorphic line bundle over $\Mbar_{1,1}^\red$ has a non-zero
meromorphic section, we use the fact that every line bundle $L \to
\Mbar_{1,1}^\red$ is the quotient of a holomorphic line bundle $N \to Y$ over a
Riemann surface $Y$ by the action of a finite group $G$, where $G$ acts
effectively on $Y$.\footnote{For example, one can take $Y$ to be the level $m$
moduli space $\Mbar_{1,1}[m]$ for any $m\ge 3$. It is constructed in
Section~\ref{sec:level}.}   The Riemann-Roch Theorem implies that $N \to Y$ has
a non-zero meromorphic section $s$ such that
$$
\Tr(s) := \sum_{g\in G} g\cdot s
$$
is a non-zero, $G$-invariant, meromorphic section of $N$ over $Y$. It is an easy
exercise to show that the homomorphism (\ref{eqn:homom}) is injective. This
implies that $\Pic\Mbar_{1,1}^\red$ is infinite cyclic. 

Since the Ramanujan tau function $\Delta$ is a section of $\Lbar_{12} =
\Lbar_2^{\otimes 6}$ and since $\div(\Delta) = [\infty]$, it follows that
$6[\Lbar_2] = [\O_{\Mbar_{1,1}^\red}(P)]$. This establishes the surjectivity
(\ref{eqn:homom}) and that $\Pic(\M_{1,1}^\red)$ is generated by $[\Lbar_2]$.
\end{proof}

\begin{exercise}[cf.\ {\cite[Thm.~3, p.~80]{serre}}]
\label{ex:mod_forms}
Use the preceding result to show that the meromorphic modular form $f :\h \to
\C$ is a section of $\Lbar_{2k}$ over $\Mbar_{1,1}$ (and $\Mbar_{1,1}^\red$) if
and only if
$$
\nu_i(f)/2 + \nu_\rho(f)/3
+ \sum_{\substack{P\in \coarseM_{1,1} \cr P \neq [i],[\rho]}} \nu_P(f)
= \frac{k}{6}.
$$
Use this to prove the statements (\ref{eqn:orderG}).\footnote{These results
imply Proposition~\ref{prop:dimensions} as in \cite[p.~88]{serre}.}
\end{exercise}

As in the local case, there is a short exact sequence
$$
0 \to \Pic \Mbar_{1,1}^\red \to \Pic \Mbar_{1,1} \to \Char(C_2) \to 1.
$$
Since the class of $\Lbar$ in $\Pic \Mbar_{1,1}$ maps to the generator of
$\Char(C_2)$, the previous result implies:

\begin{theorem}
\label{thm:pic_mbar}
The Picard group of $\Mbar_{1,1}$ is infinite cyclic and is generated by
the class of the Hodge bundle $\Lbar$. For all
$P\in \coarseM_{1,1}$,
$$
[\O_{\Mbar_{1,1}}(P)] = 12[\Lbar] \in \Pic\Mbar_{1,1}.
$$
Consequently, the sequence
$$
0 \to \Z[\O_{\Mbar_{1,1}}(\infty)] \to \Pic\Mbar_{1,1} \to \Pic\M_{1,1} \to 0
$$
is exact.
\end{theorem}

\section{The Algebraic Topology of $\Mbar_{1,1}$}

The homotopy type of a basic orbifold has already been discussed in
Section~\ref{sec:htpy_type}. Global orbifolds, such as $\Mbar_{1,1}$, also have
a well defined homotopy type. In this section we discuss the homotopy type of
$\Mbar_{1,1}$ and use it to compute its low dimensional (co)homology groups.

\subsection{The homotopy type of $\Mbar_{1,1}$}
Let $U$ be a contractible topological space on which $\SL_2(\Z)$ acts properly
discontinuously and fixed point freely, such as the standard model of
$E\SL_2(\Z)$. The groups
$$
C_2 = \{\pm \id\} \text{ and } \PZ
$$
are subgroups of $\SL_2(\Z)$, and thus act freely and discontinuously on $U$
as well. We can therefore consider the diagram
\begin{equation}
\label{eqn:diag}
\xymatrix@!C=3pc{
& (C_2\times\Z)\bs(U \times \h)
\ar[dl]_{\id_U \times p} \ar[dr]^{\id_U \times q} \cr
\SL_2(\Z)\bs (U\times \h) && C_2\bs (U\times \D)
}
\end{equation}
of topological spaces, where $\SL_2(\Z)$ acts diagonally on $U\times \h$, etc.
Because $U\times \h$ and $U\times \D$ are contractible and each of the groups
acts freely and discontinuously, the homotopy type of the 3 pieces
are:\footnote{These spaces are not that exotic: $BC_2 \simeq \R\P^\infty$, $B\Z
\simeq S^1$ and $B(\Z\times C_2) \simeq B\Z\times BC_2 \simeq S^1\times
\R\P^\infty$.}
$$
B\SL_2(\Z),\quad B C_2,\quad B(C_2\times \Z).
$$
So we can represent the diagram above as
$$
\xymatrix@!C=3pc{
& B(C_2\times\Z) \ar[dl]_P \ar[dr]^Q\cr
B\SL_2(\Z) && B C_2
}
$$
One can form the space
$$
\M_{1,1} \cup_{BC_2 \times \D^\ast} (BC_2\times\D^\ast) :=
B\SL_2(\Z) \cup_{B(C_2\times\Z)} B C_2
$$
by taking the pushout of the diagram (\ref{eqn:diag}) in the homotopy category.
Explicitly, it is the homotopy type of the space
$$
\big[\SL_2(\Z)\bs (U\times \h) \disjt [0,1]\times (C_2\times\Z)\bs(U \times \h)
\disjt C_2\bs (U\times \D)\big]/\sim
$$
obtained by identifying the $C_2\times \Z$ orbit of $(0,u,\tau) \in [0,1]\times
(C_2\times\Z)\bs(U \times \h)$ with the $\SL_2(\Z)$ orbit of $(u,\tau) \in
U\times \h$, and the orbit of $(1,u,\tau) \in [0,1]\times (C_2\times\Z)\bs(U
\times \h)$ with the $C_2$ orbit of $(u,q(\tau))$. Its homotopy type is well
defined. There is a well defined morphism
\begin{equation}
\label{eqn:covering}
\M_{1,1} \cup_{BC_2 \times \D^\ast} (BC_2\times\D^\ast) \to \Mbar_{1,1}
\end{equation}
of topological
orbifolds obtained by projecting the diagram (\ref{eqn:diag}) to the atlas of
$\Mbar_{1,1}$ along $U$.

\begin{definition}
The homotopy type of $\Mbar_{1,1}$ is defined to be the homotopy type of the
space $B\SL_2(\Z) \cup_{B(C_2\times\Z)} B C_2$ defined above.
\end{definition}

\begin{exercise}
Use the presentation (\ref{eqn:presentation}) to prove that $\PSL_2(\Z)$ is
isomorphic to the free product of $C_2\ast C_3$ and that $\SL_2(\Z)$ is
an extension
$$
1 \to C_2 \to \SL_2(\Z) \to C_2\ast C_3 \to 1.
$$
Deduce that $B\SL_2(\Z)$ is a $BC_2$ bundle over $BC_2 \vee BC_3$.
\end{exercise}

\begin{exercise}
Show that $\M_{1,1}$ and $\Mbar_{1,1}$  are both homotopy equivalent to a
CW-complex with only a finite number of cells in each dimension. Deduce that
their homology and cohomology groups are finitely generated in each degree.
\end{exercise}

Invariants of the homotopy type of $\Mbar_{1,1}$, such as its homotopy, homology
and cohomology groups are defined to be those of its homotopy type.

\subsection{The fundamental group of $\Mbar_{1,1}$}

We can apply van~Kampen's theorem to compute the fundamental group of
$\Mbar_{1,1}$. It is the amalgamated free product
\begin{equation}
\label{eqn:free_prod}
\SL_2(\Z)\ast_{C_2 \times \Z} C_2
\end{equation}
obtained by pushing out the diagram
$$
\xymatrix@C=4pc{\SL_2(\Z) & C_2 \times \PZ \ar[l]_(.55){\text{inclusion}}
\ar[r]^(.65){\text{projection}} & C_2}
$$
in the category of groups.

\begin{exercise}
Recall the definition of $S,T,U \in \SL_2(\Z)$ and the presentation
(\ref{eqn:presentation}) of $\SL_2(\Z)$. Set
$$
\That = \begin{pmatrix} 1 & 0 \cr 1 & 1\end{pmatrix}.
$$
Verify that $U=T^{-1}\That$ and that $S=T^{-1}\That T^{-1}$. Deduce that
$\SL_2(\Z)$ is generated by $T$ and $\That$. Use this to show that the
amalgamated free product (\ref{eqn:free_prod}) is the trivial group.
\end{exercise}

This proves:

\begin{proposition}
The orbifold $\Mbar_{1,1}$ is simply connected. Consequently,
$H_1(\Mbar_{1,1};\Z)=0$.
\end{proposition}

\subsection{Chern classes}

Orbifold vector bundles over $\M_{1,1}$ give rise to genuine vector bundles over
its homotopy type $B\SL_2(\Z)$. Similarly, an orbifold vector bundle over
$\Mbar_{1,1}$ determines a genuine vector bundle over its homotopy type.
One can therefore define Chern classes
$$
c_j(E) \in H^{2j}(\M_{1,1};\Z),\quad c_j(F) \in H^{2j}(\Mbar_{1,1};\Z)
$$
of orbifold vector bundles $E$ over $\M_{1,1}$ and $F$ over $\Mbar_{1,1}$.
In particular, we have Chern class homomorphisms
$$
c_1 : \Pic \M_{1,1} \to H^2(\M_{1,1};\Z) \text{ and }
c_1 : \Pic \Mbar_{1,1} \to H^2(\Mbar_{1,1};\Z)
$$
such that the diagram
$$
\xymatrix{
\Pic \Mbar_{1,1} \ar[r]^(.45){c_1} \ar[d] & H^2(\Mbar_{1,1};\Z) \ar[d]\cr
\Pic \M_{1,1} \ar[r]^(.45){c_1} & H^2(\M_{1,1};\Z)
}
$$
commutes.

\begin{exercise}
Show that $c_1 : \Pic \M_{1,1} \to H^2(\M_{1,1};\Z)$ is an isomorphism.
\end{exercise}

\subsection{Low dimensional cohomology of $\Mbar_{1,1}$}

The homology and cohomology groups of $\Mbar_{1,1}$ can be computed using the
Mayer-Vietoris sequence
\begin{multline*}
\cdots \to H^k(\Mbar_{1,1}) \to H^k(\M_{1,1}) \oplus H^k(BC_2) \cr
  \to H^k(\D^\ast\times BC_2) \to H^{k+1}(\Mbar_{1,1}) \to \cdots
\end{multline*}
associated to the covering (\ref{eqn:covering}) or
use the ``Gysin sequence''
$$
\cdots \to H^k(\Mbar_{1,1}) \to H^k(\M_{1,1}) \to
\widetilde{H}^{k-1}(BC_2) \to H^{k+1}(\Mbar_{1,1}) \to \cdots
$$
associated to the cofibration sequence
$$
\M_{1,1} \to \Mbar_{1,1} \to (\D,S^1)\times BC_2.
$$

\begin{exercise}
Justify these sequences.
\end{exercise}

Since $\Mbar_{1,1}$ is simply connected, $H_1(\Mbar_{1,1})$ and
$H^1(\Mbar_{1,1})$ vanish with all coefficients.

\begin{proposition}
The first Chern class
$$
c_1 : \Pic\Mbar_{1,1} \to H^2(\Mbar_{1,1};\Z)
$$
is an isomorphism. Consequently, $H^2(\Mbar_{1,1};\Z)$ is infinite cyclic.
\end{proposition}

\begin{proof}
Consider the diagram
\begin{equation}
\label{eqn:diag2}
\xymatrix{
0 \ar[r] & \Z[\O_{\Mbar_{1,1}}(\infty)] \ar[r] \ar@{.>}[d]_e &
\Pic \Mbar_{1,1} \ar[r]\ar[d]_{c_1} & \Pic \M_{1,1} \ar[r]\ar[d]_{c_1} & 0 \cr
0 \ar[r] & H^0(BC_2;\Z) \ar[r] & H^2(\Mbar_{1,1};\Z) \ar[r] &
H^2(\M_{1,1};\Z) \ar[r] & 0
}
\end{equation}
The top row is exact by Theorem~\ref{thm:pic_mbar}. The second row is a  portion
of the Gysin sequence. It is exact as $H^1(\M_{1,1};\Z)=0$ and as $H^1(BC_2;\Z)
= \Hom(C_2,\Z)=0$. Since $c_1 : \Pic \M_{1,1} \to H^2(\M_{1,1};\Z)$ is an
isomorphism, there is a map $e : \Z[\O_{\coarseM_{1,1}}(\infty)] \to
H^0(BC_2;\Z)$ making the diagram commute. This implies that
$H^2(\Mbar_{1,1};\Z)$ is infinite cyclic. To see that the middle vertical map is
an isomorphism requires more work. To complete the proof we will sketch a proof
that $e$ is an isomorphism.

Consider the portion
$$
\xymatrix{
\Z[\O_{\coarseM_{1,1}}(\infty)] \ar[r] \ar[d]_e &
\Pic \coarseM_{1,1} \ar[d]_{c_1} \cr
H^0(\{\infty\};\Z) \ar[r] & H^2(\coarseM_{1,1};\Z)
}
$$
of the analogue of the diagram (\ref{eqn:diag2}) for $(\coarseM_{1,1},M_{1,1})$.
Since $(\coarseM_{1,1},M_{1,1})$ is isomorphic to $(\P^1,\C)$, all four maps in
this diagram are isomorphisms. The map $\pi : (\Mbar_{1,1},\M_{1,1}) \to
(\coarseM_{1,1},M_{1,1})$ induces a morphism of Gysin sequences that is
compatible with Chern classes. It maps the commutative square in this diagram to
the left-hand square in (\ref{eqn:diag2}). The map on the top left corner is an
isomorphism. The map on the bottom left hand corner is the map
$$
H^0(\{\infty\};\Z) \cong H^1(\infty\times (\D,S^1)) \to
H^1(BC_2\times (\D,S^1)) \cong H^0(BC_2)
$$
which is an isomorphism. It follows that $e$ is an isomorphism as claimed.
\end{proof}

\section{Concluding Remarks}

Our goal in this final section is to tie together several loose ends to explain
how the moduli space $\Mbar_{1,1}$ can be viewed as a Deligne-Mumford stack in
the category of schemes over $\Q$. Along the way, we identify the fundamental
group of several moduli spaces of elliptic curves with the braid group on 3
strings, the group of the trefoil knot, and with a canonical central extension
of $\SL_2(\Z)$.

\subsection{The moduli space $\M_{1,\vec{1}}$}

In this section we will consider the problem of determining the moduli space
$\M_{1,\vec{1}}$ of triples $(X,P,\v)$, where $(X,P)$ is an elliptic curve and
$\v \in T_P X$ is a non-zero holomorphic tangent vector to $X$ at $P$. Since the
holomorphic cotangent bundle of $X$ is trivial, such a triple is determined by
and determines a triple $(X,P,\w)$, where $\w$ is a non-zero holomorphic
differential on $X$. The correspondence is given by insisting that $\langle
\w,\v \rangle = 1$.

It follows from Lemma~\ref{lem:hodge} that $\M_{1,\vec{1}} = \L^\ast$ where
$\L^\ast$ is the $\C^\ast$-bundle obtained by removing the zero section from
$\L$. This is a genuine complex surface, not just an orbifold.

\begin{exercise}
Show that the action of $\SL_2(\Z)$ on $\C^\ast\times\h$
\begin{equation}
\label{eqn:action}
\begin{pmatrix} a & b \cr c & d \end{pmatrix} : (u,\tau)
\mapsto \big((c\tau+d)u,(a\tau+b)/(c\tau+d)\big)
\end{equation}
is fixed point free. Deduce that $\L^\ast$ is a genuine complex surface whose
universal covering is $\C\times \h$ and whose fundamental group is an extension
$$
0 \to \Z \to \pi_1(\L^\ast) \to \SL_2(\Z) \to 1.
$$
\end{exercise}

\begin{exercise}
\label{ex:fine_moduli}
Show that the complex surface $\M_{1,\vec{1}}$ is a fine moduli space of triples
$(X,P,\w)$, where $(X,P)$ is an elliptic curve and $\w$ is a non-zero
holomorphic differential on $X$.
\end{exercise}

The moduli space $\M_{1,\vec{1}}$ has a natural {\em partial} compactification.
Namely
$$
\Mbar_{1,\vec{1}} := \Lbar^\ast,
$$
the $\C^\ast$-bundle associated to $\Lbar^\ast$.\footnote{This is {\em not}
compact. However, $\M_{1,\vec{1}}$ does admit a natural smooth compactification
by adding the sections $0$ and $\infty$ to $\Lbar^\ast$. Explicitly, this is:
$$
\Lbar \cup_{\Lbar^\ast} \Lbar_{-1} = \P(\Lbar\oplus \O_{\Mbar_{1,1}})
$$
which is a $\P^1$-bundle over $\Mbar_{1,1}$. We will not use this
compactification.}

\subsection{The topology of $\M_{1,\vec{1}}$}

Since the maximal compact subgroup of $\SL_2(\R)$ is the circle $\SO(2)$, its
fundamental group is isomorphic to $\Z$. This implies that the universal
covering group of $\SL_2(\R)$ is an extension
$$
0 \to \Z \to \SLtilde_2(\R) \to \SL_2(\R) \to 1.
$$
Denote the inverse image of $\SL_2(\Z)$ in $\SLtilde_2(\R)$ by $\SLtilde_2(\Z)$.
It is an extension
$$
0 \to \Z \to \SLtilde_2(\Z) \to \SL_2(\Z) \to 1.
$$

\begin{proposition}
There is a natural isomorphism
$$
\pi_1(\M_{1,\vec{1}},p) \overset{\simeq}{\longrightarrow} \SLtilde_2(\Z),
$$
where $p:\C \times \h \to \C^\ast \times \h \to \M_{1,\vec{1}}$ is the base
point that takes $(v,\tau)$ to the $\SL_2(\Z)$ orbit of $(e^v,\tau) \in
\C^\ast\times \h$.
\end{proposition}

\begin{proof}
The group $\SL_2(\R)$ acts on $\C\times\h$ by the formula (\ref{eqn:action}). It
preserves the metric $\|(u,\tau)\| = |u|\Im(\tau)^{-1/2}$ and therefore
restricts to an action on $S^1\times\h$. This action is easily checked to be
transitive. The isotropy group of $(1,i)$ is trivial. The mapping
$$
\SL_2(\R) \to S^1\times \h; \qquad g\mapsto g(1,i)
$$
is therefore a diffeomorphism. It is also $\SL_2(\R)$-equivariant with respect
to the two natural left $\SL_2(\R)$-actions. The mapping therefore lifts to a
diffeomorphism
$$
\SLtilde_2(\R) \to \R \times \h,
$$
which implies that $\SLtilde_2(\R)$ is contractible.

The $\SL_2(\R)$-action on $S^1\times \h$ can be lifted to an
$\SLtilde_2(\R)$-action on $\R\times \h$ by defining the previous mapping to be
$\SLtilde_2(\R)$-equivariant. It follows that the unit circle bundle of $\L$ is
the quotient
$$
\SLtilde_2(\Z)\bs \SLtilde_2(\R) \cong \SLtilde_2(\Z)\bs (\R\times\h)
\cong \SL_2(\Z)\bs (S^1\times\h).
$$
Since the inclusion of the unit circle bundle into $\L^\ast$ is a homotopy
equivalence, it follows that the fundamental group of $\L^\ast$ is isomorphic
to $\SLtilde_2(\Z)$.
\end{proof}

\subsection{Plane cubics and $\M_{1,\vec{1}}$}

Consider the universal family of cubics $E \to \C^2$ where
$$
E = 
\big\{([x,y,z],(a,b) \in \P^2 \times \C^2  : zy^2 = 4 x^3 - axz^2 - bz^3 \big\}.
$$
The total space $E$ is smooth except over the origin $a=b=0$. The point
$[0,1,0]$ lies in each fiber and therefore defines a section of $E \to \C^2$.

Recall that $D(a,b) = a^3 - 27 b^2$ is the discriminant of the cubic $4x^3 - ax
-b$.  Let $\Delta$ be the divisor in $\C^2$ defined by $D=0$. It is called the
{\em discriminant locus}. Over $\C^2 - \Delta$ the fibers of $E \to \C^2$ are
smooth; over $\Delta-\{0\}$ they are nodal cubics; and over the origin the fiber
is the cuspidal cubic. As we have seen in Section~\ref{sec:extended_curve} the
rational differential $dx/y$ on $E$ restricts to a non-zero holomorphic
differential on each smooth fiber of $E$. Since $\M_{1,\vec{1}}$ is a fine
moduli space for triples $(X,P;\w)$ (Cf.\  Exercise~\ref{ex:fine_moduli}), there
is a holomorphic mapping
$$
F : \C^2 - \Delta \to \M_{1,\vec{1}}
$$
that classifies the tautological family of cubics $E$ over $\C^2 - \Delta$ and
the differential $dx/y$.

Define $\C^\ast$-actions on $\C^2$ and $\M_{1,\vec{1}}$ by
$$
\lambda\cdot (a,b) := (\lambda^{-4}a,\lambda^{-6}b) \text{ and }
\lambda \cdot [X,P;\w] := [X,P;\lambda \w],
$$
respectively. The $\C^\ast$-action restricts to an action on $\C^2 - \Delta$.

\begin{proposition}
\label{prop:equivariant}
The mapping $F$ is a $\C^\ast$-equivariant biholomorphism.
\end{proposition}

\begin{proof}
Proposition~\ref{prop:plane_cubics} and the results in
Section~\ref{sec:extended_curve} imply that $F$ is a bijection. We use modular
forms to construct the inverse of $F$.

Define
$ \tilde{G} : \C^\ast \times \h \to \C^2 - D^{-1}(0)$ by
$$
f(u,\tau) = \big(u^{-4}g_2(\tau),u^{-6}g_3(\tau)\big).
$$
It is $\C^\ast$-invariant and also $\SL_2(\Z)$-equivariant with respect to the
action
$$
\gamma : (u,\tau) \mapsto \big((c\tau + d)u,\gamma(\tau)\big).
$$
It therefore induces a holomorphic function
$$
G : \L^\ast \to \C^2 - \Delta.
$$
This is an inverse of $F$ as it is $\C^\ast$-equivariant and
$(1,\tau) \in \C\times \h$ corresponds to $(\C/\Lambda_\tau;dz)$ and
$G(1,\tau)$ corresponds to the curve
$y^2 = 4x^3 - g_2(\tau)x - g_3(\tau)$ with the differential
$dx/y$ which corresponds to $dz$ by Proposition~\ref{prop:embedding}.
\end{proof}

\begin{corollary}
The fundamental group of $\M_{1,\vec{1}}$ is isomorphic to the braid group $B_3$
on $3$ strings and also to the fundamental group of the complement of the
trefoil knot in the $3$-sphere. Both groups are isomorphic to $\SLtilde_2(\Z)$.
\end{corollary}

\begin{proof}[Sketch of Proof]
Since $\M_{1,\vec{1}}$ is the variety $\C^2$ minus the cusp $D(a,b)=0$, we need
only compute the fundamental group of the space $\C^2 - \Delta$. The
intersection $L$ of the cusp $D=0$ with the unit sphere $S^3$ in $\C^2$ is the
trefoil knot. This can be seen by writing $S^3$ as the union of the two solid
tori that intersect along the 2-torus $|a| = c|b|$ for suitably chosen $c$. The
discriminant locus $a^3 = 27 b^2$ intersects this in the torus knot $\theta
\mapsto (1+c^2)^{-1/2}(c e^{2i\theta}, e^{3i\theta})$ of type $(2,3)$ --- the
trefoil knot.

The $\R_+$-action 
$$
t : (a,b) \mapsto (t^2 a,t^3 b)
$$
restricts to an action on $\C^2-\Delta$. Each orbit intersects $S^3$
transversely in a unique point, which implies that the action induces
a diffeomorphism
$$
\R_+ \times (S^3 - L) \overset{\simeq}{\longrightarrow} \C^2 - \Delta.
$$
The inclusion $S^3 - L \hookrightarrow \C^2 - \Delta$ is therefore a homotopy
equivalence so that $\pi_1(\M_{1,\vec{1}})$ is isomorphic to $\pi_1(S^3 - L)$.

The braid group $B_n$ is the fundamental group of the quotient of
$$
Y_n :=
\{(\lambda_1,\dots,\lambda_n) \in \C^n : \lambda_1 + \dots + \lambda_n = 0.
\quad \lambda_j \neq \lambda_k \text{ when } j\neq k\}
$$
by the natural action of the symmetric group $S_n$. The quotient is the space
of polynomials
$$
X_n := \big\{p(T) = T^n + a_{n-2}T^{n-2} + \dots + a_0 :
\text{ discriminant of }p(T) \neq 0\big\}.
$$
The coordinate $a_j$ of $X_n$ is $(-1)^j$ the $j$th elementary symmetric
function of the ``roots'' $\lambda_j$ of $p(T)$.

Specializing to the case $n=3$, we see that $\pi_1(\C^2 - D) \cong B_3$.
\end{proof}

\begin{remark}
The decomposition of $S^3$ into two solid tori described in the proof
restricts to a decomposition of $S^3 - L$. Van Kampen's Theorem then gives
a presentation
$$
\SLtilde_2(\Z) \cong \pi_1(\C^2 - \Delta) \cong \pi_1(S^3-L)
\cong \langle S,U : S^2 = U^3 \rangle
$$
where $U$ and $S$ represent the positive generators of the circles $a =0$ and
$b=0$, respectively, in $S^3$. These map to the generators $S$ and $U$ of
$\SL_2(\Z)$ given in the presentation (\ref{eqn:presentation}) of $\SL_2(\Z)$.
The kernel of the homomorphism to $\pi_1(\M_{1,1}) \cong \SL_2(\Z)$ is generated
by the central element $S^4 = U^6$.

The braid group $B_3$ has presentation
$$
B_3 = \langle\sigma_1, \sigma_2 :
\sigma_1\sigma_2\sigma_1 = \sigma_2\sigma_1\sigma_2 \rangle.
$$
An isomorphism with $\pi_1(S^3-L)$ is given by $S \mapsto
\sigma_1\sigma_2\sigma_1$ and $T\mapsto \sigma_1\sigma_2$. The center of $B_3$
is generated by the full twist $S^3$; the kernel of the homomorphism to
$\SL_2(\Z)$ is generated by the square of this --- a double twist.
\end{remark}

\begin{exercise}
Show that the orbifold $\L_2^\ast$ is isomorphic to the orbifold quotient $C_2
\bbs \L^\ast$ of the complex manifold $\L^\ast$ by the trivial $C_2$-action.
Deduce that $\L^\ast_2 = C_2 \bbs (\C^2 - \Delta)$ where $C_2$ acts trivially on
$\C^2 - \Delta$.
\end{exercise}

\begin{remark}
It is not difficult to show that, as {\em stacks}, $\M_{1,1}$ is isomorphic to
the quotient of $\L^\ast$ by the natural $\C^\ast$-action. Combining this with
Proposition~\ref{prop:equivariant}, we have a stack isomorphism
$$
\M_{1,1} \cong \C^\ast \bbs (\C^2 - \Delta).
$$  
This is significant for two reasons. First, it shows that $\M_{1,1}$ is the
quotient of an affine variety by an algebraic action of a reductive group.
Second, this description works over any field of characteristic not equal to 2
or 3 to give an algebraic description of the moduli stack of elliptic curves.
Below we shall explain briefly how to generalize this to write $\Mbar_{1,1}$ as
a stack over $\Q$.
\end{remark}

Note that the $\C^\ast$-action on $\C^2 - \Delta$ factors through the
homomorphism $\C^\ast \to \C^\ast$ that takes $u$ to $u^2$. This is related
to the fact that the automorphism group of every point of $\M_{1,1}$ contains
$C_2$.

\begin{exercise}
Show that if $G$ is an algebraic group and if $X$ is a variety over an
algebraically closed field $F$ of characteristic zero on which $G$ acts
transitively, then for each $x\in X$, the natural mapping
$$
G_x \bbs \{x\} \to G\bbs X
$$
is an isomorphism of stacks, where $G_x$ denotes the isotropy group of $x$.
(That is, it is an equivalence of categories when viewed as a functor of
groupoids.) In particular, if $\Gm$ acts on itself by the character $u \mapsto
u^d$, then the stacks $\Gm\bbs \Gm$ and $\bmu_d \bbs \Spec(F)$ are isomorphic,
where $\bmu_d$ acts trivially on $\Spec F$.
\end{exercise}

\subsection{$\Mbar_{1,1}$ as a stack over $\Q$}
\label{sec:mbar_Q}

The results in this section lead to a construction of $\Mbar_{1,1}$ as a 
Deligne-Mumford stack in the category of schemes over $\Q$. The starting point
is the statement that $\M_{1,1}$ is isomorphic to the stack $\C^\ast \bbs (\C^2
- \Delta)$ where the $\C^\ast$-action on $\C^2$ is defined by $\lambda\cdot(a,b)
= (\lambda^{-4} a,\lambda^{-6}b)$. One can show that if $F$ is a field of
characteristic not equal to 2 or 3, then the moduli stack $\M_{1,1/F}$ of smooth
elliptic curves over $F$ is
$$
\M_{1,1/F} \cong \Gm_{/F} \bbs (\A^2_F - \Delta)
$$
where $\Gm_{/F}$ denotes the multiplicative group over $F$.

The next observation is that, over $\C$, there is a stack isomorphism
$$
\Mbar_{1,1} \cong \C^\ast \bbs (\C^2 -\{0\}).
$$
This is a Deligne-Mumford stack. That is, for each isomorphism class $[E] \in
\coarseM_{1,1}$, there is a morphism $T \to \C^2 - \{0\}$ from a smooth
algebraic curve $T$ that is transverse to each $\C^\ast$-orbit. Such a morphism
$T=\C \to \C^2 - \{0\}$ corresponds to the family
$$
E_t :\qquad  y^2 = 4x(x-1)(x-t),\quad t\in \C
$$
of cubics, each with differential $dx/y$. This family is considered to be an
``\'etale neighbourhood'' of $[E_t]$ in $\Mbar_{1,1}$ for each $t\in \C$.

This construction works equally well over any field of characteristic not equal
to 2 or 3 to give a construction of the moduli stack $\Mbar_{1,1/F}$ of stable
elliptic curves in the category of schemes over $F$:
$$
\Mbar_{1,1/F} \cong \Gm_{/F} \bbs (\A^2_F - \{0\})
$$
It is a Deligne-Mumford stack.

\appendix
\section{Background on Riemann Surfaces}
\label{app:RS}

This is a very brief summary of some basic facts about Riemann surfaces.
Detailed expositions can be found in \cite{forster,griffiths,griffiths-harris}.

\subsection{Topology}

Riemann surfaces, like all complex manifolds, have a natural orientation.

\begin{exercise}
\label{ex:pos}
Denote the complex parameter on the disk $\D$ by $z$. Write $z = x+iy$ where $x$
and $y$ are real. Show that if $\w$ is a non-vanishing holomorphic 1-form on
$\D$, then $i\,\w\wedge\wbar$ is a positive multiple of $dx\wedge dy$. Deduce
that every Riemann surface has a natural orientation which locally agrees with
the standard orientation of the complex plane. Deduce that if $\w$ is a
holomorphic 1-form on a compact Riemann surface, then
$$
i\int_X \w\wedge \wbar \ge 0
$$
with equality if and only if $\w = 0$. This is equivalent to Riemann's second
bilinear relation.
\end{exercise}

\begin{exercise}
Suppose that $X$ is a Riemann surface and that $\w$ is a holomorphic 1-form
on $X$. Show that $\w$ is closed and therefore determines an element of
$$
H^1(X;\C) := \Hom_\Z(H_1(X),\C)
$$
by the formula
$$
\w : \gamma \mapsto \int_\gamma \w.
$$
Show that if $\w$ is an exact 1-form, then $\w = df$, where $f: X \to \C$ is
holomorphic. Deduce that if $X$ is compact, then the mapping
$$
H^0(X,\Omega^1_X) \to H^1(X;\C)
$$
is injective, and therefore that $H^0(X,\Omega^1_X)$ is finite dimensional.
\end{exercise}

\begin{exercise}
Suppose that $\phi : X \to \C$
is a smooth function. Show that if
$$
d\phi = \w_1 + \wbar_2
$$
where $\w_1$ and $\w_2$ are holomorphic 1-forms, then $\phi$ his harmonic.
Deduce that if $X$ is compact, then $\phi$ is constant and $\w_1=\w_2=0$. Use
this to show that the $\C$-linear mapping
\begin{equation}
\label{eqn:hodge}
H^0(X,\Omega^1_X) \oplus \overline{H^0(X,\Omega^1_X)} \to H^1(X;\C),
\end{equation}
where overlining denotes complex conjugation, that takes $(\w_1,\wbar_2)$ to
$$
\gamma \mapsto \int_\gamma (\w_1 + \wbar_2)
$$
is injective.
\end{exercise}

The {\em genus} of a compact Riemann surface\footnote{Our Riemann surfaces are
assumed to be connected.} $X$ is, by definition, the dimension of its space of
holomorphic 1-forms:
$$
g(X) := \dim H^0(X,\Omega^1_X).
$$
The topological surface that underlies $X$ is orientable, and is therefore 
determined up to diffeomorphism by its first Betti number
$$
b_1(X) := \rank H_1(X;\Z).
$$
A basic fact in the theory of Riemann surfaces, which can be proved using Hodge
theory, is that $b_1(X) = 2g(X)$. This equality is equivalent to the statement
that (\ref{eqn:hodge}) is an isomorphism and is a special case of the Hodge
theorem for compact K\"ahler manifolds.

\subsection{Local structure of holomorphic mappings}

Recall from complex analysis that if $w=f(z)$ is meromorphic and defined in a
neighbourhood of $z=0$, then we can write
$$
f(z) = z^k g(z)
$$
where $g$ is holomorphic near $z=0$ and $g(0) \neq 0$. The integer $k$ is called
the {\em order of $f$ at $z=0$}. We shall write $k = \ord_0 f$. 

If $f(z)$ is holomorphic and satisfies $f(0) = 0$, then the order $k$ of $f$ at
$z=0$ is positive. In this case, by basic complex analysis, there is a
holomorphic function $\phi(z)$ defined in a neighbourhood of the origin such
that $g(z) = \phi(z)^k$. Then $u = z\phi(z)$ is a local holomorphic coordinate
defined in a neighbourhood of the origin and, with respect to the new coordinate
$u$, $f$ is given by $w=u^k$.

From this it follows that every non-constant holomorphic mapping  between
Riemann surfaces is locally of the form $z \mapsto z^k$ for some positive
integer $k$. More precisely, given $x \in X$, there is a local holomorphic
coordinate $z$ about $x\in X$ and a local holomorphic coordinate $w$ about $y\in
Y$ such that $f$ is locally given by $w=z^k$ in terms of these coordinates.
We call $k$ the {\em local degree} of $f$ at $x$ and denote it by $\nu_x(f)$.

\begin{exercise}
Prove that every non-constant holomorphic mapping $f : X \to Y$ between Riemann
surfaces is open. (That is, the image of every open set is open.)
\end{exercise}

\begin{exercise}
Show that if $X$ is compact and $Y$ is connected, then every holomorphic mapping
$f : X \to Y$ is surjective. Deduce that every holomorphic function $f:X \to Y$
is constant when $Y$ is non-compact.
\end{exercise}

\begin{exercise}
\label{ex:order}
Show that if $f:X \to Y$ is a non-constant mapping between compact Riemann
surfaces, then the function
$$
y \mapsto \sum_{x\in f^{-1}(y)} \nu_x(f)
$$
is locally constant and therefore constant if $Y$ is connected. The common value
is called the {\em degree} of $f$. Show that if $\nu_x(f) = 1$ for all $x\in X$,
then $f$ is a covering map.
\end{exercise}

\begin{exercise}[Riemann-Hurwitz formula]
\label{ex:RH}
Suppose that $f : X \to Y$ is a non-constant mapping of degree $d$ between
compact Riemann surfaces. For each $x \in X$, define
$$
b = \sum_{x\in X} (\nu_x(f)-1).
$$
(This is well define as $\nu_x(f) = 1$ for all but a finite number of $x\in X$.)
Show that
$$
\chi(Y) = d\chi(X) - b.
$$
(Hint: Triangulate $Y$ so that each critical value is a vertex. Lift this
triangulation to $X$. Compute Euler characteristics.) Deduce that $b$ is even.
Show that if $Y=\P^1$, then the genus of $X$ is
$$
g(X) = b/2 + 1 - d.
$$
In particular, if $g=1$ and $d=2$, then $b=4$.
\end{exercise}

\begin{exercise}
Show that a 1-1 holomorphic mapping $f : X \to Y$ between compact Riemann
surfaces is a biholomorphism. (That is, $f$ has a holomorphic inverse $g: Y \to
X$.)
\end{exercise}

\subsection{Divisors and line bundles}

A divisor $D$ on a Riemann surface $X$ is a locally finite formal linear
combination
$$
\sum_{x\in X} n_x [x]
$$
of points of $X$. We say that $D$ is {\em effective} and write $D\ge 0$ when
each $n_x \ge 0$. When $X$ is compact, the sum is finite. Thus one can define
the {\em degree} $\deg D$ of a divisor $D$ on a compact Riemann surface to be
the sum of its coefficients:
$$
\deg D = \sum_{x\in X} n_x.
$$

To a holomorphic line bundle $L \to X$ with a meromorphic section $s$, we may
associate the divisor
$$
\div s := \sum_{x\in X} \ord_x s.
$$
The section $s$ is holomorphic precisely when $\div s \ge 0$. Every other
meromorphic section of $L$ is of the form $fs$, where $f$ is a meromorphic
function. The space of holomorphic sections $H^0(X,L)$ of $L$ thus equals
$$
L(D) := \{\text{meromorphic functions $f$ on $X$: } \div f + D \ge 0\}.
$$
The corresponding sheaf is denoted by $\O_X(D)$. Its space of sections
$\O_X(D)(U)$ over the open subset $U$ of $X$ is defined by
$$
\O_X(D)(U) = \{\text{meromorphic functions $f$ on $U$: } \div f + D|_U \ge 0\}
$$
where the restriction $D|_U$ of $D$ to $U$ is defined by
$$
D|_U := \sum_{x\in U} n_x [x].
$$
Evidently, $L(D) = H^0(X,\O_X(D))$.

Two divisors $D$ and $D'$ on a Riemann surface are {\em linearly equivalent} if
there is a non-zero meromorphic function $f$ on $X$ such that
$$
D' = D + \div f.
$$
A divisor of the form $\div f$ is said to be {\em principal}. Every principal
divisor on a compact Riemann surface has degree 0. Thus linearly equivalent
divisors have the same degree.

\begin{exercise}
Suppose that $f : X \to \P^1$ is a non-constant meromorphic function. To each $a
\in \P^1$, define
$$
D_a = \sum_{x\in f^{-1}(a)} \nu_x(f) [x].
$$
Show that if $a\neq \infty$, then
$$
D_a - D_\infty = \div (f-a)
$$
Deduce that any two fibers $D_a$ of $f$ are linearly equivalent. Note that
the degree of each $D_a$ equals the degree of $f : X \to \P^1$.
\end{exercise}

It is a fact that every holomorphic line bundle on a Riemann surface has a
non-zero meromorphic section.\footnote{If $X$ is not compact, then every
holomorphic line bundle has a holomorphic section.}

\begin{exercise}
Show that for every divisor $D$ on a Riemann surface $X$ there is a holomorphic
line bundle $\L \to X$ with a meromorphic section $s$ whose divisor is $D$. Show
that the map
$$
\O_X(D)(U) \to H^0(U,\L)
$$
that takes $f$ to $fs$ induces an isomorphism between $\O_X(D)$ and the sheaf of
holomorphic sections of $\L$. In particular, $L(D)$ is isomorphic to
$H^0(X,\L)$.
\end{exercise}

Suppose that $\L$ is a line bundle over $X$ and that $D$ is a divisor on $X$.
Define $\L(D) = \L\otimes \O_X(D)$. 

\begin{exercise}
Show that the group (under tensor product) of holomorphic line bundles on $X$ is
isomorphic to the group of divisors on $X$ modulo principal divisors.
\end{exercise}

\subsection{Riemann-Roch formula}
\label{sec:rr}

The canonical divisor class $K_X$ of a compact Riemann surface $X$ is the
divisor class associated to it holomorphic cotangent bundle. In concrete terms,
the canonical class is the divisor class of a non-zero meromorphic 1-form on
$X$.

For a divisor $D$ on $X$, define
$$
\ell(D) := \dim L(D) \in \N \cup \{\infty\}.
$$
\begin{exercise}
Note that $g(X) = \ell(K_X)$. Suppose that $X$ is a compact Riemann surface of
genus $g$ and that $D$ is a divisor on $X$. Show that
\begin{enumerate}

\item if $\deg D = 0$, then $\ell(0) = 0$ or $1$ and that $\ell(D)=1$ if and
only if $D$ is principal;

\item if $\deg D < 0$, then $\ell(D) = 0$;

\item if $P \in X$, then $\ell(D)\le \ell(D+P) \le 1+\ell(D)$;

\item if $D$ is effective, then $\ell(D) \le 1 + \deg D$;

\item $\ell(D)$ is finite for all $D$.

\end{enumerate}
\end{exercise}

\begin{theorem}[Riemann-Roch formula]
\label{thm:rr}
If $X$ is a compact Riemann surface of genus $g$ and $D$ is a divisor on $X$,
then
$$
\ell(D) - \ell(K_X - D) = \deg D + 1 - g.
$$
\end{theorem}

\begin{exercise}
Show that $\deg K_X = 2g-2$.
\end{exercise}

\subsection{Moduli of genus 0 Riemann surfaces}

\begin{exercise}
Show that if $X$ is a compact Riemann surface of genus 0, then there is a degree
1 holomorphic mapping $f: X \to \P^1$. Deduce that $f$ is a biholomorphism.
\end{exercise}

\begin{exercise}
Show that the group of biholomorphisms of $\P^1$ is isomorphic to $\PSL_2(\C)$,
where $\SL_2(\C)$ acts on $\P^1$ via fractional linear transformations:
$$
\begin{pmatrix} a & b \cr c & d \end{pmatrix} : z \mapsto
\frac{az+b}{cz+d}.
$$
Show that $\PSL_2(\C)$ acts 3-transitively on $\P^1$. That is, given two sets of
distinct points $\{x_1,x_2,x_3\}$ and $\{y_1,y_2,y_3\}$, there exists $\phi \in
\PSL_2(\C)$ such that $\phi(x_j) = y_j$, $j=1,2,3$. Show that $\phi$ is unique.
\end{exercise}

\begin{exercise}
Suppose that $x_1,\dots, x_n \in \P^1$ are distinct. Define
$$
\Aut(\P^1,\{x_1,\dots,x_n\}) 
= \{\phi \in \Aut \P^1 : f(x_j) = x_j,\ j=1,\dots,n\}.
$$
Show that this group is finite (resp.\ trivial) if and only $n\ge 3$.
\end{exercise}

\subsection{The action of $\SL_2(\Z)$ on the upper half plane}

The group $\SL_2(\Z)$ acts on the upper half plane $\h$ by
$$
\begin{pmatrix}a & b\cr c & d\end{pmatrix} : z \mapsto \frac{az+b}{cz+d}.
$$
The boundary of the upper half plane is $\P^1(\R) = \R \cup \{\infty\}$. This is
a circle on the Riemann sphere which forms the boundary of $\h$. Let $\hbar$ be
the closure of $\h$ in the Riemann sphere $\P^1$; it is the union of $\h$ and
$\P^1(\R)$. Recall that every non-trivial element of $\PSL_2(\C)$ has at most
two fixed points in $\P^1$. Note that the fixed points of elements of
$\PSL_2(\R)$ are real or occur in complex conjugate pairs. Consequently, each
element of $\PSL_2(\R)$ has at most one fixed point in $\h$.

\begin{exercise}
Suppose that $A \in \SL_2(\Z)$ is not a scalar matrix. Show that $A$ has
exactly
\begin{enumerate}
\item one fixed point in $\h$ if and only if $|\tr A| < 2$;
\item one fixed point in $\P^1(\R)$ if and only if $|\tr A| = 2$;
\item two fixed points in $\P^1(\R)$ if and only if $|\tr A| > 2$.
\end{enumerate}
Show that $T\in \SL_2(\R)$ has finite order if and only if $A$ has a fixed
point in $\h$.
\end{exercise}

Suppose that $m\in \N$. The {\em level $m$ congruence subgroup} of $\SL_2(\Z)$
is defined by
$$
\SL_2(\Z)[m] := \big\{A \in \SL_2(\Z) : A \equiv \id \mod m \big\}.
$$

\begin{exercise}
\label{ex:level_tf}
Show that if $m>0$, then $\SL_2(\Z)[m]$ has finite index in $\SL_2(\Z)$ and that
$\SL_2(\Z)[m]$ is torsion free when $m\ge 3$. Show that the torsion subgroup of
$\SL_2(\Z)[2]$ is its center $C_2 = \{\pm I\}$. Use the fact that
$\SL_2(\Z)[m]\bs \h$ is a non-compact Riemann surface to prove that
$\SL_2(\Z)[m]$ is a free group for all $m\ge 3$.
\end{exercise}

\subsection{Quotients by discrete group actions}

The action of a discrete group $\G$ on a topological space $X$ is said to be
{\em properly discontinuous} if each $x \in X$ has a neighbourhood $U$ such that
if $\gamma \in \G$, then $\gamma U \cap U \neq \emptyset$ implies that $\gamma x
= x$. An action of $\G$ on $X$ is {\em free} or {\em fixed point free} if $x\in
X$ and $\gamma \in \G$, then $\gamma x = x$ implies that $\gamma$ is the
identity. If the action of $\G$ on $X$ is both properly discontinuous and free,
then $X \to \G\bs X$ is a covering projection. (Exercise: prove this.)

\begin{exercise}
Prove that the action of $\SL_2(\Z)$ on $\h$ is properly discontinuous and that
$\SL_2(\Z)[m]$ acts fixed point freely (and properly discontinuously) on $\h$
when $m\ge 3$.
\end{exercise}

\begin{exercise}
Suppose that $p: X \to Y$ is a Galois (i.e., normal or regular) covering map
with Galois group (i.e., group of deck transformations) $\G$. Suppose that $X$
is a Riemann surface and that $\G$ acts on $X$ as a group of biholomorphisms.
Show that $Y$ has a unique complex structure such that $p$ is holomorphic.
Deduce that if $X$ is a Riemann surface and $\G$ is a subgroup of $\Aut X$ that
acts properly discontinuously and fixed point freely on $X$, then $\G\bs X$ has
a unique Riemann surface structure such that the covering projection $X \to
\G\bs X$ is holomorphic.
\end{exercise}

Coverings of Riemann surfaces with punctures can be extended across the
punctures. This is a local problem.

\begin{exercise}
Show that all finite coverings of the punctured disk $\D^\ast$ are isomorphic to
$p_n : \D^\ast \to \D^\ast$ where $p_n(z) = z^n$. Deduce that all such coverings
$U \to \D^\ast$ can be completed to a proper holomorphic map $X \to \D$ where $X
$ is a Riemann surface containing $U$ as an open dense subset.
\end{exercise}

\begin{exercise}
\label{ex:completion}
Suppose that $Y$ is a compact Riemann surface and that $F$ is a finite subset of
$Y$. Set $Y'=Y-F$. Show that if $f: X'\to Y'$ is a finite, unramified covering,
there exists a compact Riemann surface $X$, a finite subset $F_X$ of $X$ and a
holomorphic mapping $\tilde{f} : X \to Y$ such that $X'= X-F_X$ and the
restriction of $\tilde{f}$ to $X'$ is $f$.
\end{exercise}

\begin{exercise}
Supposed that $X$ is a Riemann surface and that $P \in X$. Let
$$
\Aut (X,P) = \{\phi \in \Aut X : \phi(P) = P\}.
$$
Show that taking $\phi$ to its derivative at $P$ defines a homomorphism
$\rho:\Aut(X,P) \to \C^\ast$. Denote its kernel by $\Aut_0(X,P)$.
\begin{enumerate}

\item Show that $\Aut_0(X,P)$ is torsion free. (Hint: use power series.)

\item Deduce that if $\G$ is a finite subgroup of $\Aut(X,P)$ of order $d$, then
the restriction of $\rho$ to $\G$ is injective and that $\rho(\G)$ is the group
$\bmu_d$ of $d$th roots of unity.

\item With $\G$ as above, show that there is a holomorphic coordinate $z$ in
$X$, centered at $P$, such that (for $Q$ in a neighbourhood of $P$) the action
of $\G$ is given by $\gamma : z \mapsto \rho(\gamma)z$. More precisely,
$$
z(\gamma(Q)) = \rho(\gamma)z(Q).
$$
Hint: Let $w$ be any holomorphic coordinate centered at $P$ and consider how
$\G$ acts on a $d$th root of $\prod_{\gamma\in \G} \gamma^\ast w$.

\item Show that $\G\bs X$ has a natural Riemann surface structure such that the
projection $X \to \G\bs X$ is holomorphic. (Hint: Localize about each fixed
point of $\G$.)

\end{enumerate}
\end{exercise}

Recall that the action of a group $\G$ on a set $X$ is {\em virtually free} if
$\G$ has a finite index subgroup $\G'$ such that the restriction of the action
to $\G'$  is free.

\begin{exercise}
Show that if the discrete group $\G$ acts properly discontinuously and virtually
freely on a Riemann surface $X$, then $\G\bs X$ has a unique Riemann surface
structure such that the projection $X \to \G\bs X$ is holomorphic. 
\end{exercise}

\begin{exercise}
\label{ex:quotient}
Show that the action of $\SL_2(\Z)$ on $\h$ is virtually free. Deduce that
$\SL_2(\Z)\bs \h$ has a unique Riemann surface structure such that the
projection $\h \to \SL_2(\Z)\bs \h$ is holomorphic.
\end{exercise}

\section{A Very Brief Introduction to Stacks}
\label{app:stacks}

This appendix is a very brief and informal introduction to stacks. The book
\cite{laumon-mb} by Laumon and Moret-Bailly is a standard reference. There are
also the notes \cite{stacks} by Fulton et al. Recall that a {\em
groupoid} is a category in which every morphism is an isomorphism. A starting
observation, explained in Remark~\ref{rem:groupoid}, is that an orbifold may be
viewed as a groupoid in the category of (say) topological spaces.

Suppose that $\cC$ is a category in which fibered products always exist, such as
the category of complex manifolds, the category of varieties over a field, or
the category of schemes over a fixed base. A {\em stack} $\X$ in $\cC$ is a
groupoid in $\cC$. A groupoid in $\cC$ consists of two objects $U$ and $R$ of
$\cC$ together with five morphisms $s,t,e,m,i$ called the source, target,
identity, multiplication, and inverse:
$$
\xymatrix{R \ar@<.6ex>[r]^s\ar@<-.6ex>[r]_t & U},\quad
\xymatrix{U \ar[r]^e & R},\quad
\xymatrix{R_t \times_s R \ar[r]^(.65)m & R},\quad
\xymatrix{R \ar[r]^i & R}.
$$
These satisfy natural axioms which can be worked out from the example in
Remark~\ref{rem:groupoid}, where $U = X$, $R = \G\times X$, and
\begin{equation}
\label{eqn:quotientstack1}
s(\gamma,x) = x,\ t(\gamma,x) = \gamma x,\ e(x) = (\id,x),\
i(\gamma,x) = (\gamma^{-1},\gamma x)
\end{equation}
and
\begin{equation}
\label{eqn:quotientstack2}
m\big((\mu,\gamma x),(\gamma,x)\big) = (\mu\gamma,x).
\end{equation}
The structure $(U,R,s,t,e,m,i)$ is called an {\em atlas} on $\X$. It is the
analogue of an open covering of a topological space. One can define equivalence
classes of atlases and consider a stack to be an equivalence class of atlases,
just as one can consider a manifold to be an equivalence class of atlases.
Roughly speaking, an equivalence of atlases is induced by an equivalence of
categories that induces the identity on isomorphism classes and which has ``good
descent properties". Morphisms of stacks in $\cC$ are induced by functors in
$\cC$ from one atlas to another.

\begin{exercise}
Show that if $m\ge 3$, then the atlases of $\SL_2(\Z)\bbs \h$ and
$\SL_2(\Z/m\Z)\bbs \M_{1,1}[m]$ are equivalent, where $\M_{1,1}[m]$ is the
Riemann surface defined in Section~\ref{sec:level}.
\end{exercise}

Basic orbifolds are stacks in the category of topological spaces (or smooth
manifolds, complex manifolds, etc.) In particular, we may regard $\M_{1,1}$ as a
stack in the category of Riemann surfaces.

The rules (\ref{eqn:quotientstack1}) and (\ref{eqn:quotientstack2}) above can be
used to define the quotient $\G\bbs X$ of an object $X$ of a category $\cC$ by a
group object $\G$ of $\cC$. A typical example is taking the quotient of a
variety (or scheme) $X$ by the action of an algebraic group $\G$.

\subsection{The stack $\Mbar_{1,1}$}

Crudely speaking, $\Mbar_{1,1}$ is the stack in the category of Riemann surfaces
that is obtained by attaching the disk $\D$ to $\M_{1,1}$ along the morphism
$\D^\ast \to \M_{1,1}$ constructed in Exercise~\ref{ex:disk}:
$$
\Mbar_{1,1} = \M_{1,1}\cup_{\D^\ast} \D.
$$
The coordinate in the disk will be denoted by $q$. It is related to the
coordinate $\tau$ of $\h$ by $q = \exp(2\pi i \tau)$.

Set $R = e^{-2\pi}$ and $\h_a = \{\tau:\Im \tau > a\}$. Denote the open $q$-disk
of radius $R$ by $\D_R$. The mapping $q=\exp(2\pi i \tau)$ defines a covering
$\h_1 \to \D^\ast_R$.

The analytic stack $\Mbar_{1,1}$ is defined by the atlas where
$$
U = \h \disjt \D_R
$$
and
$$
R = \Iso(\h,\h) \disjt \Iso(\D_R,\h) \disjt \Iso(\h,\D_R)
\disjt \Iso(\D_R,\D_R)
$$
where these and the source and target maps are defined by
\begin{xalignat*}{2}
&\Iso(\h,\h)= \SL_2(\Z)\times \h, &(s,t) : (\gamma,\tau) &
\mapsto (\tau,\gamma\tau)\cr
&\Iso(\D,\D)= C_2 \times \D_R, &(s,t) : q \mapsto q\cr
&\Iso(\h,\D_R)= C_2 \times \Z \times \h_1,
&(s,t) : (n,\tau)&\mapsto (\tau,q(\tau))  \cr
&\Iso(\D_R,\h)= C_2 \times \Z \times \h_1,
&(s,t) : (n,\tau)&\mapsto  (q(\tau),\tau + n).
\end{xalignat*}
The identity maps
$$
e : \D_R \to \Iso(\D_R,\D_R) = C_2\times \D_R \text{ and }
e : \h \to \Iso(\h,\h) = \SL_2(\Z)\times \h
$$
are $q\mapsto (\id,q)$ and $\tau \mapsto (\id,\tau)$, respectively.

\begin{exercise}
Define the composition mappings
$$
m : \Iso(Y,Z) \times \Iso(X,Y)  \to \Iso(X,Z)
$$
where $X,Y,Z \in \{\h,\D\}$.
\end{exercise}

Two other constructions of $\Mbar_{1,1}$ are sketched in these notes. The
construction given in Section~\ref{sec:level} is as the stack quotient of the
compact Riemann surface (algebraic curve) $\Mbar_{1,1}[m]$, where $m\ge 3$. The
construction given in Section~\ref{sec:mbar_Q} is as a quotient $\C^\ast\bbs
(\C^2-\{0\})$. Each construction has advantages and disadvantages: the
construction above makes clear the connection with modular forms, but is
transcendental; the other two constructions are as quotients of an algebraic
variety by an algebraic group and lie within algebraic geometry; the third works
over any field of characteristic not equal to 2 or 3.

\begin{exercise}
Show that if $m\ge 3$, then $\Mbar_{1,1}$ is isomorphic to the stack
$\SL_2(\Z/m\Z)\bbs \Mbar_{1,1}[m]$.
\end{exercise}

\begin{exercise}
Construct stack morphisms $\M_{1,1} \to \Mbar_{1,1}$ and
$\D \to \Mbar_{1,1}$ such that the diagram
$$
\xymatrix{
\D^\ast \ar[r]\ar[d] & \D \ar[d] \cr
\M_{1,1} \ar[r] & \Mbar_{1,1}
}
$$
where the left hand vertical mapping is the one constructed in
Exercise~\ref{ex:disk}.
\end{exercise}

\subsection{Bundles over stacks}

A vector bundle $\cV$ over a stack $\X$ in $\cC$ with atlas
$$
\xymatrix{R \ar@<.6ex>[r]^s\ar@<-.6ex>[r]_t & U}
$$
consists of a vector bundle $V$ over $U$ in $\cC$ together with isomorphisms
(the ``transition functions'')
\begin{equation}
\label{eqn:trans}
s^\ast V \to t^\ast V
\end{equation}
of vector bundles over $R$ whose pullback along the identity section $e : U \to
R$ is the identity. This can be thought of as a family of linear isomorphisms
$f_\ast : V_{s(f)} \to V_{t(f)}$ between the fibers of $V$ over $s(f)$ and
$t(f)$, indexed by the $f \in R$. It is the identity when $f$ is. The map
(\ref{eqn:trans}) is also required to be compatible with multiplication in the
sense that if $x,y,z \in U$ and $f,g\in R$ such that 
$$
\xymatrix{x \ar[r]^f & y \ar[r]^g & z}
$$
are morphisms, then the diagram
$$
\xymatrix{
V_x \ar[r]^{f_\ast}\ar[dr]_{m(g,f)_\ast} & V_y \ar[d]^{g_\ast} \cr & V_z
}
$$
commutes.

\begin{exercise}
Show that the orbifold line bundle $\Lbar_k \to \Mbar_{1,1}$ is a line bundle
when $\Mbar_{1,1}$ is viewed as a stack.
\end{exercise}

Bundles with other structure groups can be defined similarly.

\end{document}